\documentclass[12]{amsart}
\usepackage{amsaddr,amsthm,amsmath,amssymb,stmaryrd,tikz}
\usepackage[numbers]{natbib}

\newtheorem{thm}{Theorem}[section]
\newtheorem{cor}[thm]{Corollary}
\newtheorem{lem}[thm]{Lemma}
\newtheorem{prop}[thm]{Proposition}
\newtheorem{fact}[thm]{Fact}

\theoremstyle{definition}
\newtheorem{defn}[thm]{Definition}

\newtheorem{rmrk}[thm]{Remark}
\newtheorem{rmrks}[thm]{Remarks}

\newtheorem{conv}[thm]{Convention}

\numberwithin{equation}{section}

\newcommand\s{0.5}
\newcommand\scale{1}
\newcommand\scaleIII{1}
\newcommand\scaleIV{0.75}
\newcommand\scaleV{1}

\newcommand{\R}{\mathbb{R}}

\newcommand{\M}{\mathcal{M}}
\newcommand{\se}{\subseteq}
\newcommand{\sm}{\setminus}
\newcommand{\pw}[1]{\left\{\begin{array}{ll} #1 \end{array}\right.}

\renewcommand{\P}{\mathcal{P}}
\renewcommand{\restriction}{\text{\,}\mathord{\upharpoonright}\text{\,}}
\newcommand{\ess}{\mathrm{ess}}
	
\newcommand\gap{\quad\qquad}
\newcommand{\wid}{0.3mm}

\newcommand{\Rm}[1]{\mathrm{\uppercase\expandafter{\romannumeral #1\relax}}}
\newcommand{\D}{\mathcal{D}}
\newcommand{\sh}{\mathrm{Sh}}
\newcommand{\ish}{\mathrm{isoSh}}

\begin{document}
\title{Graph-like Distributions and Types in Ultrapowers}

\author[1]{Michael Wheeler}

\address[1]{Department of Mathematics, University of Colorado Boulder\\michael.wheeler@colorado.edu}

\begin{abstract}
We study types that appear in ultraproducts that have distributions which can be thought of as a sequence of graphs. The property of having distributions that are captured by graphs is motivated by a commonality of $\mathrm{SOP}_2$-types and types corresponding to pre-cuts in ultrapowers of linear orders. Through this study, we come to a simple set-theoretic description of the distributions for these types and to a new description of good ultrafilters in terms of the difference between ``small'' external complete subgraphs and internal complete subgraphs in ultraproducts of particular families of finite graphs arising from the comparability relation on the complete binary tree or the intersection of intervals in an infinite linear order.

\noindent \emph{Key Words}: distribution, good ultrafilter, graph-like type, ultragraph, ultraproduct

\noindent \emph{MSC2020}: 03C20; 03E05, 05C05, 05C63, 06A05
\end{abstract}

\maketitle

\section{Introduction}

The idea of a distribution for a type in an ultrapower using a regular ultrafilter on an infinite cardinal $\lambda$ was introduced by Keisler in \cite{keisler_order}. Distributions arise naturally in the study of types over ultraproducts created using regular ultrafilters. In particular, if $(\M_\alpha)_{\alpha<\lambda}$ is a sequence of structures in a common language and $p(x)=\{\varphi_\beta(\bar{x};\bar{a}_\beta):\beta<\lambda\}$ is a type of the ultraproduct $\M:=\prod_{\alpha<\lambda}\M_\alpha/\D$, then we can use the regularity of $\D$ to show that $p(\bar{x})$ is realized in $\M$ if and only if there is a way to ``distribute'' the conditions $\varphi_\beta(\bar{x};\bar{a}_\beta)$ to the various index structures so that each index structure is only required to satisfy finitely many of the conditions, each index can simultaneously satisfy the conditions placed on it, and every finite subset of conditions is distributed to ultrafilter-many indices. We study types we call graph-like that have the simplifying properties that $\varphi_\beta=\varphi_\gamma$ for all $\beta,\gamma<\lambda$ and that a finite collection of the constraints is simultaneously satisfiable in a given index if and only if every subcollection of two constraints can be satisfied by a single element of the given index structure.

Graph-like types are so called because the structure of which conditions can be simultaneously realized in any given index structure is captured by a graph, where each vertex represents a condition and an edge between the vertices $v_1$ and $v_2$ indicates that the two conditions represented by $v_1$ and $v_2$ are compatible. In general, we would need to worry about whether arbitrarily large finite subsets of conditions are compatible (which one might wish to represent as edges in a hypergraph), but the property of being graph-like allows us to ignore these higher levels of complexity.

A small (by small, we will mean of cardinality no larger than the cardinality of the index set we have chosen a regular ultrafilter on) graph-like type $p(\bar{x})$, along with a collection of representatives for the parameters that appear in $p(\bar{x})$, is represented by a finite graph $\mathbb{G}_\alpha$ in each index structure $\M_\alpha$, and the type $p(\bar{x})$ can be recovered from the sequence of graphs $(\mathbb{G}_\alpha)_{\alpha<\lambda}$ if the collection of representations is also known (see the proof of Corollary~\ref{multiffint}). One might then hope that the ultraproduct (we call this the ultragraph) $\prod_{\alpha<\lambda}\mathbb{G}_\alpha/\D$ carries some information about $p(\bar{x})$. In this vein we show in Corollary~\ref{multiffint} that $p(\bar{x})$ is realized in the ultraproduct $\prod_{\alpha<\lambda}\M_\alpha/\D$ if and only if the image of a natural (set) embedding of $p(\bar{x})$ in the ultragraph is contained within a complete subgraph of the form $\prod_{\alpha<\lambda}K_\alpha/\D$ where $K_\alpha$ is a complete subgraph of $\mathbb{G}_\alpha$. The image of this embedding being a complete subgraph of the ultragraph is equivalent to the fact that $p(\bar{x})$ is a type and says nothing about whether or not the type $p(\bar{x})$ is realized.

Besides allowing us to reframe the question of whether or not a graph-like type $p(\bar{x})$ is realized in an ultrapower, we also show that graph-like types have distributions that must reflect the class of graphs that can appear in the sequence $(\mathbb{G}_\alpha)_{\alpha<\lambda}$. In cases where there is a discernible pattern in which graphs are allowed to appear in $(\mathbb{G}_\alpha)_{\alpha<\lambda}$, we can exploit this information to gain a characterization of the distributions of $p(\bar{x})$ (see Theorem~\ref{realAalphatype}).

Our primary application will be to the study of distributions in theories with $\mathrm{SOP}_n$ for $n\geq 2$ (defined in \cite{sh500}). In particular, the compatibility of conditions coming from an $\mathrm{SOP}_2$-type (as defined in \cite{mmss_pt}) has the same structure as graphs arising from the compatibility relation on finite trees, and the compatibility of conditions that arise from types indicating a cut within an ultraproduct of infinite linear orders has the structure of graphs where the vertices are intervals in the order (or any infinite linear order) and the edges indicate non-empty intersections, allowing a characterization of the distributions for these sorts of types in terms of the corresponding classes of graphs. Moreover, since we know from \cite{sh500} and \cite{mmss_pt} that an ultrafilter $\D$ is good if and only if $\D$ $\lambda^+$-saturates structures having $\mathrm{SOP}_2$, we find a characterization in Theorem~\ref{sop2goodequivalents} of good ultrafilters both in terms of whether a relatively small class of distributions has multiplicative refinements and in terms of whether ultraproducts of comparability graphs of finite trees allow small complete subgraphs to have internal complete extensions.

We also show that every theory with $\mathrm{SOP}_3$ (or with $\mathrm{SOP}_n$ for $n>3$) admits types that mimic arbitrary distributions of types corresponding to (pre-)cuts in ultrapowers of infinite linear orders. The idea is similar to Shelah's proof in \cite[2.11]{sh500} that theories with $\mathrm{SOP}_3$ are maximal in the Keisler order, but avoids use of the bi-interpretability ordering defined in the same paper.

It is worth pointing out that the characterization of $\lambda^+$-good ultrafilters in terms of (external) complete subgraphs of cardinality $\lambda$ being lifted to internal complete subgraphs is related to the result of Douglas Ulrich \cite[Lem.~2.4]{dulrich_pt} showing that small uses of external recursion can be lifted to an internal use of recursion in $\omega$-nonstandard models $\hat{V}$ of set theory where $\mathfrak{p}_{\hat{V}}>\lambda$. The results presented here show that $\lambda^+$-good ultrafilters are characterized by being able to lift small, external, complete subgraphs of an ultraproduct of specific families of graphs to an internal complete subgraph; whereas Lemma 2.4 of \cite{dulrich_pt} can be thought of as showing that, in any ultraproduct of graphs via a $\lambda^+$-good ultrafilter, it is possible to extend all complete subgraphs of cardinality $\lambda$ to internal complete subgraphs. In particular, if we take $\hat{V}$ to be an ultrapower of $V\vDash\mathrm{ZFC}$ via a good ultrafilter, then we can think of finding an internal complete subgraph of an ultragraph $\prod_{\alpha<\lambda}\mathbb{G}_\alpha/\D\in V$ as finding a complete subgraph of the graph represented by the function $\alpha\mapsto\mathbb{G}_\alpha$ \emph{within} $\hat{V}$. This process could be done by lifting a small external induction to an internal one.

\section{Background Definitions and Conventions}
\begin{conv}
	Let $I$ be an infinite set and $\D$ an ultrafilter on $I$. If for all $i\in I$ the structure $\mathcal{M}_i$ is in the language $\mathcal{L}$, then $\mathcal{M}_I^\D:=\prod_{i\in I}\mathcal{M}_i/\D$ is the ultraproduct of the $\mathcal{M}_i$ with respect to $\D$.
\end{conv}

Typically, we will take the $I$ in the above convention to be an infinite cardinal $\lambda$.

\begin{conv}\label{lambdatype}
	If $\lambda$ is an infinite cardinal and for all $\alpha<\lambda$ the structure $\mathcal{M}_\alpha$ is in the language $\mathcal{L}$ and $\D$ is an ultrafilter on $\lambda$, then $p(\bar{x})=\{\varphi_\beta(\bar{x},\bar{a}_\beta):\beta<\lambda\}$ is a $\lambda$-type of $\mathcal{M}_\alpha^\D$ if whenever $(\bar{r}_\beta)_{\beta<\lambda}$ is a sequence of representatives coming from $\prod_{\alpha<\lambda}\mathcal{M}_\alpha$ for the $\bar{a}_\beta$ in $\mathcal{M}_\alpha^\D$ and whenever $\Delta\se \lambda$ is finite it is the case that 
	\[\llbracket\Delta\rrbracket_{p(\bar{x})}^{\D}:=\left\{\alpha:\mathcal{M}_\alpha\vDash\exists x, \bigwedge_{\beta\in\Delta}\varphi_\beta(\bar{x},\bar{r}_\beta(\alpha))\right\}\in\D.\]
\end{conv}

The convention above agrees with the usual definition of type for a structure by the compactness theorem and \L o\'s theorem.

\begin{conv}
	Whenever we use $\alpha$ we will mean an index for the structures $\M_\alpha$ being considered and whenever we use $\beta$ we will mean an index to be used either for the formula $\varphi_\beta$ or for the representative/parameter $\bar{r}_\beta/\bar{a}_\beta$ for a given $\lambda$-type $p(\bar{x})$. We will variously use $\gamma,\delta,$ and $\epsilon$ as indices when more variables are needed.
\end{conv}

\begin{conv}
We fix an infinite cardinal $\lambda$, a regular ultrafilter $\D$ on $\lambda$, a sequence of structures $(\mathcal{M}_\alpha)_{\alpha<\lambda}$ in the common language $\mathcal{L}$, and a $\lambda$-type $p(\bar{x})$ on $\widehat{\mathcal{M}}=\mathcal{M}_\lambda^\D$. We will continue to use these general fixed values throughout unless otherwise specified.
\end{conv}

We use the following notation that frequently appears in the literature for set theory and Ramsey theory.

\begin{defn}
	Let $A$ be any set and $\kappa$ be any cardinal.
	\begin{enumerate}
		\item $[A]^\kappa$ is the set $\{B\se A:|B|=\kappa\}$.
		\item $\P_\omega(A)$ is the set $\{B\se A:|B|<\omega\}$.
	\end{enumerate}
\end{defn}

\begin{defn}\label{losmap}
	Let $(\bar{r}_\beta)_{\beta<\lambda}$ be given as in the definition of $\lambda$-type. The \emph{\L o\'s map} for $p(\bar{x})$ in $\widehat{\M}$ given $(\bar{r}_\beta)_{\beta<\lambda}$ is the function $L\colon\mathcal{P}_\omega(\lambda)\to\D$ that is defined for each $\Delta\in\P_\omega(\lambda)$ by
	\[L(\Delta)=\llbracket\Delta\rrbracket_{p(x)}^\D.\]
\end{defn}

For any finite $\Delta\se \lambda$ we have that $L(\Delta)\in\D$ by the definition of a $\lambda$-type of $\widehat{\M}$.

We will frequently have the sequence of representations $(\bar{r}_\beta)_{\beta<\lambda}$ for the parameters appearing in a $\lambda$-type $p(\bar{x})$ implicit.

\begin{defn}\label{distribution}
	A \emph{distribution} for a $\lambda$-type $p(\bar{x})$ in $\mathcal{M}^\D$ is a function $f\colon\mathcal{P}_\omega(\lambda)\to\D$ such that 
	\begin{enumerate}
		\item for all $\Delta\in\P_\omega(\lambda)$ we have $f(\Delta)\se L(\Delta)$, and
		\item all $\Gamma\in[\P_\omega(\lambda)]^\omega$ satisfy $\bigcap_{\Delta\in\Gamma}f(\Delta)=\emptyset$.
	\end{enumerate}
	A distribution $f$ is called \emph{monotone} if for all $\Delta,\Phi\in\P_\omega(\lambda)$ we have that $f$ has the property that $\Delta\se\Phi$ implies $f(\Delta)\supseteq f(\Phi)$ and the distribution $f$ is called \emph{multiplicative} if $f$ satisfies $f(\Delta\cup\Phi)=f(\Delta)\cap f(\Phi)$.
\end{defn}

\begin{rmrk}
	A distribution that is multiplicative must also be monotone. Throughout, we will assume that all distributions are monotone.
\end{rmrk}

\begin{fact}[Keisler]\label{keislerfact}
	The $\lambda$-type $p(\bar{x})$ is realized in $\widehat\M$ if and only if there is a multiplicative distribution for $p(\bar{x})$ in $\widehat\M$.
\end{fact}

\begin{defn}
	Suppose that $f$ and $g$ are both distributions for the type $p(\bar{x})$. We say that $f$ \emph{refines} $g$ (or $f$ \emph{is a refinement of} $g$) and write $f\se g$ if, for all $\Delta\in\P_\omega(\lambda)$ we have that $f(\Delta)\se g(\Delta)$.
\end{defn}

\begin{rmrk}
	Condition $(1)$ of \ref{distribution} is frequently stated as ``$f$ refines the \L o\'s map $L$.''
\end{rmrk}

One may view the definition of a distribution $f$ as saying that given a finite subset $\Delta$ of a type in the ultrapower the distribution outputs a set of indices $f(\Delta)$ in $\D$ so that if a representative for a potential realization satisfied $\Delta$ in $\M_\alpha$ for each $\alpha\in f(\Delta)$, then the potential realization at least realizes $\Delta$ in the ultrapower (Note: this statement is also true of the \L o\'s map and is based entirely on part $(1)$ of the definition of a distribution). Moreover, a distribution uses the regularity of the ultrafilter to divide up (distribute) the indices where these finite subsets must be realized so that the representative for the potential realization need only realize finitely many conditions in each index. The finiteness aspect of the distribution comes from part $(2)$ of the definition of a distribution.

It may appear that the existence of a distribution is enough to show that a type is realized (if each index structure only needs to satisfy finitely many of the formula in the type and all finite types are realized, just take a realization of the needed finite type in the given index structure). However, we are not guaranteed that the projections of $p(\bar{x})$ to the various index structures remain finitely satisfiable (i.e.\ such projections may fail to be types in the index structures). The type $p(\bar{x})$ may not be finitely satisfiable in \emph{any} of the structures $\M_\alpha$.

In general, distributions are difficult to work with, so in the next section we propose a different way of looking at the distributions of certain ``nice'' types that appear to be important when working with Keisler's order.

\begin{defn}
	A subset $A\se\widehat\M$ is called \emph{internal} if there is a sequence of unary predicates $(P_\alpha)_{\alpha<\lambda}$ so that if $\M_\alpha$ is $\M$ expanded by the predicate $P_\alpha$ in the language $\mathcal{L}\cup\{P\}$ then $A=P(\M_\alpha^\D)$. Equivalently, $A\se\widehat\M$ is \emph{internal} if for each $\alpha<\lambda$ there are sets $A_\alpha\se\M_\alpha$ such that $A=\prod_{\alpha<\lambda}A_\alpha/\D$.
\end{defn}

One can think of internal subsets as being subsets that arise from the product structure on an ultraproduct. It is not obvious that there will be subsets of an ultraproduct that do not arise from the product structure (the quotient may collapse some subsets that do not come from a product to subsets that do). However, any $\aleph_1$-incomplete ultraproduct of any infinite set will have subsets that are not internal (frequently called external subsets).

One class of types that we frequently reference are the $\mathrm{SOP}_2$-types defined by Malliaris and Shelah in \cite{mmss_pt}. We record an equivalent definition here for ease of reference.

\begin{defn}\label{sop2type}
	We say that a structure $\M$ \emph{has an $\mathrm{SOP}_2$-tree} if there is a formula $\varphi(\bar{x};\bar{y})$ and a sequence of tuples $(\bar{a}_\eta)_{\eta\in 2^{<\omega}}$ from $\M$ all the same length as $\bar{y}$ and for all $\Gamma\se 2^{<\omega}$ the sets
	\[\{\varphi(\bar{x};\bar{a}_\eta):\eta\in\Gamma\}\]
	are consistent if and only if $\Gamma$ is a chain in $2^{<\omega}$. We say that the pair $(\varphi(\bar{x};\bar{y}),(\bar{a}_\eta)_{\eta\in 2^{<\omega}})$ is an \emph{$\mathrm{SOP}_2$-tree in $\M$}.
	
	A first-order complete theory $T$ is said to have $\mathrm{SOP}_2$ if there is some $\M\vDash T$ such that $\M$ has an $\mathrm{SOP}_2$-tree.
	
	A $\lambda$-type 
	\[p(\bar{x})=\{\varphi(\bar{x};\bar{a}_\beta):\beta<\lambda\}\]
	is an \emph{$\mathrm{SOP}_2$-type} if there is a sequence $(\bar{r}_\beta)_{\beta<\lambda}$ of tuples from $\prod_{\alpha<\lambda}\M_\alpha$ such that
	\begin{enumerate}
		\item for each $\beta<\lambda$ we have that $\bar{r}_\beta/\D=\bar{a}_\beta$, and
		\item for each $\alpha<\lambda$ there is an $\mathrm{SOP}_2$-tree $(\varphi(\bar{x};\bar{y}),(\bar{b}_\eta)_{\eta\in 2^{<\omega}})$ in $\M_\alpha$ with
		\[\{\bar{r}_\beta(\alpha):\beta<\lambda\}\se\{\bar{b}_\eta:\eta\in 2^{<\omega}\}.\]
	\end{enumerate}

\end{defn}

\section{Graph-like Distributions}

\subsection{Graph-like Distributions and Conjugates}
We begin by identifying a simplifying property that distributions may have. We will later see that $\mathrm{SOP}_2$-types and types corresponding to pre-cuts in ultrapowers of linear orders both have this nice simplifying property.

\begin{defn}
	A distribution $f$ for $p(\bar{x})$ in $\widehat\M$ is called \emph{graph-like} if it satisfies the condition
	\[f(\Delta)=\bigcap_{\Phi\in[\Delta]^2}f(\Phi)\]
	whenever $\Delta\in\P_\omega(\lambda)$ has $|\Delta|\geq 2$.
\end{defn}

A graph-like distribution is a distribution in which the combinatorial information contained in the distribution can be fully captured by a (simple) graph. Put another way, a graph-like distribution $f$ is a distribution where a finite $\Delta\se p(\bar{x})$ is guaranteed to be realized in $\M_\alpha$ by $f$ if and only if each of the two-element subsets of $\Delta$ are guaranteed to be realized in $\M_\alpha$ by $f$. Note: If $\Delta$ is guaranteed to be realized by $f$ then all of its two element subsets are also guaranteed to be realized by the monotonicity of $f$. Thus the content of being a graph-like distribution is that $f$ is both monotone and that finite sets of size strictly grater than $2$ have their simultaneous realizability determined by the two-element subsets of $p(\bar{x})$.

A distribution being graph-like is also an approximation of the distribution being multiplicative.

\begin{lem}\label{multimpgl}
	If $f$ is a distribution such that
		 for all $\Delta,\Gamma\in\mathcal{P}_\omega(\lambda)$ with $|\Delta|,|\Gamma|\geq 2$ we have that $f(\Delta\cup\Gamma)=f(\Delta)\cap f(\Gamma)$ then $f$ is a graph-like distribution.
\end{lem}
\begin{proof}
	By induction, the property given above can be extended to any finite number of subsets of cardinality at least $2$. In particular, if $\Delta\in\mathcal{P}_\omega(\lambda)$ with $|\Delta|\geq 2$, then $[\Delta]^2$ is a finite collection of two-element subsets of $\lambda$. We thus get that
	\[f(\Delta)=f\left(\bigcup\, [\Delta]^2\right)=\bigcap_{\Phi\in[\Delta]^2} f(\Phi).\qedhere\]
\end{proof}

\begin{lem}\label{partialmultimpglmult}
	If $f$ is a graph-like distribution  then $f$ is multiplicative if and only if $f$ satisfies the property
	\begin{equation}\label{partialmult}
		f(\{\beta,\gamma\})=f(\{\beta\})\cap f(\{\gamma\})
	\end{equation}
	for all $\beta,\gamma\in\lambda$.
\end{lem}
\begin{proof}
	We will first show that the property \eqref{partialmult} can be extended to all finite subsets of $\lambda$ using the fact that $f$ is graph-like, and then derive multiplicativity from this expanded property.
	Suppose that $\Delta\in\mathcal{P}_\omega(\lambda)$ with $|\Delta|>2$. Since $f$ is graph-like, we get that
	\[f(\Delta)=\bigcap_{\Phi\in[\Delta]^2}f(\Phi).\] 
	Because each $\Phi\in[\Delta]^2$ has exactly two elements, we may use the property \eqref{partialmult} to further expand this equation to
	\begin{equation}\label{expandedpartialmult}
	    f(\Delta)=\bigcap_{\Phi\in[\Delta]^2}\bigcap_{\beta\in\Phi}f(\{\beta\})=\bigcap_{\beta\in\Delta}f(\{\beta\}).
	\end{equation}
	Now suppose that $\Gamma,\Theta\in\P_\omega(\lambda)$. Using Equation~\eqref{expandedpartialmult} we derive that
	\begin{align*}
		f(\Gamma\cup\Theta)&=\bigcap_{\beta\in(\Gamma\cup\Theta)}f(\{\beta\})\\
		&=\left(\bigcap_{\beta\in\Gamma}f(\{\beta\})\right)\cap\left(\bigcap_{\gamma\in\Theta}f(\{\gamma\})\right)\\
		&=f(\Gamma)\cap f(\Theta).
	\end{align*}
	
	The other direction is a direct consequence of the definition of a distribution being multiplicative.
\end{proof}

\begin{cor}\label{conjmult}
	If $f$ is a distribution then $f$ is multiplicative if and only if $f$ is graph-like and $f$ satisfies 
	\[
		f(\{\delta,\epsilon\})=f(\{\delta\})\cap f(\{\epsilon\})
	\]
	for all $\delta,\epsilon\in\lambda$.
\end{cor}
\begin{proof}
	($\Rightarrow$) A multiplicative distribution $f$ satisfies the assumptions of Lemma~\ref{multimpgl} and the property~\eqref{partialmult}.
	
	($\Leftarrow$) This is Lemma~\ref{partialmultimpglmult}.
\end{proof}

Because the information of a graph-like distribution $f$ is captured by a graph, we wish to explicitly say what these graphs are. In order to do this, it will be convenient to change the question (asked by a distribution) from ``Given a finite subset $\Delta$ of the type $p(\bar{x})$, in what set of indices do we know that $\M_\alpha\vDash\Delta$ by looking at $f$?'' to the question (asked by what we call the conjugate distribution) ``Given an index structure $\M_\alpha$, what are the finite subsets $\Delta$ of $p(\bar{x})$ so that we are guaranteed by $f$ that $\M_\alpha\vDash \Delta$?''

\begin{defn}\label{conjugate}
	Let $f$ be a distribution for $p(\bar{x})$ in $\widehat\M$. The \emph{conjugate of $f$} is the $\omega$-sequence of functions $g_n\colon\lambda\to\P_\omega([\lambda]^n)$ defined for each $\alpha<\lambda$ by 
	\begin{equation}\label{conjeqn}
	    g_n(\alpha)=\{\Delta\in[\lambda]^n:\alpha\in f(\Delta)\}.
	\end{equation}
\end{defn}

\begin{rmrks}\hspace{0em}
    \begin{enumerate}
        \item One could also reasonably define the conjugate as the single function $\hat{g}\colon\displaystyle\alpha\mapsto\bigcup_{n\in\omega} g_n(\alpha)$ for all $\alpha<\lambda$, although this would be more cumbersome for our purposes. The function $\hat{g}$ has crucially appeared in several proofs before but has not been thoroughly studied.
        \item The fact that $g_n(\alpha)$ is finite for each $n\in\omega$ and each $\alpha<\lambda$ is a consequence of the definition of a distribution (in particular, Definition~\ref{distribution}(2)), and the existence of functions satisfying this definition is equivalent to the ultrafilter $\D$ being regular. In short, we are able to make the $g_n(\alpha)$ finite because we are working with a regular ultrafilter.
        \item For elements of $g_1(\alpha)$ we will not be careful about the distinction between elements of $\lambda$ and elements of $[\lambda]^1$. In particular, we will usually think of $g_1(\alpha)$ as a finite subset of $\lambda$.
    \end{enumerate}
\end{rmrks}

The conjugate $(g_n)_{n\in\omega}$ of the distribution $f$ contains all of the information that is contained within the distribution $f$.

\begin{lem}\label{distviaconj}
	If $f$ is a distribution with conjugate $(g_n)_{n\in\omega}$ then
	\begin{equation}\label{disteqconj}
	f(\Delta)=\{\alpha\in\lambda:\Delta\in g_{|\Delta|}(\alpha)\}.
	\end{equation}
\end{lem}
\begin{proof}
	The result follows from the string of equalities
	\begin{align*}
	    \{\alpha\in\lambda:\Delta\in g_{|\Delta|}(\alpha)\}&=\big\{\alpha\in\lambda:\Delta\in \{\Gamma\in[\lambda]^{|\Delta|}:\alpha\in f(\Gamma)\}\big\}\\
	    &=\{\alpha\in\lambda:\alpha\in f(\Delta)\}\\
	    &=f(\Delta).\qedhere
	\end{align*}
\end{proof}

\begin{rmrk}
    The previous lemma shows that distributions $f$ are uniquely associated with their conjugate $(g_n)_{n\in\omega}$ as $f$ can be recovered from the conjugate using Equation~\eqref{disteqconj}.
\end{rmrk}

In the case of a graph-like distribution, the behavior of the two-element subsets determines the behavior of the larger subsets, so we may expect that a graph-like distribution and its conjugate are uniquely determined by just $g_1$ and $g_2$, as the following theorem demonstrates.

\begin{thm}\label{uniqglconj}
	If $f$ is a graph-like distribution with conjugate $(g_n)_{n\in\omega}$ then for $n\geq3$ we have that
	\begin{equation}\label{glconjeq}
	    g_n(\alpha)=\{\Delta\in[\lambda]^n:[\Delta]^2\se g_2(\alpha)\}.
	\end{equation}
	In particular, the conjugate $(g_n)_{n\in\omega}$ of a graph-like distribution is determined by $g_1$ and $g_2$.
\end{thm}
\begin{proof}
    The result follows from the following chain of equalities:
    \begin{align*}
	    g_n(\alpha)&=\{\Delta\in[\lambda]^n:\alpha\in f(\Delta)\}\\
	    &=\{\Delta\in[\lambda]^n:\alpha\in \bigcap_{\Gamma\in[\Delta]^2}f(\Gamma)\}\\
	    &=\{\Delta\in[\lambda]^n:\forall \Gamma\in[\Delta]^2,\, \alpha\in f(\Gamma)\}\\
	    &=\{\Delta\in[\lambda]^n:\forall \Gamma\in[\Delta]^2,\, \Gamma\in g_2(\alpha)\}\\
	    &=\{\Delta\in[\lambda]^n:[\Delta]^2\se g_2(\alpha)\}.\qedhere
	\end{align*}
\end{proof}

From the previous two theorems, we see that the conjugate of a distribution $f$ carries with it all the information that is contained in $f$. Fortunately, we can also translate the axioms of a distribution into axioms for a sequence of functions $(g_n)_{n\in\omega}$ of the type ${g_n\colon\lambda\to\P([\lambda]^n)}$ being the conjugate of a distribution or a graph-like distribution. In fact, it will be easier for us to work with the conjugates of graph-like distributions than with graph-like distributions themselves.

\begin{lem}\label{distconjlem}
	Suppose that $f\colon\P_\omega(\lambda)\to\D$ is an arbitrary function and $p(\bar{x})$ is a $\lambda$-type of the ultraproduct $\widehat\M:=\prod_{\alpha<\lambda}\M_\alpha/\D$ where $\D$ is a regular ultrafilter on the infinite cardinal $\lambda$. Suppose that the conjugate of $f$ is defined in the same manner as in Definition~\ref{conjugate} and that the conjugate of $f$ is $(g_n)_{n\in\omega}$ (the codomain of $g_n$ is now $\P([\lambda]^n)$ instead of $\P_\omega([\lambda]^n)$). Then
	\begin{enumerate}
		\item The function $f$ refines the \L o\'s map $L$ of $p(\bar{x})$ if and only if for all $n\in\omega$, all $\alpha\in\lambda$, and all $\Delta\in g_n(\alpha)$ we have that 
		\[\M_\alpha\vDash\exists \bar{x},\bigwedge_{\beta\in \Delta}\varphi_\beta(\bar{x},\bar{r}_\beta(\alpha))\] 
		where the $\bar{r}_\beta$ are the fixed representatives in $\prod_{\alpha<\lambda}\M_\alpha$ of the parameters in $p(\bar{x})$.
		\item The function $f$ satisfies $\bigcap_{\Delta\in\Gamma}f(\Delta)=\emptyset$ for all $\Gamma\in[\P_\omega(\lambda)]^\omega$ if and only if for all $\alpha\in\lambda$ the set $\bigcup_{n\in\omega}g_n(\alpha)$ is finite.
		\item The function $f$ is monotone if and only if for all $n>m$ in $\omega$, all $\alpha\in\lambda$, and all $\Delta\in g_n(\alpha)$ we have that $[\Delta]^m\se g_m(\alpha)$.
	\end{enumerate}
\end{lem}
\begin{proof}
	$(1)$: ($\Rightarrow$) Fix an $\alpha\in\lambda$ and $n\in\omega$ and suppose that $\Delta\in g_n(\alpha)$. We then have that $\alpha\in f(\Delta)\se\llbracket \Delta\rrbracket_{p(\bar{x})}^\D$. That is,
	\[\M_\alpha\vDash\exists\bar{x},\bigwedge_{\beta\in\Delta}\varphi_\beta(\bar{x},\bar{r}_\beta(\alpha)).\]
	
	($\Leftarrow$) Suppose that $\Delta\in[\lambda]^n$ and $\alpha\in f(\Delta)$. This setup guarantees that $\Delta\in g_n(\alpha)$ so 
	\[\M_\alpha\vDash\exists\bar{x},\bigwedge_{\beta\in\Delta}\varphi_\beta(\bar{x},\bar{r}_\beta(\alpha)).\]
	That is, $\alpha\in\llbracket\Delta\rrbracket_{p(\bar{x})}^\D=L(\Delta)$.
	
	$(2)$: Both directions will be easier using the claim that $\Gamma$ is a subset of $\bigcup_{n\in\omega}g_n(\alpha)$ of any cardinality if and only if $\alpha$ is an element of $\bigcap_{\Delta\in\Gamma}f(\Delta)$. The claim follows from the equivalences
	\begin{align*}
	\Gamma\se\bigcup_{n\in\omega}	g_n(\alpha)&\iff \forall \Delta\in\Gamma,\, \alpha\in f(\Delta)\\
	&\iff \alpha\in\bigcap_{\Delta\in\Gamma} f(\Delta)
	\end{align*}
	where the first equivalence comes from the definition of the conjugate and the fact that $\P_\omega(\lambda)=\bigcup_{n\in\omega}[\lambda]^n$.

	($\Leftarrow$) Proceeding using the contrapositive, pick $\Gamma\in[\P_\omega(\lambda)]^\omega$ so that ${A:=\bigcap_{\Delta\in \Gamma}f(\Delta)}$ is non-empty. Choose $\alpha\in A$. By the above claim, $\Gamma\se\bigcup_{n\in\omega}g_n(\alpha)$ so \[\left|\bigcup_{n\in\omega}g_n(\alpha)\right|\geq|\Gamma|=\aleph_0.\]
	
	($\Rightarrow$) Again proceeding using the contrapositive, pick an $\alpha\in\lambda$ so that $\Gamma'=\bigcup_n g_n(\alpha)$ is infinite and define $\Gamma$ to be a countable infinite subset of $\Gamma'$. By the above claim, we have that $\alpha\in\bigcap_{\Delta\in\Gamma}f(\Delta)$ where $\Gamma\in[\P_\omega(\lambda)]^\omega$.
	
	$(3)$: ($\Rightarrow$) Suppose that $\alpha\in f(\Delta)$ and $|\Delta|=n$ (i.e.\ $\Delta\in g_n(\alpha)$). If $\Gamma\in[\Delta]^m$ with $m<n$ then $f(\Gamma)\supseteq f(\Delta)$ by monotonicity, so $\alpha\in f(\Gamma)$. That is, $\Gamma\in g_m(\alpha)$.
	
	($\Leftarrow$) Suppose that $\Gamma\se\Delta$ where $|\Gamma|=m$ and $|\Delta|=n$. Then $\alpha\in f(\Delta)$ if and only if $\Delta\in g_n(\alpha)$. As $\Gamma\in[\Delta]^m$, we get that $\Gamma\in g_m(\alpha)$ which is equivalent to saying $\alpha\in f(\Gamma)$. Thus $f(\Delta)\se f(\Gamma)$.
\end{proof}

We will mostly use the preceding lemma in order to show that a given sequence of functions $(g_n)_{n\in\omega}$ where $g_n\colon \lambda\to\P_\omega([\lambda]^n)$ is the conjugate sequence of some distribution $f$. Generally, we will use this strategy to create a graph-like distribution out of $g_1,g_2$, so we present the theory both in the general case and in the specialized case that we are most concerned with.

\begin{lem}\label{distconj}
	If $(g_n)_{n\in\omega}$ is a sequence of functions with $g_n\colon \lambda\to\P_\omega([\lambda]^n)$ then $(g_n)_{n\in\omega}$ is the conjugate of a monotone distribution $f$ for the $\lambda$-type $p(\bar{x})$ in $\widehat{M}=\prod_{\alpha<\lambda}\M_\alpha/\D$ with the representatives $(\bar{r}_\beta)_{\beta\in\lambda}$ if and only if the sequence $(g_n)_{n\in\omega}$ satisfies the following conditions:
	\begin{enumerate}
		\item For all $n\in\omega$, all $\alpha\in\lambda$, and all $\Delta\in g_n(\alpha)$
			\[\M_\alpha\vDash\exists\bar{x},\bigwedge_{\beta\in\Delta}\varphi_\beta(\bar{x},\bar{r}_\beta(\alpha)).\]
		\item For all $n>m$ in $\omega$, all $\alpha\in\lambda$, and all $\Delta\in g_n(\alpha)$ it is the case that $[\Delta]^m\se g_m(\alpha)$.
		\item For all $n\in\omega$ and all $\Delta\in[\lambda]^n$ it is the case that $\{\alpha\in\lambda:\Delta\in g_n(\alpha)\}$ is an element of $\D$.
	\end{enumerate}
\end{lem}
\begin{proof}
	($\Leftarrow$): In order to use Lemma \ref{distconjlem}, we need to show that $(g_n)_{n\in\omega}$ is the conjugate of a function $f\colon \P_\omega(\lambda)\to\D$. We claim that the function $\Delta\mapsto \{\alpha\in\lambda:\Delta\in g_{|\Delta|}(\alpha)\}$ for all $\Delta\in\P_\omega(\lambda)$ is the desired $f$. That the image of $f$ is a subset of $\D$ follows from condition $(3)$ above. If $(h_n)_{n\in\omega}$ is the conjugate of $f$ then 
	\begin{align*}
		h_n(\alpha)&=\{\Delta\in[\lambda]^n:\alpha\in f(\Delta)\}\\
		&=\{\Delta\in[\lambda]^n:\alpha\in\{\gamma\in\lambda:\Delta\in g_n(\gamma)\}\}\\
		&=\{\Delta\in[\lambda]^n:\Delta\in g_n(\alpha)\}\\
		&=g_n(\alpha)
	\end{align*}
	so $(g_n)_{n\in\omega}$ is the conjugate of $f$.
	
	We now immediately get from Lemma \ref{distconjlem} that $f$ refines the \L o\'s map using condition $(1)$ of this lemma and that $f$ is monotone using condition $(2)$ for this lemma. It remains to check that $f$ satisfies the condition in Lemma \ref{distconjlem}$(2)$.
	
	Fix $n\in\omega$. We claim that $g_n(\alpha)\se [g_1(\alpha)]^n$. This claim will be enough because $g_1(\alpha)$ is finite so the sequence $([g_1(\alpha)]^m)_{m\in\omega}$ is a sequence of finite sets with the property that $[g_1(\alpha)]^k=\emptyset$ whenever $k>|g_1(\alpha)|$. By using the claim, we get that 
	\[\bigcup_{m\in\omega}g_m(\alpha)\se\bigcup_{m\in\omega}[g_1(\alpha)]^m=\bigcup_{m=0}^{|g_1(\alpha)|}[g_1(\alpha)]^m=\P(g_1(\alpha))\] 
	where the right-hand side is a finite union of finite sets, so $\bigcup_{m\in\omega}g_m(\alpha)$ must also be finite. By Lemma \ref{distconjlem}$(2)$, the sequence $(f(\Delta))_{\Delta\in\P_\omega(\lambda)}$ is regularizing in $\D$.
	
	To prove the claim: Suppose that $\Delta\in g_n(\alpha)$. Using condition $(2)$ with $m=1$ we get that $[\Delta]^1\se g_1(\alpha)$. Under the identification of $[\Delta]^1$ with $\Delta$, this says that $\Delta\se g_1(\alpha)$ which is equivalent to $\Delta\in[g_1(\alpha)]^n$ as $|\Delta|=n$ by assumption.
	
	($\Rightarrow$): From Lemma~\ref{distconjlem}(1) and (3) we immediately have that $(g_n)_{n\in\omega}$ satisfies conditions (1) and (2) above. The condition (3) of this lemma is satisfied because $\{\alpha\in\lambda:\Delta\in g_n(\alpha)\}=f(\Delta)$ (Equation~\eqref{disteqconj}) and because the codomain of $f$ is $\D$.
\end{proof}

\begin{lem}\label{gldistconj}
	If $g_1,g_2$ are functions $g_i\colon \lambda\to\P_\omega([\lambda]^i)$ then $g_1,g_2$ determine the conjugate of a monotone graph-like distribution $f$ for the $\lambda$-type $p(\bar{x})$ in $\widehat{M}=\prod_{\alpha<\lambda}\M_\alpha/\D$ with the representatives $(\bar{r}_\beta)_{\beta\in\lambda}$ if and only if the functions $g_1,g_2$ satisfy the following conditions:
	\begin{enumerate}
		\item $\beta\in g_1(\alpha)$ implies that $\M_\alpha\vDash\exists \bar{x},\,\varphi_\beta(\bar{x},\bar{r}_\beta(\alpha))$.
		\item For all $n\geq 2$ in $\omega$ and all $\Delta$ with $|\Delta|=n$ and $[\Delta]^2\se g_2(\alpha)$ it is the case that 
		\[\M_\alpha\vDash\exists\bar{x},\bigwedge_{\beta\in\Delta}\varphi_\beta(\bar{x},\bar{r}_\beta(\alpha)).\]
		\item For all $\alpha\in\lambda$ it is the case that $g_2(\alpha)\se [g_1(\alpha)]^2$.
		\item For each $\Delta\in\P_\omega(\lambda)$ with $|\Delta|\geq 2$ the set $\{\alpha\in\lambda:[\Delta]^2\se g_2(\alpha)\}$ is an element of $\D$. Equivalently, for each $\Delta\in[\lambda]^2$ the set $\{\alpha\in\lambda:\Delta\in g_2(\alpha)\}$ is in $\D$.
	\end{enumerate}
\end{lem}
\begin{proof}
	($\Leftarrow$): As in the proof of Theorem \ref{uniqglconj}, we know that if there is a graph-like $f$ with $g_1$ and $g_2$ part of its conjugate then the entire conjugate is $(g_n)_{n\in\omega}$ where $g_n(\alpha)=\{\Delta\in[\lambda]^n:[\Delta]^2\se g_2(\alpha)\}$ for $n>2$ and $g_0$ is the constant function with output $\{\emptyset\}$. Thus it will be enough to show that $(g_n)_{n\in\omega}$ so defined is the conjugate of a graph-like distribution $f$.
	
	We will proceed by showing that $(g_n)_{n\in\omega}$ satisfies the conditions in Lemma \ref{distconj} and that the distribution $f$ for which $(g_n)_{n\in\omega}$ is the conjugate must be graph-like. The condition of Lemma \ref{distconj}(1) is satisfied with $n=1$ by condition $(1)$ and for $n\geq 2$ by condition $(2)$ (because $\Delta$ is an element of $g_n(\alpha)$ if and only if $|\Delta|=n$ and $[\Delta]^2\se g_2(\alpha)$).
	
	The condition of Lemma \ref{distconj}(2) is satisfied for $n=2,m=1$ by condition $(3)$ above. The cases with $m=0$ are trivial, so it remains to check the cases with $n>2$, $m>0$. When $m=2$, this follows by the definition of the $g_n$. For $m=1$, this follows from the fact that one element subsets of $\Delta$ are one element subsets of two element subsets of $\Delta$ and using the cases $m=2$ and $n=2$, $m=1$. For $m>2$, we note that $\Delta\in g_n(\alpha)$ is equivalent to $[\Delta]^2\se g_2(\alpha)$ and if $\Gamma\in[\Delta]^m$ then $[\Gamma]^2\se g_2(\alpha)$ which is equivalent to $\Gamma\in g_m(\alpha)$.
	
	The condition of Lemma \ref{distconj}(3) is satisfied for $n\geq 2$ by condition $(4)$ and the definition of $g_n$. Because we know that $(g_n)_{n\in\omega}$ satisfies the condition $(3)$, we have that if $\beta,\gamma$ are distinct elements of $\lambda$ and $\{\beta,\gamma\}\in g_2(\alpha)$ then $\beta\in g_1(\alpha)$. That is, 
	\[\{\alpha\in\lambda:\beta\in g_1(\alpha)\}\supseteq\{\alpha\in\lambda:\{\beta,\gamma\}\in g_2(\alpha)\}\in\D\] 
	so $\{\alpha:\beta\in g_1(\alpha)\}\in\D$ because $\D$ is a filter.
	
	($\Rightarrow$): Conditions (1) and (2) of this lemma are both implied by Lemma~\ref{distconj}(1). Condition (3) of this lemma is a direct consequence of Lemma~\ref{distconj}(2).
	Condition 4 of this lemma follows from Lemma~\ref{distconj}(3) and Equation~\eqref{glconjeq}.
\end{proof}

\section{Graphs Arising from Graph-like Distributions}

We now assume that $p(\bar{x})$ has a graph-like distribution given by $f:\P_\omega(\lambda)\to \D$ and that the conjugate of $f$ is $(g_n)_{n\in\omega}$.

As mentioned before, the information carried by a graph-like distribution can be represented in each index by a finite simple graph. This idea is made formal in the following definition.

\begin{defn}
	The \emph{distribution graph sequence of $f$} is the $\lambda$-sequence of graphs $(\mathbb{G}_\alpha)_{\alpha\in\lambda}$ such that the vertex set of $\mathbb{G}_\alpha$ is $G_\alpha:=g_1(\alpha)$ and the edge relation in $\mathbb{G}_\alpha$ is $E_\alpha:=g_2(\alpha)$.
\end{defn}

\begin{rmrks}\label{distgraphseqrmrks}\hspace{1 em}
\begin{enumerate}
	\item The distribution $f$ and conjugate sequence can be easily recovered from the graph sequence as $g_1(\alpha)$ is the set of vertices of $\mathbb{G}_\alpha$ and $g_2(\alpha)$ is the set of edges in $\mathbb{G}_\alpha$. Using Theorem \ref{uniqglconj}, this information is enough to determine $f$ as a graph-like distribution.
	\item If $\Delta\in\mathcal{P}_\omega(\lambda)$ then $\alpha\in f(\Delta)$ if and only if $\Delta\se g_1(\alpha)$ and $\Delta$ induces a complete subgraph of $\mathbb{G}_\alpha$. One consequence of this is that we know that the set of formulas $\{\varphi_\beta(x,\bar{r}_\beta(\alpha)):\beta\in\Delta\}$ is satisfiable in $\M_\alpha$ if (but usually not only if) $\Delta$ induces a complete subgraph of $\mathbb{G}_\alpha$.
\end{enumerate}
\end{rmrks}

It will be useful to record some translations of concepts coming from distributions in terms of the graph sequence. As the distribution graph sequence contains all of the information of a graph-like distribution, one of the translations we will get is a version of Lemma \ref{gldistconj} in terms of the distribution graph sequence instead of the conjugate. We will also see how to characterize refinements and multiplicative distributions in terms of the distribution graph sequence.

\begin{lem}\label{glseqdist}
	A sequence of finite graphs $(\mathbb{G}_\alpha)_{\alpha<\lambda}$ where the set of vertices of each $\mathbb{G}_\alpha=(G_\alpha,E_\alpha)$ is a subset of $\lambda$ is a distribution graph sequence of a monotone graph-like distribution $f$ for the $\lambda$-type $p(\bar{x})$ in $\widehat{M}=\prod_{\alpha<\lambda}\M_\alpha/\D$ with the representatives $(\bar{r}_\beta)_{\beta\in\lambda}$ if and only if the sequence $(\mathbb{G}_\alpha)_{\alpha<\lambda}$ satisfies the following conditions:
	\begin{enumerate}
		\item If $\Delta\se G_\alpha$ induces a complete subgraph of $\mathbb{G}_\alpha$ then
		\[\M_\alpha\vDash\exists\bar{x},\bigwedge_{\beta\in\Delta}\varphi_\beta(\bar{x},\bar{r}_\beta(\alpha)).\]
		\item For all $\Delta\in[\lambda]^2$ the set $\{\alpha\in\lambda:\Delta\in E_\alpha\}$ is an element of $\D$.
	\end{enumerate}
\end{lem}
\begin{proof}
	Set $g_1(\alpha)=G_\alpha$ and $g_2(\alpha)=E_\alpha$. We claim that $g_1$ and $g_2$ satisfy the conditions of Lemma \ref{gldistconj} to determine a graph-like distribution for $p(\bar{x})$. The conditions of Lemma \ref{gldistconj}$(1)$--$(2)$ are satisfied because of condition $(1)$ above. That is, single vertices induce complete subgraphs and any $\Delta\se g_1(\alpha)$ induces a complete subgraph of $\mathbb{G}_\alpha$ with $|\Delta|\geq 2$ if and only if every pair of elements of $\Delta$ has an edge between them in $\mathbb{G}_\alpha$ if and only if $[\Delta]^2\se g_2(\alpha)$.
	The condition Lemma \ref{gldistconj}(3) is satisfied as the edge set of a graph must be a subset of the two element subsets of the set of vertices. The condition Lemma \ref{gldistconj}(4) is implied by the condition $(2)$ above using the fact that $E_\alpha=g_2(\alpha)$.
\end{proof}

\begin{lem}\label{refviaseq}
	Suppose that $(\mathbb{G}_\alpha)_{\alpha<\lambda}$ is the distribution graph sequence of a monotone graph-like distribution $f$ for the type $p(\bar{x})$. Then a sequence of finite graphs $(\mathbb{H}_\alpha)_{\alpha<\lambda}$ is the graph distribution sequence of a graph-like refinement of $f$ if and only if $\mathbb{H}_\alpha$ is a subgraph of $\mathbb{G}_\alpha$ for all $\alpha<\lambda$ and for all $\Delta\in[\lambda]^2$ the set $\{\alpha\in\lambda:\Delta\in E^{\mathbb{H}_\alpha}\}$ is an element of $\D$.
\end{lem}
\begin{proof}
	($\Leftarrow$) The condition in Lemma \ref{glseqdist}(1) is preserved by subgraphs and the condition in Lemma \ref{glseqdist}(2) is assumed to be true of $(\mathbb{H}_\alpha)_{\alpha<\lambda}$, so $(\mathbb{H}_\alpha)_{\alpha<\lambda}$ is the distribution graph sequence of some distribution $\ell$. We wish to check that $\ell$ is a refinement of $f$. Let $(h_n)_{n\in\omega}$ be the conjugate of $\ell$ and $\Delta$ a fixed element of $\P_\omega(\lambda)$. Suppose that $\epsilon\in \ell(\Delta)$ or, equivalently, $\Delta\in h_{|\Delta|}(\epsilon)$. If $|\Delta|\geq 2$, then $[\Delta]^2\se h_2(\epsilon)$ by Lemma \ref{distconjlem}(2) and $h_2(\epsilon)\se g_2(\epsilon)$ because the set of edges in $\mathbb{H}_\epsilon$ is a subset of the set of edges in $\mathbb{G}_\epsilon$. Now using Lemma \ref{distconjlem}(2) on $(g_n)_{n\in\omega}$ yields $\Delta\in g_{|\Delta|}(\epsilon)$, which is equivalent to $\epsilon\in f(\Delta)$, so $\ell(\Delta)\se f(\Delta)$. The case where $|\Delta|=1$ follows from the set of vertices of $\mathbb{H}_\epsilon$ being a subset of the set of vertices of $\mathbb{G}_\epsilon$. 
	
	($\Rightarrow$) We know that if $(\mathbb{H}_\alpha)_{\alpha<\lambda}$ is the distribution graph sequence of  graph-like distribution then $(\mathbb{H}_\alpha)_{\alpha<\lambda}$ must satisfy Lemma \ref{glseqdist}(2), so for any $\Delta\in\P_\omega(\lambda)$ we have that $\{\alpha\in\lambda:\Delta\in E^{\mathbb{H}_\alpha}\}$ is an element of $\D$. Let $\ell\colon\P_\omega(\lambda)\to\D$ be the distribution whose distribution graph-sequence is $(\mathbb{H}_\alpha)_{\alpha<\lambda}$. Then for all $\Delta\in\P_\omega(\lambda)$ we have that $\ell(\Delta)\se f(\Delta)$, so $\Delta\in h_{|\Delta|}(\alpha)$ implies $\Delta\in g_{|\Delta|}(\alpha)$. Thus $h_n(\alpha)\se g_n(\alpha)$ for all $\alpha\in\lambda$ and $n\in\omega$, which, when applied with $n=1,2$, implies that $\mathbb{H}_\alpha$ is a subgraph of $\mathbb{G}_\alpha$.
\end{proof}

\begin{lem}\label{multviaseq}
	Suppose that $(\mathbb{G}_\alpha)_{\alpha<\lambda}$ is the distribution graph sequence of a graph-like distribution $f$ for the type $p(\bar{x})$. Then $f$ is multiplicative if and only if $\mathbb{G}_\alpha$ is complete for every $\alpha\in\lambda$.
\end{lem}
\begin{proof}
	From Corollary \ref{conjmult} we know that $f$ is multiplicative if and only if $f(\{\beta,\gamma\})=f(\{\beta\})\cap f(\{\gamma\})$ for all distinct $\beta,\gamma\in\lambda$ as we have already assumed that $f$ is graph-like. That is, $\{\beta,\gamma\}\in g_2(\alpha)$ if and only if $\beta,\gamma\in g_1(\alpha)$ where $(g_n)_{n\in\omega}$ is the conjugate of $f$. This is equivalent to asking that $\mathbb{G}_\alpha$ is a complete graph.
\end{proof}

Because we get a finite graph $\mathbb{G}_\alpha$ for each index $\alpha<\lambda$ and because the graphs were ultimately derived from an ultraproduct, it would make sense to take the ultraproduct of all the $\mathbb{G}_\alpha$ and explore what this hyperfinite graph can tell us about the original type $p(\bar{x})$. There is a slight technicality here; the distribution graph sequence may be empty in some indices, making the ultraproduct empty. However, most of the elements (according to $\D$) of the distribution graph sequence must be nonempty, so we can take the ultraproduct of just the nonempty elements of the sequence.

\begin{defn}\label{essf}
	Suppose that $f$ is a distribution. Then the \emph{essential range of $f$} is the set 
	\[\mathrm{ess}(f):=\bigcup_{\Delta\in\P_\omega(\lambda)}f(\Delta).\]
\end{defn}

\begin{lem}\label{essfprop}
	Suppose that $f$	is a graph-like distribution with distribution graph sequence $(\mathbb{G}_\alpha)_{\alpha<\lambda}$. Then 
	\begin{enumerate}
		\item $\ess(f)\in\D$ and
		\item $\mathrm{ess}(f)=\{\alpha\in\lambda:\mathbb{G}_\alpha\neq \langle\emptyset,\emptyset\rangle\}$.
	\end{enumerate}
\end{lem}
\begin{proof}
	$(1)$: Pick $\Delta\in\P_\omega(\lambda)$. Then $f(\Delta)\se\ess(f)$ and $f(\Delta)\in\D$.
	
	$(2)$: Let $g_1,g_2$ determine the conjugate of $f$.
	From the definition of $\mathbb{G}_\alpha$, we have that $\mathbb{G}_\alpha\neq\langle\emptyset,\emptyset\rangle$ if and only if $g_1(\alpha)\neq\emptyset$. From the definition of the conjugate, $g_1(\alpha)=\emptyset$ if and only if $\alpha\notin f(\{\gamma\})$ for all $\gamma\in\lambda$. By the monotonicity of $f$, $\alpha\notin f(\{\gamma\})$ implies that $\alpha\notin f(\Delta)$ for any $\Delta\in\P_\omega(\lambda)$ with $\gamma\in\Delta$. Thus \[\{\alpha:\mathbb{G}_\alpha\neq\langle\emptyset,\emptyset\rangle\}\se\ess(f).\]
	
	Conversely, if $\alpha\in\ess(f)$, then $\alpha\in f(\Delta)$ for some $\Delta\in\P_\omega(\lambda)$. By the monotonicity of $f$, if $\gamma\in\Delta$, then $\alpha\in f(\{\gamma\})$. We then follow the same chain of if and only ifs in the previous paragraph to get the other inclusion.
\end{proof}

\begin{cor}\label{Drestriction}
	Suppose that $f$ is a distribution and $I=\ess(f)$. Then
	\begin{enumerate}
		\item If $A\se\lambda$ then $A\in\D$ if and only if $A\cap I\in\D$.
		\item The set $I$ has cardinality $\lambda$.
		\item The set \[\D\restriction I=\{A\in \P(I):A\in\D\}=\{A\cap I:A\in\D\}\] is a regular ultrafilter on $I$.
	\end{enumerate}
\end{cor}
\begin{proof}
	Statement $(1)$ and the fact that $\D\restriction I$ is an ultrafilter on $I$ are well known to hold for $I\in\D$. Checking that restriction preserves regularity follows from the associativity of general intersections. The set $I$ must have cardinality $\lambda$ since $I\in\D$ and all regular ultrafilters are uniform.
\end{proof}

\begin{conv}
	Suppose $f$ is a distribution for the type $p(\bar{x})$ and $I=\ess(f)$.
    If $p(\bar{x})$ is a type of $\widehat{M}=\M_\lambda^\D$ then we will use the same symbol $p$ for the type in $\widehat{M}'=\M_I^{\D\restriction I}$ whose parameters are the projections of the parameters of $p(\bar{x})$. This is justified by the fact that $\widehat{M}$ and $\widehat{M}'$ are isomorphic via projection onto the coordinates in $I$.
\end{conv}

\begin{defn}
	Let $f$ be a graph-like distribution with distribution graph sequence $(\mathbb{G}_\alpha)_{\alpha<\lambda}$ and $I=\ess(f)$. Then the \emph{ultragraph of $f$} is the graph $\mathbb{G}_f=\mathbb{G}_I^{\D\restriction I}$.
\end{defn}

Since the element $\beta\in g_1(\alpha)$ represents the realizability of the formula $\varphi_\beta(\bar{x},\bar{r}_\beta(\alpha))$ in $\M_\alpha$ (see e.g.\ Lemma~\ref{gldistconj}(1)), we should expect that the equivalence class of the constant function $\alpha\mapsto \beta$ in $\mathbb{G}_f$ carries some information about the realizability of the formula $\varphi_\beta(\bar{x},\bar{r}_\beta/\D)$ in $\widehat\M$. Unfortunately, $\beta$ will frequently not be an element of $g_1(\alpha)$ for all $\alpha$, so we will need to find representatives of elements of $\mathbb{G}_f$ that are equivalent modulo $\D$ to these constant functions because such representatives should have the same behavior as the constant functions that we really want.

\begin{lem}\label{eta}
	Let $f$ be a graph-like distribution with distribution graph sequence $(\mathbb{G}_\alpha)_{\alpha<\lambda}$ and $I=\ess(f)$. Then there is a unique injective function ${\eta:\lambda\to\mathbb{G}_f}$ such that if $(s_\beta)_{\beta\in I}$ has the property that for all $\beta\in I$ we have that $s_\beta\in\prod_{\alpha\in I}g_1(\alpha)$ and ${s_\beta/(\D\restriction I)=\eta(\beta)}$, then the set $\{\alpha\in I:s_\beta(\alpha)=\beta\}$ is an element of $\D\restriction I$.
\end{lem}
\begin{proof}
	For each $\beta\in\lambda$ define a function $h_\beta\in\prod_{\alpha\in I}\mathbb{G}_\alpha$ so that 
	\begin{equation}\label{hbeta}
	h_\beta(\alpha)=\pw{\beta,&\text{if }\beta\in G_\alpha\\ \min(G_\alpha),&\text{otherwise}}
	\end{equation}
	for all $\alpha\in I$.
	We note that 
	$\{\alpha\in\lambda:\beta\in G_\alpha\}=f(\{\beta\})\in\D$ and $f(\{\beta\})\se I$
	so, whenever $\beta,\gamma\in\lambda$ with $\beta\neq\gamma$, we have that \[\{\alpha\in I:h_\beta(\alpha)\neq h_\gamma(\alpha)\}\supseteq f(\{\beta\})\cap f(\{\gamma\})\in\D\restriction I.\] By \L o\'s theorem, the equivalence classes of $h_\beta$ and $h_\gamma$ are not equal in $\mathbb{G}_f$ and the function $\eta(\beta)=h_\beta/(\D\restriction I)$ is injective.
	
	Now suppose that $(s_\beta)_{\beta\in I}$ is such that $s_\beta/(\D\restriction I)=\eta(\beta)$. Then, since \[A=\{\alpha\in I:h_\beta(\alpha)=\beta\}\in \D\restriction I\text{ and }B=\{\alpha\in I:h_\beta(\alpha)=s_\beta(\alpha)\}\in\D\restriction I\] we have that 
	\[\{\alpha\in I:s_\beta(\alpha)=\beta\}\supseteq A\cap B\in\D\restriction I.\]
	We note that this property determines the value of $\eta(\beta)$ in $\mathbb{G}_f$.
\end{proof}

In Equation~\eqref{hbeta}, the choice of the element of $G_\alpha$ in the case that $\beta\notin G_\alpha$ does not affect the value of $\eta(\beta)$ as the collection of indices in which we have to make this choice is not in $\D\restriction \ess(f)$.

\begin{conv}
For any distribution $f$ we will fix $\eta_f\colon\lambda\to\mathbb{G}_f$ to be the unique function guaranteed by using Lemma~\ref{eta}.
\end{conv}

\section{Complete Subgraphs of $\mathbb{G}_f$}

Here we take a slight departure from talking about the type $p(\bar{x})$ specifically in order to get a feeling for what sort of information is contained in $\mathbb{G}_f$ and how that information might be accessed. One of our main results in this section is that the distribution $f$ has a multiplicative refinement if and only if the subgraph of $\mathbb{G}_f$ induced by $\eta_f[\lambda]$ sits inside an internal complete subgraph of $\mathbb{G}_f$ (see Corollary~\ref{multiffint}). More generally, we show how to take a complete subgraph $H$ of $\mathbb{G}_f$ of cardinality $\lambda$ and shuffle the vertices of $\mathbb{G}_f$ so that $\eta_f[\lambda]$ is moved to $H$ and so that the number of edges we need to remove from the domain copy of $\mathbb{G}_f$ to make this rearrangement a homomorphism is small enough to still have the domain copy of $\mathbb{G}_f$ be an ultragraph for a distribution $f_H$ refining $f$. We also show, under some assumptions on $H$, that $f_H$ is a distribution for a new type $q(\bar{x})$ that is related to $H$. Finally, we show that $f=f_{\eta_f[\lambda]}$ and that $f_H$ has a multiplicative refinement if and only if $H\se K\se\mathbb{G}_f$ where $K$ is an \emph{internal} complete subgraph of $\mathbb{G}_f$. We consider complete subgraphs of $\mathbb{G}_f$ other than $\eta_f[\lambda]$ in order to show that there is a strong connection between such graphs and $\lambda$-types of $\widehat{\M}$ that have distributions that are refinements of $f$.

\begin{defn}\label{rh}
    Let $f$ be a distribution with conjugate $(g_n)_{n\in\omega}$ and let $h$ be an element of $\prod_{\alpha\in\ess(f)}g_1(\alpha)$. Let $\D'$ refer to the regular ultrafilter $\D\restriction\ess(f)$. Recall the definition of $(\bar{r}_\beta)_{\beta<\lambda}$ from Convention~\ref{lambdatype}. If there is some $\beta\in\lambda$ so that the set 
    \[B_\beta:=\{\alpha\in\ess(f):\varphi_{h(\alpha)}(\bar{x};\bar{y})=\varphi_\beta(\bar{x};\bar{y})\}\]
    is an element of $\D'$, define the
    element $\bar{r}_h\in\prod_{\alpha<\lambda}M_\alpha^{n_\beta}/\D$ (where $n_\beta$ is the length of the tuple $\bar{r}_\beta$) to be the $\D$-equivalence class of any function where
    $\alpha\mapsto\bar{r}_{h(\alpha)}(\alpha)$ for all $\alpha\in B_\beta$. We will sometimes abuse
    notation to say that $\bar{r}_h\in\M_\lambda^\D$. We say that $\bar{r}_h$ is \emph{the parameter induced by $h$}.
\end{defn}

\begin{rmrks}\hspace{0em}
    \begin{enumerate}
	\item Given an element $h\in\prod_{\alpha\in\ess(f)}g_1(\alpha)$, the element $\bar{r}_h$ is well-defined because the representation of $\bar{r}_h$ given in the definition is defined for each $\alpha\in B_\beta$ and $B_\beta\in\D'\se\D$.
	\item The prototypical example of the sort of $h\in\prod_{\alpha\in\ess(f)}g_1(\alpha)$ that we wish to consider are the $h_\beta$ defined by Equation~\eqref{hbeta} in Lemma~\ref{eta}. Our main use of this definition will be to take $\lambda$ elements of $\mathbb{G}_f$ and treat a given set of their representations as if they were $h_\beta$.
	\item The condition that there exists a $\beta\in\lambda$ such that $B_\beta\in\D'$ is not only to guarantee that we can say that the length of $\bar{r}_h$ is finite according to $\D'$, but also because we will wish to use $\bar{r}_h$ as a parameter in the formula $\varphi_\beta(\bar{x};\bar{r}_h)$. 
	\item The condition that there is such a $\beta$ becomes trivial in the case that $\varphi_\beta=\varphi_\gamma$ for all $\beta,\gamma\in\lambda$ (in this case we say that $p(\bar{x})$ is a $\varphi$-type). This case is what occurs for the types we consider later (see e.g.\ the definition of $\mathrm{SOP}_2$-types in Definition~\ref{sop2type}). More generally, it was shown by Malliaris in \cite{mmvarphitypes} that the study of Keisler's order can be reduced to the study of whether $\varphi$-types are realized in regular ultrapowers.
	\item The choice of $\beta$ such that $B_\beta\in\D'$ is unique in the sense that $\gamma$ will also have $B_\gamma\in\D'$ if and only if $\varphi_\beta=\varphi_\gamma$.
    \end{enumerate}
\end{rmrks}

\begin{lem}\label{compimptype}
	Let $f$ be a graph-like distribution with conjugate determined by $g_1,g_2$ and with graph distribution sequence $(\mathbb{G}_\alpha)_{\alpha<\lambda}$ and ultragraph $\mathbb{G}_f=\mathbb{G}_{\ess(f)}^{\D'}$ where $\D'=\D\restriction \ess(f)$. Let $H\se\mathbb{G}_f$ induce a complete subgraph of $\mathbb{G}_f$. If 
	\begin{enumerate}
	    \item $H'$ is a set of representatives for the elements of $H$ where $H'$ is a subset of $\prod_{\alpha\in\ess(f)}g_1(\alpha)$,
	    \item each element of $H$ has exactly one representative in $H'$, and 
	    \item for each $h\in H'$ there exists a $\beta_h\in\lambda$ so that 
	\[B_h:=\{\alpha\in\ess(f):\varphi_{h(\alpha)}(\bar{x};\bar{y})=\varphi_{\beta_h}(\bar{x};\bar{y})\}\in\D'\]
	\end{enumerate} 
	then the collection of formulas
	\begin{equation}\label{typeq}
	q(\bar{x})=\{\varphi_{\beta_h}(\bar{x},\bar{r}_h):h\in H'\}
	\end{equation}
	is a type in $\widehat\M$.
\end{lem}
\begin{proof}
	By the characterization of a type in an ultraproduct, it will be enough to show that $\llbracket\Delta\rrbracket_{q(\bar{x})}^\D\in\D$ for every finite $\Delta\se H'$. For each $\Delta\in\P_\omega(H')$ define $B_\Delta=\bigcap_{h\in\Delta} B_h$. Following definitions,
	\[
		\llbracket\Delta\rrbracket_{q(\bar{x})}^\D
		\supseteq\Big\{\alpha\in B_\Delta:\M_\alpha\vDash\exists x,\bigwedge_{h\in \Delta}\varphi_h(x,\bar{r}_{h(\alpha)}(\alpha))\Big\}=:\llbracket \Delta\rrbracket^{B_\Delta}_{q(\bar{x})}
	\]
	where $\bar{r}_h$ is the parameter induced by $h$. For all $\alpha\in B_\Delta$ define $\Delta(\alpha):=\{h(\alpha):h\in\Delta\}$. If $|\Delta|\geq 2$, define
	\begin{equation}\label{kdelta}
		K_\Delta :=\{\alpha\in B_\Delta:[\Delta(\alpha)]^2\se g_2(\alpha)\}.
	\end{equation}

	The set $K_\Delta$ is the set of indices $\alpha$ for which $\Delta(\alpha)$ induces a complete subgraph of $\mathbb{G}_\alpha$. Let $\D'=\D\restriction \ess(f)$. Then, because $\Delta$ induces a finite complete subgraph of $H$ after modding out by $\D'$, we have that 
	\[\mathbb{G}_f\vDash \bigwedge_{\{h_1,h_2\}\in[\Delta]^2} \left(h_1/\D'\mathrel{E} h_2/\D'\right)\vee \left(h_1/\D'=h_2/\D'\right),\] 
	so, by \L o\'s theorem we have that 
	\[\mathbb{G}_\alpha\vDash \bigwedge_{\{h_1,h_2\}\in[\Delta(\alpha)]^2} \left(h_1(\alpha)\mathrel{E} h_2(\alpha)\right)\vee \left(h_1(\alpha)=h_2(\alpha)\right)\] 
	for $\D'$-many $\alpha$, and hence $\D$-many $\alpha$ by Corollary~\ref{Drestriction}(1). Because $B_\Delta\in\D'\se \D$, we have that $K_\Delta$ is an element of $\D$. By Lemma~\ref{gldistconj}(2), we have that 
	\[K_\Delta\se \llbracket \Delta\rrbracket^{B_\Delta}_{q(\bar{x})}\se\llbracket\Delta\rrbracket_{q(\bar{x})}^\D\]
	and so $\llbracket\Delta\rrbracket_{q(\bar{x})}^\D\in\D$. 
	
	The case of $|\Delta|=1$ follows from $\Delta(\alpha)\se g_1(\alpha)$, which implies $\llbracket\Delta\rrbracket_{q(\bar{x})}^\D=\lambda$.
\end{proof}

\begin{rmrks}\hspace{1 em}
\begin{enumerate}
	\item The converse to the above lemma is not true in general. There may be types of the same form as $q(\bar{x})$ (that is, types with representations $(h_\beta)_{\beta<\lambda}$ for the parameters coming from $\prod_{\alpha<\lambda}\{r_\beta(\alpha):\beta<\lambda\}$ and with $\beta_h$'s) but such that $\{h_\beta/\D:\beta<\lambda\}$ does not induce a complete subgraph of $\mathbb{G}_f$. This can be thought of as saying that $f$ carries enough information to recognize that $p(\bar{x})$ is a type, but generally does not carry enough information to recognize that all types with similar representations for their parameters are types. There is a converse of this sort when $i_f[\mathbb{G}_f]$ is an induced subgraph of $\mathbb{G}_L$ (the argument is similar to the proof of Corollary~\ref{indsublug}). See Section~6 for the definitions of $i_f$ and $\mathbb{G}_L$.
	\item The type $q(\bar{x})$ does not depend on the choice of the representatives in $H'$. If there are two representatives $h,h'$ for the same element of $\mathbb{G}_f$ they must be equal at $\D$-many indices, the representations for $\bar{r}_h$ and $\bar{r}_{h'}$ must agree at $\D$-many indices, and hence $\bar{r}_h$ and $\bar{r}_{h'}$ are equal. A similar argument shows that $\beta_h=\beta_{h'}$ in this case.
\end{enumerate}
\end{rmrks}

We wish to study the connection between induced complete subgraphs of $\mathbb{G}_{\ess(f)}^{\D'}=\mathbb{G}_f$ and the types of the form given in Equation~\eqref{typeq}. In order to gain an idea of where this analysis will go, recall that complete subgraphs of $\mathbb{G}_\alpha$ indicate subsets of $p(\bar{x})$ that, when viewed in the $\alpha$-th index structure, are satisfiable (Lemma~\ref{glseqdist}(1)). A major reason why we are able to make this conclusion is that each $g_1(\alpha)$ is finite. If $g_1(\alpha)$ were \emph{not} finite we would still be able to say that an infinite complete subgraph corresponds to a finitely satisfiable set of formula in the $\alpha$-th index structure (i.e.\ corresponds to a type) but there is no reason to assume that such a type is realized. Take, for example, the type $\{0<x<1/n:n\in\mathbb{N}^+\}$ in the dense linear ordering $\mathbb{Q}$. The satisfiability graph for this type is complete, but there is no realization of the type in $\mathbb{Q}$. Since the type $p(\bar{x})$ corresponds to the (infinite!) image of $\eta_f$, why should we expect that we will get any information from $\mathbb{G}_f$ other than the observation that $p(\bar{x})$ is a type?

If we treat $\mathbb{G}_f$ as an arbitrary graph and forget that it comes from an ultraproduct of finite graphs then we cannot extract any more information about whether $p(\bar{x})$ is realized. However, if we do remember this extra information, we then find that \L o\'s theorem tells us that $\mathbb{G}_f$ has many of the features of a finite graph so long as we restrict ourselves to looking at internal parts of the structure. We continue by exploring this interaction of apparent finiteness of internal subsets and the types defined in the previous lemma, after taking care of a few technical details allowing us to pinpoint particular distributions for the types described in Lemma~\ref{compimptype}.

\begin{lem}\label{hdist}
	Let $f$ be a graph-like distribution with conjugate $(g_n)_{n\in\omega}$ and $H$ a $\lambda$-sized complete subgraph of $\mathbb{G}_f$ with transversal $H'$. Let $\D'=\D\restriction\ess(f)$. If $(h_\beta)_{\beta\in\lambda}$ is an enumeration of $H'$ and for each $\beta\in \lambda$ there exists an $\beta_h\in\lambda$ so that 
	\[B_\beta:=\{\alpha\in\ess(f):\varphi_{h_\beta(\alpha)}(\bar{x};\bar{y})=\varphi_{\beta_h}(\bar{x};\bar{y})\}\in\D',\] then there is a graph-like distribution $f_{H'}\colon\P_\omega(\lambda)\to\D$ for the type \[q(\bar{x})=\{\varphi_{\beta_h}(\bar{x},\bar{r}_{h_\beta}):\beta<\lambda\}\] where, if $(\ell_n)_{n\in\omega}$ is the conjugate of $f_{H'}$, then $\ell_1,\ell_2$ are defined for all  $\alpha\in\lambda$ by
	\begin{align*}
		\ell_1(\alpha)&:=g_1(\alpha)\cap\{\beta\in\lambda:\alpha\in B_\beta\}\text{ and }\\
		\ell_2(\alpha)&:=[\ell_1(\alpha)]^2\cap\{\{\beta,\gamma\}\in[\lambda]^2:\{h_\beta(\alpha),h_\gamma(\alpha)\}\in g_1(\alpha)\cup g_2(\alpha)\}.
	\end{align*}
\end{lem}
\begin{proof}
	The function $\ell_1$ produces a finite subset of $\lambda$ for any input $\alpha<\lambda$ as $\ell_1(\alpha)\se g_1(\alpha)$. The image of $\ell_2$ consists of finite sets as $\ell_2(\alpha)\se[\ell_1(\alpha)]^2$ and $\ell_1(\alpha)$ is finite.
	We will proceed by checking that the functions $\ell_1,\ell_2$ satisfy conditions $(1)$--$(4)$ of Lemma~\ref{gldistconj}.  
	
	$(1)$: Suppose that $\beta\in\ell_1(\alpha)\se g_1(\alpha)$. Then $\alpha\in\ess(f)$ by Lemma~\ref{essfprop}(2). We thus have
	\[\M_\alpha\vDash\exists \bar{x},\,\varphi_{\beta_h}(\bar{x},\bar{r}_{h_\beta(\alpha)}(\alpha))\]
	because $h_\beta(\alpha)\in g_1(\alpha)$ by definition and because we can apply Lemma~\ref{gldistconj}(1) to $g_1$.
	
	$(2)$: Suppose that $\Delta\in[\lambda]^n$ for $n\geq 2$, $\alpha<\lambda$, and $[\Delta]^2\se \ell_2(\alpha)$. Then $[\Delta]^2\se [g_1(\alpha)]^2$ as $\ell_2(\alpha)\se [\ell_1(\alpha)]^2\se[g_1(\alpha)]^2$, so $\alpha\in\ess(f)$. Moreover, for each $\{\beta,\gamma\}\in[\Delta]^2$, we must have $\{h_\beta(\alpha),h_\gamma(\alpha)\}\in g_2(\alpha)$. This means that 
	\[\big\{\{h_\beta(\alpha),h_\gamma(\alpha)\}:\{\beta,\gamma\}\in[\Delta]^2\big\}\se[g_2(\alpha)]^2.\] 
	Now applying Lemma~\ref{gldistconj}(2) to $\Delta'=\{h_\beta(\alpha):\beta\in\Delta\}$ and $g_2$, we get that
	\[\M_\alpha\vDash\exists \bar{x},\bigwedge_{\beta\in\Delta}\varphi_{\beta_h}(\bar{x},\bar{r}_{h_\beta(\alpha)}(\alpha))\]
	as desired. In the case that $\{h_\beta(\alpha),h_\gamma(\alpha)\}\in g_1(\alpha)$ we know that $h_\beta(\alpha)=h_\gamma(\alpha)$ and then the desired result is the statement proved in $(1)$ above.
	
	$(3)$: We have for any $\alpha<\lambda$ that 
	$\ell_2(\alpha)\se [\ell_1(\alpha)]^2$
	by the definition of $\ell_2$.
	
	$(4)$: Pick $\Delta\in[\lambda]^2$. In order to show that 
	\[A:=\{\alpha\in\lambda:\Delta\in \ell_2(\alpha)\}\in\D,\] it will be enough to check that
	\begin{align*}
		C&:=\big\{\alpha\in\lambda:\Delta\in\{\{\beta,\gamma\}\in[\lambda]^2:\{h_\beta(\alpha),h_\gamma(\alpha)\}\in g_2(\alpha)\}\big\}\in\D \text{ and }\\
		D&:=\big\{\alpha\in\lambda:\Delta\in [\ell_1(\alpha)]^2\big\}\in\D
	\end{align*}
	as $A= C\cap D$. Define 
	\[D'= \{\alpha\in\lambda:\Delta\in g_2(\alpha)\}\text{ and }B_\Delta=\bigcap_{\beta\in\Delta}B_{\beta}\]
	and notice that $D'\cap B_\Delta\se D$ by the fact that $g_2(\alpha)\se[g_1(\alpha)]^2$ by Lemma~\ref{gldistconj}(3). By Lemma~\ref{gldistconj}(4) applied to $g_2$, we have that $D'\in\D$, so $D\in\D$ as well. The set $C$ may be rewritten as 
	\[C=\big\{\alpha\in\lambda:\Delta=\{\beta,\gamma\}\text{ and }\{h_\beta(\alpha),h_\gamma(\alpha)\}\in g_2(\alpha)\big\}.\]
	Define the set
	\[\Delta':=\{h_\beta:\beta\in\Delta\}.\]
	Then $K_{\Delta'}\se C$ where $K_{\Delta'}$ is defined as in Equation~\eqref{kdelta}. The same argument as in the proof of the previous lemma then shows that $K_{\Delta'}\in\D$.
\end{proof}

\begin{lem}\label{psi}
	Let $f$ be a graph-like distribution with conjugate $(g_n)_{n\in\omega}$ and $H$ a $\lambda$-sized complete subgraph of $\mathbb{G}_f$ with transversal $H'$. Let $\D_1=\D\restriction\ess(f)$ and $(h_\beta)_{\beta\in\lambda}$ be an enumeration of $H'$ such that for each $\beta\in \lambda$ there exists a $\beta_h\in\lambda$ so that 
	\[B_\beta:=\{\alpha\in\ess(f):\varphi_{h_\beta(\alpha)}(\bar{x};\bar{y})=\varphi_{\beta_h}(\bar{x};\bar{y})\}\in\D_1.\] Let $q(\bar{x})$, $f_{H'}$, $\ell_1$, and $\ell_2$ be defined as in the previous lemma and $\D_2=\D\restriction \ess(f_{H'})$. Then
	\begin{enumerate}
		\item Define $\mathrm{loop}(\mathbb{G}_f)$ to be the graph obtained by adding a loop to every vertex in $\mathbb{G}_f$. The map $\Psi:\mathrm{loop}(\mathbb{G}_{f_{H'}})\to\mathrm{loop}(\mathbb{G}_f)$ defined for each $j/\D_2\in \mathbb{G}_{f_{H'}}$ by 
			\[\Psi(j/\D_2)=(h_{j(\alpha)}(\alpha))_{\alpha\in\ess(f)}/\D_1\]
			is a graph homomorphism that also preserves non-edges.
		\item $\Psi\restriction\eta_{f_{H'}}[\lambda]$ is an isomorphism between the induced subgraphs of $\eta_{f_{H'}}[\lambda]$ and $H$ in $\mathbb{G}_{f_{H'}}$ and $\mathbb{G}_f$ respectively.
	\end{enumerate}
\end{lem}

\begin{proof}
	We note that $\Psi$ is well-defined as $j/\D_2=k/\D_2$ implies that $j(\alpha)=k(\alpha)$ for $\D_2$-many $\alpha$ in $\ess(f)$ and $h_{j(\alpha)}(\alpha)=h_{k(\alpha)}(\alpha)$ for at least these $\alpha$. Using that $\ess(f_{H'})\se\ess(f)$ (as $\ell_1(\alpha)\se g_1(\alpha)$ for all $\alpha\in\lambda$ and then using Lemma~\ref{essfprop}(2)) we have that $\D_2=\D_1\restriction \ess(f_{H'})$. Then, by Corollary~\ref{Drestriction}(3), we have that $h_{j(\alpha)}(\alpha)$ and $h_{k(\alpha)}(\alpha)$ are equal for a set of $\alpha$ in $\D_1$. 
	
	$(1)$: Suppose that $j/\D_2$ and $k/\D_2$ are elements of $\mathbb{G}_{f_{H'}}$. Then $\Psi(j/\D')$ and $\Psi(k/\D')$ are edge related in $\mathrm{loop}(\mathbb{G}_f)$ if and only if
	\[\{\alpha\in\ess(f):\{h_{j(\alpha)}(\alpha),h_{k(\alpha)}(\alpha)\}\in g_1(\alpha)\cup g_2(\alpha)\}\in\D_1.\]
	By the definition of $\ell_2$, this is equivalent to $\{j(\alpha),k(\alpha)\}\in\ell_2(\alpha)$ or $j(\alpha)=k(\alpha)$ for $\D_2$-many $\alpha$ in $\ess(f)$ (by the definition of $\mathbb{G}_{f_{H'}}$ we have $j(\alpha)$, $k(\alpha)\in\ell_1(\alpha)$), which is equivalent to $j/\D_2$ and $k/\D_2$ being edge related in $\mathrm{loop}(\mathbb{G}_{f_{H'}})$. As each step in this argument is an equivalence, we must have that $\Psi(j/\D_2)$ is edge related to $\Psi(k/\D_2)$ in $\mathrm{loop}(\mathbb{G}_f)$ if and only if $j/\D_2$ and $k/\D_2$ are edge related in $\mathrm{loop}(\mathbb{G}_{f_{H'}})$.
	
	$(2)$ Now that we know that $\Psi$ is a graph homomorphism that also preserves non-edges, showing that $\Psi\restriction\eta_{f_{H'}}[\lambda]$ is an isomorphism amounts to showing that the restriction is bijective. Having that $j/\D_2=\eta_{f_{H'}}(\beta)$ for some $\beta\in\lambda$ is equivalent, by Lemma~\ref{eta}, to having that 
	\[\{\alpha\in\ess(f):j(\alpha)=\beta\}\in\D_2.\]
	Then the statement $k/\D_1=\Psi(j/\D_2)$ is equivalent to 
	\[\{\alpha\in\ess(f):k(\alpha)=h_\beta(\alpha)\}\in\D_1,\] 
	which is equivalent by \L o\'s theorem to $k/\D_1=h_\beta/\D_1$. As this argument consists of a series of ``if and only if'' statements, we have that $\Psi$ is bijective when restricted to $\eta_{f_{H'}}[\lambda]$.
\end{proof}

\begin{conv}
    We continue to make the distinction between $\D_1$ and $\D_2$ as defined above. However, since $\D_2\se \D_1$ and $A\in \D_1$ if and only if $A\cap\ess(f_{H'})\in\D_2$, we will treat the statements $A\in\D_1$ and $A\in\D_2$ as if they were equivalent and do not spend time worrying about how to translate between the statements.
\end{conv}

In order to get to the result that $\Psi$ allows some transfer of information between the types $p(\bar{x})$ and $q(\bar{x})$, we will first prove that ``nice enough'' homomorphisms between ultraproducts of graphs using essentially the same ultrafilter preserve internal subsets and complete subsets. The ``nice enough'' referers to functions that are what non-standard analysts would call ``internal functions'' that are either homomorphisms or preserve non-edges. The following proof is well-known in non-standard analysis, but we feel that it is worth presenting here.

\begin{lem}\label{inthomandcomp}
    Suppose that $I$ and $J$ are sets with $J\se I$ and $(\mathbb{G}_i)_{i\in I}$ and $(\mathbb{H}_j)_{j\in J}$ are sequences of graphs. Further suppose that $\D_I$ is an ultrafilter on $I$ and $\D_J$ is an ultrafilter on $J$ with $\D_J\se\D_I$ and that there is a sequence of functions $(\theta_j)_{j\in J}$ where $\theta_j\colon \mathbb{H}_j\to\mathbb{G}_j$ for each $j\in J$. Define 
    \[\mathbb{G}:=\prod_{i\in I}\mathbb{G}_i/\D_I\,\text{ and }\,\mathbb{H}:=\prod_{j\in J}\mathbb{H}_j/\D_J.\]
    Then
    \begin{enumerate}
        \item The function $\Theta:\mathbb{H}\to\mathbb{G}$ defined for all $p/\D_J$ in $\mathbb{H}$ by
        \[\Theta(p/\D_J):=\left[i\mapsto\pw{\theta_i(p(i));&\text{if }i\in J\\\text{any element of }\mathbb{G}_i;&\text{otherwise}}\right]\bigg/\D_I\]
        is well-defined and unique,
        \item if $A\se \mathbb{G}$ and $B\se\mathbb{H}$ are internal then so are $\Theta[B]$ and $\Theta^{-1}[A]$,
        \item if $\{j\in J:\theta_j\text{ is a graph homomorphism}\}$ is an element of $\D_J$ and $K\se \mathbb{H}$ induces a complete subgraph of $\mathbb{H}$ then $\Theta[K]$ induces a complete subgraph of $\mathbb{G}$, and
        \item if $\{j\in J:\theta_j\text{ preserves non-edges as a map }\mathbb{H}_j\to\mathrm{loop}(\mathbb{G}_j)\}$ is an element of $\D_J$ and $K\se\mathbb{G}$ induces a complete subgraph of $\mathbb{G}$ then $\Theta^{-1}[K]$ is a complete subgraph of $\mathbb{H}$.
    \end{enumerate}
\end{lem}
\begin{proof}
    $(1)$: If $p/\D_J=p'/\D_J$, then $p(i)=p'(i)$ for $\D_j$-many $i\in J$, so $\theta_i(p(i))=\theta_i(p'(i))$ for $\D_J$-many $i\in J$ as well. Since $\D_J\se \D_I$, this is enough to show well-definedness and uniqueness.
    
    $(2)$: We start with $\Theta[B]$. As $B$ is internal, we can write $B$ as an ultraproduct of the form
    \[\prod_{j\in J}B_j/\D_J\]
    where each $B_j\se H_j$. Define for each $i\in J$ the set
	\[B'_i=\{\theta_i(h):h\in B_i\}\] and for $i\in I\sm J$ define $B'_i=G_i$.
	We claim that $\Theta[B]=\prod_{i\in I}B'_i/\D_I=:\widehat{B}$. Suppose that $p\in\prod_{i\in I}B'_i$. Then 
	\[(\theta_i(p(i)))_{i\in I}/\D_I=\Theta((p\restriction J)/\D_J).\]
	As every element of $\mathbb{H}$ can be written as $\Theta((p\restriction J)/\D_J)$ for some $p\in\prod_{i\in I}B'_i$, we have that $\Theta[B]=\widehat{B}$ and so is internal.
	
	The proof that $\Theta^{-1}[A]$ is internal in $\mathbb{G}$ is essentially the same after swapping $A$ and $B$, $I$ and $J$, and $\Theta$ and $\Theta^{-1}$.
	
	$(3)$: Suppose that $k_1,k_2$ are representatives for distinct elements of $K$. Then $\theta_j(k_1(j))$ and $\theta_j(k_2(j))$ are edge related for $\D_J$-many $j\in J$. Then the representatives given for $\Theta(k_1)$ and $\Theta(k_2)$ are edge related for $\D_I$-many $i\in I$. In particular, $\Theta$ is a graph homomorphism.
	
	$(4)$: As in $(3)$, $\Theta\colon\mathbb{H}\to\mathrm{loop}(\mathbb{G})$ will inherit the property that non-edges are preserved. Suppose that $k_1,k_2$ are distinct elements of $\Theta^{-1}[K]$. Then either there is an edge or a loop between $\Theta(k_1)$ and $\Theta(k_2)$, so $k_1$ and $k_2$ must be edge related in $\mathbb{H}$.
\end{proof}

\begin{thm}\label{hmult}
	Let $f$ be a graph-like distribution for the $\lambda$-type $p(\bar{x})$ with graph distribution sequence $(\mathbb{G}_\alpha)_{\alpha<\lambda}$, let $\D_1=\D\restriction\ess(f)$, and let $H\se\mathbb{G}_f=\mathbb{G}_{\ess(f)}^{\D_1}$ induce a complete subgraph with $|H|=\lambda$. If $H'=\{h_\beta:\beta<\lambda\}$ is a transversal of $H$ and for each $\beta\in \lambda$ there exists a $\beta_h\in\lambda$ so that 
	\[B_{\beta'}:=\{\alpha\in\ess(f):\varphi_{h_\beta(\alpha)}(\bar{x};\bar{y})=\varphi_{\beta_h}(\bar{x};\bar{y})\}\in\D_1,\]
	then the distribution $f_{H'}$ for the type \[q(\bar{x})=\{\varphi_{\beta_h}(\bar{x},\bar{r}_{h_\beta}):\beta<\lambda\}\] has a multiplicative refinement if and only if there is an internal $K\se\mathbb{G}_f$ that induces a complete subgraph with $H\se K$.
\end{thm}
\begin{proof}
	($\Rightarrow$) Suppose that $k:\P_\omega(\lambda)\to\D$ is a multiplicative refinement of $f_{H'}$. We know from Lemma~\ref{multviaseq} that the distribution graph sequence for $k$ consists of complete graphs. The ultragraph $\mathbb{G}_k$ is complete as completeness is a first-order property of a graph and hence preserved by ultraproducts. From Lemma~\ref{refviaseq} we know that the $\alpha$-th graph of the graph distribution sequence for $k$ is a subgraph of the $\alpha$-th graph of the graph distribution sequence for $f_{H'}$. Using Lemma~\ref{inthomandcomp}(1)--(3) where $I=\ess(f_{H'})$, $J=\ess(k)$, $A=\mathbb{G}_k$, and $(\theta_\alpha)_{\alpha\in J}$ where $\theta_\alpha$ is the inclusion map from the $\alpha$-th graph in the distribution graph sequence of $k$ to the distribution graph sequence of $f_{H'}$, we have that the inclusions maps induce a graph homomorphism from $\mathbb{G}_k$ to $\mathbb{G}_{f_{H'}}$ whose image is an internal complete subgraph, $K'$. Furthermore, this homomorphism sends $\eta_k[\lambda]$ to $\eta_{f_{H'}}[\lambda]$ so $\eta_{f_{H'}}\se K'$. We claim that $\Psi[K']$ is the desired $K\se\mathbb{G}_f$.
	
	From the definition of $\Psi$, we know that we can use Lemma~\ref{inthomandcomp}(2) and (3). That is, $\Psi[K']$ induces an internal complete subgraph of $\mathbb{G}_f$, and, by Lemma~\ref{psi}(2), $\Psi[K']$ contains $H$.
	
	($\Leftarrow$) Suppose now that $H\se K$ where $K$ is an internal subset of $G_f$ that induces a complete subgraph of $\mathbb{G}_f$. Then, by using Lemma~\ref{inthomandcomp}, we have that $\mathbb{G}_K:=\Psi^{-1}[K]$ is internal in $G_{f_{H'}}$ and induces a complete subgraph of $\mathbb{G}_{f_{H'}}$. We will use the fact that $\mathbb{G}_K$ is internal and complete to construct a multiplicative refinement of $f_{H'}$.
	
	Let $\D_2=\D\restriction\ess(f_{H'})$.
	As $\mathbb{G}_K$ is internal in $\mathbb{G}_{f_{H'}}$ we can express $\mathbb{G}_K$ in the form
	\[\prod_{\alpha\in\ess(f_{H'})}K_\alpha/\D_2\] 
	where $K_\alpha\se\ell_1(\alpha)$ for all $\alpha\in\ess(f_{H'})$. For each $\alpha\in\ess(f_{H'})$, define $C_\alpha$ to be $K_\alpha$ if $[K_\alpha]^2\se \ell_2(\alpha)$ and to be $\{\min K_\alpha\}$ otherwise (so that $[C_\alpha]^2=\emptyset\se\ell_2(\alpha)$ in this case). Then $C_\alpha$ is complete for all $\alpha$ in $\ess(f_{H'})$ and $C_\alpha=K_\alpha$ for $\D_2$-many $\alpha$ (the ultraproduct of the $K_\alpha$ is the complete graph $\mathbb{G}_K$ so $\D_2$-many $K_\alpha$ must induce complete subgraphs in $(\ell_1(\alpha),\ell_2(\alpha))$ by \L o\'s theorem).
	We claim that the functions defined for $\alpha\in\ess(f_{H'})$ by $s_1(\alpha):=C_\alpha$ and $s_2(\alpha):=[s_1(\alpha)]^2$ belong to the conjugate for a multiplicative refinement of $f_{H'}$ (multiplicative distributions must be graph-like by Lemma~\ref{multimpgl}). By Lemma~\ref{refviaseq}, $s_1$ and $s_2$ determine the conjugate of a refinement of $f_{H'}$ if and only if $(s_1(\alpha),s_2(\alpha))$ is a subgraph of $(\ell_1(\alpha),\ell_2(\alpha))$ for all $\alpha$ in $\ess(f_{H'})$ and $\eta_{f_{H'}}[\lambda]\se s_1(\ess(f_{H'}))^{\D_2}$ induces a complete subgraph of 
	\[\prod_{\alpha\in\ess(f_{H'})}(s_1(\alpha),s_2(\alpha))/\D_2\cong\mathbb{G}_K\]
	(note that this is a rewriting of the condition in Lemma~\ref{refviaseq}(2)). These follow from 
	\[s_1(\alpha)=C_\alpha\se K_\alpha\se\ell_1(\alpha)\text{ and }s_2(\alpha)\se\ell_2(\alpha)\]
	and the fact that $C_\alpha=K_\alpha$ for $\D_2$-many $\alpha$ in $\ess(f_{H'})$ so 
	\begin{align*}
	    C_{\ess(f)}^{\D_2}&=\mathbb{G}_K\\
	    &\supseteq \Psi^{-1}[H]\\
	    &= \eta_{f_{H'}}[\lambda]
	\end{align*}
	where the last equality follows from Lemma~\ref{psi}(2).
	Further, by Lemma~\ref{multviaseq}, we know that the distribution determined by $s_1,s_2$ is multiplicative as $s_2(\alpha)=[s_1(\alpha)]^2$, making the graph $(s_1(\alpha),s_2(\alpha))$ complete for each $\alpha$ in $\ess(f_{H'})$.
	Thus the distribution determined by $s_1,s_2$ is a multiplicative refinement of $f_{H'}$ and has $\mathbb{G}_K$ as its ultragraph.
\end{proof}

\begin{cor}\label{multiffint}
	 A graph-like distribution $f$ for $p(\bar{x})$ has a multiplicative refinement if and only if the image of $\eta_f$ in $\mathbb{G}_f$ is a subset of an internal complete subgraph.
\end{cor}
\begin{proof}
	Let $H=\eta_f[\lambda]$ and $H'=\{h_\beta:\beta\in\lambda\}$ where the $h_\beta$ are the functions defined in Equation~\eqref{hbeta}. Then $H'$ is a transversal of $H$ and 
	\[B_\beta=\{\alpha\in\ess(f):\varphi_{h_\beta(\alpha)}(\bar{x};\bar{y})=\varphi_\beta(\bar{x},\bar{y})\}\supseteq f(\{\beta\})\in\D\]
	by the definitions of the $h_\beta$ and $g_1(\alpha)$.
	We claim that the distribution $f_{H'}$ and type $q(\bar{x})$ defined in Lemma~\ref{hdist} are equal to $f$ and $p(\bar{x})$ respectively. The desired result then follows from Theorem~\ref{hmult}. To check that $f_{H'}=f$, it will be enough to show that $\ell_1=g_1$ and $\ell_2=g_2$ by Theorem~\ref{uniqglconj}. Having $\beta\in g_1(\alpha)$ is equivalent to $h_\beta(\alpha)=\beta$, which implies $\alpha\in B_\beta$. That is, 
	\[\ell_1(\alpha)=g_1(\alpha)\cap\{\beta\in\lambda:\alpha\in B_\beta\}=g_1(\alpha).\]
	Having $\{\beta,\gamma\}\in\ell_2(\alpha)$ is equivalent to $\{h_\beta(\alpha),h_\gamma(\alpha)\}\in g_2(\alpha)$ and $\beta,\gamma\in\ell_1(\alpha)= g_1(\alpha)$, so $h_\beta(\alpha)=\beta$ and $h_\gamma(\alpha)=\gamma$ and $\{\beta,\gamma\}\in g_2(\alpha)$. The reverse argument follows the same steps in reverse.
	
	The type $q(\bar{x})$ is 
	\[\{\varphi_\beta(\bar{x},\bar{r}_{h_\beta}):\beta<\lambda\}.\] Since $h_\beta(\alpha)=\beta$ for $\D_2$-many $\alpha$ in $\ess(f_{H'})$, we have that $\bar{r}_{h_\beta}=\bar{r}_\beta$, so $q(\bar{x})=p(\bar{x})$.
\end{proof}

\section{Graph-like Types}

The particular types ($\mathrm{SOP}_2$-types and types expressing that cuts in linear orders are filled) that we wish to study are not only special in that there are distributions for the type that are graph-like but that, for certain choices of the representations $\bar{r}_\beta\in\prod_{\alpha<\lambda}\M_\alpha$, \emph{every} distribution for the type can be extended to one that is graph-like. As we will see in Lemma~\ref{glextensions}, this special property is related to the fact that the \L o\'s map exhibits graph-like behavior.

\begin{conv}
    From now on, we will assume that $p(\bar{x})$ is a $(\varphi,\lambda)$-type (that is, both a $\varphi$-type and a $\lambda$-type) unless otherwise stated.
\end{conv}

\begin{defn}\label{gltype}
	The type $p(\bar{x})$ is called a \emph{graph-like type} if there is a choice of representatives $(\bar{r}_\beta)_{\beta<\lambda}$ for the parameters of $p(\bar{x})$ so that $L$, the \L o\'s map of $p(\bar{x})$ with this choice of representatives, satisfies 
	\[L(\Delta)=\bigcap_{\Phi\in[\Delta]^2}L(\Phi)\]
	for all finite $\Delta\se\lambda$ with $|\Delta|\geq 2$.
\end{defn}

\begin{prop}
    Every type of the form $p(\bar{x})=\{\varphi(\bar{x},\bar{a}_\beta):\beta<\lambda\}$ of the ultraproduct $\widehat{\M}=\prod_{\beta<\lambda}\M_\alpha/\D$ with a multiplicative distribution (i.e.\ $p(\bar{x})$ is realized in $\widehat\M$) is graph-like.
\end{prop}
\begin{proof}
    If $f$ is a multiplicative distribution of $p(\bar{x})$ with parameters given by $(\bar{r}_\beta)_{\beta<\lambda}$, then $f$ is graph-like by Lemma~\ref{multimpgl} and every graph in its distribution graph sequence is complete by Lemma~\ref{multviaseq}. A choice of representatives that witness $p(\bar{x})$ being a graph-like type is given by 
    \[\bar{s}_\beta(\alpha)=\pw{\bar{r}_{h_\beta(\alpha)}(\alpha);&\alpha\in\ess(f)\\\bar{\gamma}_\alpha;&\text{otherwise}}\] 
    where the $h_\beta$ are the functions in $\prod_{\alpha\in\ess(f)}\M_\alpha$ defined in Equation~\eqref{hbeta} and $\bar{\gamma}_\alpha$ is a fixed tuple of parameters from $\M_\alpha$. The \L o\'s map under this choice of representatives is
    \[L(\Delta)=\ess(f)\cup\big\{\alpha\in\lambda\sm\{\ess(f)\}:\M_\alpha\vDash \exists\bar{x},\, \varphi(\bar{x};\bar{\gamma}_\alpha)\big\}\]
    for all $\Delta\in\P_\omega(\lambda)$, which satisfies the requirement for $p(\bar{x})$ to be graph-like because $L$ is constant.
\end{proof}

In particular, over a $\lambda$-good ultrafilter $\D$ \emph{every $\varphi$-type that is also a $\lambda$-type is graph-like}. Thus, if being graph-like is not a trivial property, it must depend on the ultrafilter $\D$ being considered. We are most interested in types that are graph-like independent of the choice of $\D$ as this would appear to reveal something about the underlying structure of the type and the complete theory from which it came.

To continue the theme of investigating the ways in which the \L o\'s map acts in ways similar to a distribution and to rephrase Lemma~\ref{gldistconj}, we define the conjugate of the \L o\'s map.

\begin{defn}\label{losconj}
	The \emph{\L o\'s conjugate} for the type $p(\bar{x})$ is the sequence of functions $k_n\colon\lambda\to\mathcal{P}([\lambda]^n)$ such that 
	\[k_n(\alpha)=\{\Delta\in[\lambda]^n: \alpha\in L(\Delta)\}.\]
\end{defn}

\begin{lem}\label{gldistconjk}
    If $g_1$ and $g_2$ are functions $g_i\colon\lambda\to\P_\omega([\lambda]^i)$ then $g_1$ and $g_2$ belong to the conjugate of a (unique) monotone graph-like distribution $f$ for the $\lambda$-type $p(\bar{x})$ in $\widehat{\M}=\prod_{\alpha<\lambda}\M_\alpha/\D$ with the representatives $(\bar{r}_\beta)_{\beta<\lambda}$ if and only if the functions satisfy the conditions:
    \begin{enumerate}
        \item Let $(k_n)_{n\in\omega}$ be the \L o\'s conjugate for $p(\bar{x})$. For all $\alpha\in\lambda$ we have that $g_1(\alpha)\se k_1(\alpha)$.
        \item For all $\alpha<\lambda$ and all $\Delta\in\P_\omega(\lambda)$ with $n:=|\Delta|\geq2$ and $[\Delta]^2\se g_2(\alpha)$ we have that $\Delta\in k_n(\alpha)$.
        \item For all $\alpha\in\lambda$ it is the case that $g_2(\alpha)\se[g_1(\alpha)]^2$.
        \item For each $\Delta\in[\lambda]^2$ the set $\{\alpha\in\lambda:\Delta\in g_2(\alpha)\}$ is an element of $\D$.
    \end{enumerate}
\end{lem}
\begin{proof}
    By comparison with Lemma~\ref{gldistconj}, we see that it will be enough to show that condition~(1) of this lemma is equivalent to the combination of condition~(1) of Lemma~\ref{gldistconj}, and that condition~(2) of this lemma is equivalent to condition~(2) in Lemma~\ref{gldistconj}. Both of these follow from the definition of the \L o\'s map and the \L o\'s conjugate.
\end{proof}

\begin{lem}\label{glextensions}
	If $p(\bar{x})$ is a graph-like type then every distribution (with representatives witnessing $p(\bar{x})$ being graph-like) is a refinement of a graph-like distribution.
\end{lem}
\begin{proof}
	Let $f$ be a distribution for $p(\bar{x})$ and $(g_n)_{n\in\omega}$ the conjugate of $f$. We claim that $g_1,g_2$ belong to the conjugate of a graph-like distribution $f'$ with the property that $f\se f'$. We proceed by checking the conditions of Lemma~\ref{gldistconjk} and then checking that $f$ is a refinement of the unique graph-like distribution having $g_1$ and $g_2$ in its conjugate.
	
	$(1)$: This is equivalent to Lemma~\ref{distconj}(1) (for $n=1$) by the definitions of the \L o\'s map and the \L o\'s conjugate. We know that $g_1$ satisfies the referenced condition as it belongs to the conjugate of $f$.
	
	$(2)$: If $\Delta\in\P_\omega(\lambda)$ and $n:=|\Delta|\geq 2$ then
	\[[\Delta]^2\se g_2(\alpha)\Rightarrow [\Delta]^2\se k_2(\alpha)\Rightarrow \Delta\in k_n(\alpha)\] where the last implication follows from $L$ being graph-like (see Theorem~\ref{uniqglconj}).
	
	$(3)$: Follows from Lemma~\ref{distconj}(2) applied to $f$ with $n=2$ and $m=1$.
	
	$(4)$: For $\Delta\in[\lambda]^2$, we have that
	\[
	    \{\alpha\in\lambda:\Delta\in g_2(\alpha)\}\in\D
	\]
	by Lemma~\ref{distconj}(3) applied with $n=2$.
	
	It remains to show that $f\se f'$. Let $(\ell_n)_{n\in\omega}$ be the conjugate of $f'$. By Lemma~\ref{distconj}(2) applied to $f$, we get if $\Delta\in\P_\omega(\lambda)$ and $\alpha\in\lambda$ with $\Delta\in g_n(\alpha)$ (equivalently, $\alpha\in f(\Delta)$) for $n\geq 2$, then \[[\Delta]^2\se g_2(\alpha)=\ell_2(\alpha).\] In particular, this implies that $\Delta\in\ell_n(\alpha)$ by Equation~\eqref{glconjeq} (equivalently, $\alpha\in f'(\Delta)$). We thus have that $f(\Delta)\se f'(\Delta)$ for all $\Delta\in\P_\omega(\lambda)$ with $|\Delta|\geq 2$. As $\ell_1=g_1$, we also have that $f(\{\beta\})\se f'(\{\beta\})$ for all $\beta\in\lambda$ by Equation~\eqref{disteqconj}.
\end{proof}

\begin{cor}\label{glgood}
	If $p(\bar{x})$ is a graph-like $\varphi$-type of cardinality $\lambda$, then $p(\bar{x})$ is realized whenever $\D$ has multiplicative refinements for all graph-like functions $\P_\omega(\lambda)\to\D$.
\end{cor}

One consequence of the \L o\'s map having graph-like behavior is that $L$ behaves almost as if it were a maximal graph-like distribution. In particular, we can construct an ultraproduct of graphs using the \L o\'s map as a template just as we did with graph-like distributions, and this ultragraph will contain all of the distribution ultragraphs in a compatible way (i.e.\ all of the embeddings of $\lambda$ (representing the type $p(\bar{x})$) in the various ultragraphs (see Lemma~\ref{eta}) will coincide in the \L o\'s ultragraph). By using Corollary~\ref{multiffint} and Lemma~\ref{glextensions}, this will allow us to conclude that the graph-like type $p(\bar{x})$ is realized precisely in the case that the image of $\lambda$ in the \L o\'s ultragraph is contained within an internal complete subgraph.

\begin{defn}\label{losug}
	If $p(\bar{x})$ is graph-like we say that the \emph{\L o\'s ultragraph} of $p(\bar{x})$ is
	\[\mathbb{G}_L:=\prod_{\alpha\in\ess(L)}(k_1(\alpha),k_2(\alpha))/\D'\]
	where $\ess(L)$ is defined analogously to $\ess(f)$ and $\D'=\D\restriction \ess(L)$.
\end{defn}

\begin{lem}\label{ugselg}
	Suppose that $p(\bar{x})$ is a graph-like $(\varphi,\lambda)$-type and that $f$ is a graph-like distribution for $p(\bar{x})$ with conjugate $(g_n)_{n\in\omega}$. Then $\mathbb{G}_f$ is an internal (not necessarily induced) subgraph of $\mathbb{G}_L$. In particular, the map $i_f:\mathbb{G}_f\to\mathbb{G}_L$ induced by the inclusions $g_1(\alpha)\to k_1(\alpha)$ for each $\alpha\in\ess(f)$ is an injective graph homomorphism with internal image.
\end{lem}
\begin{proof}
	Let the conjugate of $f$ be $(g_n)_{n\in\omega}$. We have that $g_1(\alpha)\se k_1(\alpha)$ and $g_2(\alpha)\se k_2(\alpha)$ for all $\alpha\in\ess(f)$ as $f\se L$. For each $\alpha\in\ess(f)$, let $i_\alpha:g_1(\alpha)\to k_1(\alpha)$ be the inclusion map and note that it is a graph homomorphism. For each $\alpha\in\ess(L)$, choose an element $\gamma_\alpha$ in $k_1(\alpha)$ (recall that $\alpha\in\ess(L)$ if and only if $k_1(\alpha)\neq\emptyset$). Let $\D_1=\D\restriction \ess(L)$ and $\D_1'=\D\restriction\ess(f)$. Define the map $i_f:\mathbb{G}_f\to\mathbb{G}_L$ given for all $j/\D_1'\in\mathbb{G}_f$ by the formula
	\begin{equation}\label{coreinclmap}
	    i_f(j/\D_1')=\left[\alpha\mapsto \pw{j(\alpha); &\text{if }\alpha\in\ess(f)\\ \gamma_\alpha; &\text{otherwise}}\right]\bigg/\D_1.
	\end{equation}
	As $\D_1'\se\D_1$, the function $i_f$ is well-defined and injective with internal image (see Lemma~\ref{inthomandcomp}).
\end{proof}

The value of $i_f(j/\D_1')$ in the above proof does not depend on the choice of $\gamma_\alpha$ as $\ess(f)\in\D_1$.

\begin{lem}\label{etaequiv}
    Let $f,f'$ be two graph-like distributions for the graph-like $\varphi$-type $p(\bar{x})$ using a common sequence of representatives $(\bar{r}_\beta)_{\beta<\lambda}$ for the parameters of $p(\bar{x})$. Recall the definitions of $\eta_f$ (Lemma~\ref{eta}) and $i_f$ (Lemma~\ref{ugselg}). Then the maps $i_f\circ\eta_f$ and $i_{f'}\circ\eta_{f'}$ are equal injective functions from $\lambda$ to $\mathbb{G}_L$.
\end{lem}
\begin{proof}
    Let the conjugates of $f$ and $f'$ be given by $(g_n)_{n\in\omega}$ and $(g'_n)_{n\in\omega}$ respectively.
    Fix some $\beta\in\lambda$ and let \[h_\beta\in\prod_{\alpha\in\ess(f)}g_1(\alpha)\;\text{ and }\;h'_\beta\in\prod_{\alpha\in\ess{f'}}g'_1(\alpha)\] be the representatives for $\eta_f(\beta)$ and $\eta_{f'}(\beta)$ given by Equation~\eqref{hbeta} (applied to $\eta_f,h_\beta$ and $\eta_{f'},h_\beta'$ respectively). Define $\D_1=\D\restriction \ess(f)$ and $\D_1'=\D\restriction\ess(f')$. Then we have that
    \begin{align*}
        \{\alpha\in\ess(L):h_\beta(\alpha)=h'_\beta(\alpha)\}&\supseteq \{\alpha\in\ess(L):\beta\in g_1(\alpha)\cap g'_1(\alpha)\}\\
        &=f(\{\beta\})\cap f'(\{\beta\})\\
        &\in\D.
    \end{align*}
    Applying \L o\'s theorem to the representatives for $i_f(h_\beta/\D_1)$ and $i_{f'}(h'_\beta/\D_1')$ given in Equation~\eqref{coreinclmap} shows that $i_f(\eta_f(\beta))=i_{f'}(\eta_{f'}(\beta))$. $i_f\circ\eta_f$ is injective because both $i_f$ and $\eta_f$ are.
\end{proof}

\begin{defn}\label{etal}
    If $p(\bar{x})$ is a graph-like $\varphi$-type, we will denote by $\eta_L$ the unique map $\lambda\to\mathbb{G}_L$ that can be written as $i_f\circ\eta_f$ for any graph-like distribution $f$ of $p(\bar{x})$.
\end{defn}

Identifying the subgraph induced by $\eta_L[\lambda]$ in $\mathbb{G}_L$ allows us to capture much of the information contained within distributions by looking at internal subgraphs of $\mathbb{G}_L$ that contain $\eta_L[\lambda]$ as a complete subgraph and that are \emph{hyperfinite}. We record a proof of this correspondence here but will not use it until the next section.

\begin{defn}
    A subset $A$ of an ultraproduct $\prod_{\alpha<\lambda}\M_\alpha/\D$ is called \emph{hyperfinite} if it can be expressed as an ultraproduct $\prod_{\alpha<\lambda}A_\alpha/\D$ where each $A_\alpha$ is a finite subset of $M_\alpha$. We say that a structure is \emph{hyperfinite} if its underlying set is hyperfinite.
\end{defn}

A hyperfinite subset of an ultraproduct $\widehat\M=\prod_{\alpha<\lambda}\M_\alpha/\D$ is necessarily an internal subset of $\widehat\M$ as every hyperfinite subset is expressible as an ultraproduct of subsets of $M_\alpha$.

\begin{thm}\label{extendingetalequalsdistribution}
    Let $p(\bar{x})$ be a graph-like type with \L o\'s conjugate $(k_n)_{n\in\omega}$ and let $H\se\mathbb{G}_L$. Then the following are equivalent:
    \begin{enumerate}
        \item  $H$ is hyperfinite with $\eta_L[\lambda]\se H$.
        \item There is a graph-like distribution $f$ for $p(\bar{x})$ such that $i_f[G_f]=H$.
    \end{enumerate}
    Moreover, if $H=\prod_{\alpha<\lambda}H_\alpha/\D$ with each $H_\alpha$ finite, then $f$ can be taken so that $i_f[\mathbb{G}_f]$ is the subgraph of $\mathbb{G}_L$ induced by $H$ and so that, if $(\mathbb{G}_\alpha)_{\alpha<\lambda}$ is the distribution graph sequence of $f$, the graph $\mathbb{G}_\alpha$ is the subgraph of $\langle k_1(\alpha),k_2(\alpha)\rangle$ induced by $H_\alpha$ for all $\alpha<\lambda$.
\end{thm}
\begin{proof}
    $(2)\Rightarrow(1)$: The ultragraph $\mathbb{G}_f$ is hyperfinite because the distribution graph sequence of $f$ consists of finite graphs. Let $i_\alpha$ be the inclusion $g_1(\alpha)\to k_1(\alpha)$. Then the image of $i_f$ is equal to $\prod_{\alpha<\lambda} i_\alpha[G_\alpha]/\D$, so we have that the set $H=i_f[G_f]$ is hyperfinite. This is a specific case of internal functions preserving hyperfinite sets.
    
    $(1)\Rightarrow(2)$: Let $\mathbb{H}=\langle H,E^{\mathbb{H}}\rangle$ be the subgraph of $\mathbb{G}_L$ induced by $H$ and pick finite $H_\alpha\se k_1(\alpha)$ so that $H=\prod_{\alpha<\lambda}H_\alpha/\D$. Then $\mathbb{H}=\prod_{\alpha<\lambda}\mathbb{H}_\alpha/\D$ where each $\mathbb{H}_\alpha$ is the finite subgraph of $\mathbb{K}_\alpha:=\langle k_1(\alpha),k_2(\alpha)\rangle$ induced by $H_\alpha$. We claim that $(\mathbb{H}_\alpha)_{\alpha<\lambda}$ is the distribution graph sequence of a graph-like distribution $f$ of $p(\bar{x})$. We verify that $(\mathbb{H}_\alpha)_{\alpha<\lambda}$ satisfies the conditions of Lemma~\ref{gldistconjk} ($g_1(\alpha)=H_\alpha$ and $g_2(\alpha)=E^{\mathbb{H}_\alpha}$ following the definition of a graph distribution sequence).
    
    Conditions $(1)$-$(3)$ of Lemma~\ref{gldistconjk} are equivalent to $\mathbb{H}_\alpha$ being a subgraph of $\mathbb{K}_\alpha$ given that $p(\bar{x})$ is graph-like. Condition $(4)$ is equivalent to $\eta_L[\lambda]$ inducing a complete subgraph in $\mathbb{H}$, which is equivalent to $\eta_L[\lambda]$ inducing a complete subgraph of $\mathbb{G}_L$.
    
    The result now follows from the fact that $i_f$ is the inclusion of $\mathbb{H}$ in $\mathbb{G}_L$.
\end{proof}

\begin{thm}\label{realiffint}
	Suppose that $p(\bar{x})$ is a graph-like $\varphi$- and $\lambda$-type with {\L}o{\'s} map $L$. Then $p(\bar{x})$ is realized in $\widehat\M$ if and only if the image of $\eta_L:\lambda\to\mathbb{G}_L$ is a subset of an internal $A\se G_L$ that induces a complete subgraph.
\end{thm}
\begin{proof}
	($\Rightarrow$): The type $p(\bar{x})$ is realized if and only if there is a multiplicative distribution $j$ for $p(\bar{x})$ by Fact~\ref{keislerfact}. Then, by Corollary~\ref{conjmult}, $j$ is a graph-like distribution for $p(\bar{x})$ and, by Lemma~\ref{multviaseq} and Lemma~\ref{ugselg}, with $\mathbb{G}_j$ being an internal complete subgraph of $\mathbb{G}_L$. Moreover, by Lemma~\ref{etaequiv}, $i_j\circ\eta_j=\eta_L$ so $i_j[\mathbb{G}_j]$ contains the image of $\eta_L$. Thus $i_j[\mathbb{G}_j]$ is an internal subset of $\mathbb{G}_L$ extending $\eta_L[\lambda]$ and inducing a complete subgraph.
	
	($\Leftarrow$): By Lemma~\ref{etaequiv} we have that $\eta_f[\lambda]=\eta_L[\lambda]$ for any graph-like distribution $f$ of $p(\bar{x})$. Using Corollary~\ref{multiffint} then shows that every graph-like distribution has a multiplicative refinement, which is equivalent to $p(\bar{x})$ being realized by Fact~\ref{keislerfact}.
\end{proof}

\begin{cor}\label{indsublug}
    If $p(\bar{x})$ is a graph-like type and $f$ is a graph-like distribution such that $i_f[\mathbb{G}_f]$ is an induced subgraph of $\mathbb{G}_L$ then the type $p(\bar{x})$ is realized if and only if $f$ has a multiplicative refinement.
\end{cor}
\begin{proof}
    $(\Rightarrow)$: If $p(\bar{x})$ is realized, then, by Theorem~\ref{realiffint}, there is a complete internal subgraph $\mathbb{K}$ of $\mathbb{G}_L$ containing $\eta_L[\lambda]=i_f\circ\eta_f[\lambda]$. We must have that $i_f$ preserves non-edges because the image of $i_f$ is an induced subgraph and we already know that $i_f$ is injective. By Lemma~\ref{inthomandcomp}(2) and (4) and the internal definition of $i_f$, we have that $\mathbb{K}':=i_f^{-1}[\mathbb{K}]$ is a complete internal subgraph of $\mathbb{G}_f$ containing $\eta_f[\lambda]$. Thus, by Corollary~\ref{multiffint}, the distribution $f$ has a multiplicative refinement.

    $(\Leftarrow)$: If $f$ has a multiplicative refinement, then $p(\bar{x})$ is realized by Fact~\ref{keislerfact}.
\end{proof}

\begin{cor}\label{typelikeexts}
    For a graph-like type $p(\bar{x})$ with representatives witnessing $p(\bar{x})$ being graph-like, $p(\bar{x})$ is realized in $\widehat{\M}$ if and only if every distribution $f$ extends to a distribution $\hat{f}$ that has a multiplicative refinement. Moreover, the extension $\hat{f}$ can be taken so that if $(g_n)_{n\in\omega}$ is the conjugate of $f$ and $(\hat{g}_n)_{n\in\omega}$ is the conjugate of $\hat{f}$ then
    \begin{enumerate}
        \item $g_1(\alpha)=\hat{g}_1(\alpha)$ and
        \item $\langle \hat{g}_1(\alpha),\hat{g}_2(\alpha)\rangle$ is an induced subgraph of $\langle k_1(\alpha),k_2(\alpha)\rangle$ where $(k_n)_{n\in\omega}$ is the \L o\'s conjugate of $p(\bar{x})$.
    \end{enumerate} 
\end{cor}
\begin{proof}
    ($\Rightarrow$): Let $\hat{g}_1(\alpha)=g_1(\alpha)$ and $\hat{g}_2(\alpha)=k_2(\alpha)\cap[g_1(\alpha)]^2$. Then $\mathbb{G}_{\hat{f}}$ is an induced subgraph of $\mathbb{G}_L$ and we can use Corollary~\ref{indsublug}.
    
    ($\Leftarrow$): The existence of any distribution having a multiplicative refinement is equivalent to $p(\bar{x})$ having a multiplicative refinement by Fact~\ref{keislerfact}.
\end{proof}

\section{Graph-like Formula Sets and their Shapes}

It is frequently useful (as in the definition of $\mathrm{SOP}_2$-types) to restrict attention to certain subsets of the possible parameters for a $\varphi$-type by specifying where the representatives for the parameters can come from in each index structure $\mathcal{M}_\alpha$. We will generally choose these parameters to be such that the resulting $\varphi$-types are guaranteed to be graph-like. We begin with some definitions that allow us to talk about this situation more concretely.

\begin{defn}\label{varphiset}
	Given a formula $\varphi(\bar{x};\bar{y})$ in a language $\mathcal{L}$, a \emph{$\varphi$-set} is a nonempty collection $A\se M^n$ where $M$ is the universe of some structure $\mathcal{M}$ in the language $\mathcal{L}$ and the length of the tuple $\bar{y}$ is $n$.
\end{defn}

Note: unless explicitly stated, we will not require $\M$ to be an ultraproduct.

\begin{defn}\label{glvarphiset}
	If $\varphi(\bar{x},\bar{y})$ is a formula and $A$ is a $\varphi$-set in some structure $\mathcal{M}$ then we call $A$ a \emph{graph-like $\varphi$-set} if for all $B\se A$ such that $2\leq|B|<\aleph_0$ we have that the set $\{\varphi(\bar{x};\bar{a}):\bar{a}\in B\}$ is consistent if and only if for all $C\in[B]^2$ the set $\{\varphi(\bar{x};\bar{a}):\bar{a}\in C\}$ is consistent.
\end{defn}

We study what the collection of graphs arising from these types look like and what these collections can tell us about the possible range of types that can be created by using parameters from $A$.

\begin{defn}\label{shape}
	If $A$ is a graph-like $\varphi$-set coming from the structure $\mathcal{M}$ then \emph{the shape of $A$}, denoted $\sh(A)$, is defined to be the set of pairs $(\mathbb{G},\theta)$ where 
	\begin{enumerate}
	    \item there is an $m\in\omega\sm\{0\}$ so that $\mathbb{G}$ is a finite simple graph with vertices $\{1,2,\ldots,m\}$,
	    \item $\theta$ is a function $G\to A$,
	    \item for each $i\in G$ we have that $\M\vDash \exists \bar{x},\,\varphi(\bar{x};\theta(i))$, and
	    \item there is an edge between distinct $i$ and $j$ in $\mathbb{G}$ if and only if
	    \[\M\vDash\exists\bar{x},\,\varphi(\bar{x};\theta(i))\wedge\varphi(\bar{x};\theta(j)).\]
	\end{enumerate}
\end{defn}

\begin{defn}\label{Aalphatype}
	Let $(A_\alpha)_{\alpha<\lambda}$ be a sequence of $\varphi$-sets in the structures $(\mathcal{M}_\alpha)_{\alpha<\lambda}$. We say that a $\varphi$-type $p(\bar{x})$ of $\widehat{M}=\prod_{\alpha<\lambda}\M_\alpha/\D$ is an \emph{$(A_\alpha)$-type} when there is a sequence $(\bar{r}_\beta)_{\beta<\lambda}$ of representatives for the parameters in $p(\bar{x})$ such that for each index $\alpha<\lambda$ the set $\{\bar{r}_\beta(\alpha):\beta<\lambda\}$ is a subset of $A_\alpha$. That is, $p(\bar{x})$ is an $(A_\alpha)$-type precisely when the parameters for $p(\bar{x})$ come from the set $A_\alpha^\D$.
\end{defn}

The next result is a demonstration of the fact that $\sh(A_\alpha)$ (where $(A_\alpha)_{\alpha<\lambda}$ is a sequence of $\varphi$-sets) captures the notion of which graphs are induced subgraphs of the corresponding graph in the distribution graph sequence for the \L o\'s ultragraph of an $(A_\alpha)$-type. This will allow us to apply Corollary~\ref{indsublug} to distributions $f$ of the $(A_\alpha)$-type $p(\bar{x})$ whose distribution graph sequence comes from $\sh(A_\alpha)$ for each $\alpha<\lambda$. That is, all such distributions $f$ having multiplicative refinements will be equivalent to all $(A_\alpha)$-types in $\widehat\M$ being realized in $\widehat\M$. The remainder of this section will be dedicated to determining the form that such distributions must take and how to interpret these statements using the conjugate distribution and ultragraph.

\begin{lem}\label{shapeinducedlos}
    Let $(A_\alpha)_{\alpha<\lambda}$ be a sequence of graph-like $\varphi$-sets coming from structures $(\M_\alpha)_{\alpha<\lambda}$ all in the common language $\mathcal{L}$ and let $p(\bar{x})$ be an $(A_\alpha)$-type. Let $f$ be a distribution for $p(\bar{x})$ with representatives $(\bar{r}_\beta)_{\beta<\lambda}$ witnessing that $p(\bar{x})$ is an $(A_\alpha)$-type and let $(g_n)_{n\in\omega}$ and $(\mathbb{G}_\alpha)_{\alpha<\lambda}$ be the conjugate and distribution graph sequence of $f$. Fix $\alpha<\lambda$. Then there is an element $(\mathbb{G},\theta)$ of $\sh(A_\alpha)$ and a graph isomorphism $\psi\colon\mathbb{G}_\alpha\to\mathbb{G}$ with the property that $\theta(\psi(\beta))=\bar{r}_\beta(\alpha)$ for each $\beta\in\mathbb{G}_\alpha$ if and only if $\mathbb{G}_\alpha$ is an induced subgraph of $\langle k_1(\alpha),k_2(\alpha)\rangle$ where $(k_n)_{n\in\omega}$ is the \L o\'s conjugate for $p(\bar{x})$.
\end{lem}
\begin{proof}
    Let $\beta_0,\ldots,\beta_m$ enumerate the vertices of $\mathbb{G}_\alpha$ and define $G:=\{1,\ldots,m\}$,
    \[\theta(i):=\bar{r}_{\beta_i}(\alpha)\]
    for each $i\in G$, and $\psi(\beta_i):=i$. Let $\mathbb{G}$ be the unique graph with underlying set $G$ and such that $(\mathbb{G},\theta)\in\sh(A_\alpha)$. Then $\mathbb{G}_\alpha$ being an induced subgraph of $\langle k_1(\alpha),k_2(\alpha)\rangle$ is the same as saying that there is an edge between $\beta,\gamma$ in $\mathbb{G}_\alpha$ if and only if
    \[\M_\alpha\vDash \exists\bar{x},\,\varphi(\bar{x},\bar{r}_\beta(\alpha))\wedge\varphi(\bar{x},\bar{r}_\gamma(\alpha))\]
    which is equivalent to saying that $\psi(\beta)$ and $\psi(\gamma)$ are edge related in $\mathbb{G}$. Note that we did not have any choice (up to permuting the enumeration given for $\mathbb{G}_\alpha$) in picking $G$, $\theta$, and $\psi$.
\end{proof}

It will be useful in characterizing the distributions of $(A_\alpha)$-types to know which sequences of graphs $(\mathbb{G}_\alpha)_{\alpha<\lambda}$ occurring in $\sh(A_\alpha)$ can appear as the distribution graph sequence of an $(A_\alpha)$-type (see Theorem~\ref{realAalphatype}). For such a graph sequence, it is necessary that we can find an induced subgraph $\mathbb{H}$ of $\mathbb{G}:=\prod_{\alpha<\lambda}\mathbb{G}_\alpha/\D$ that is complete, has cardinality $\lambda$, and has maps $\iota_\alpha\colon G_\alpha\to \lambda$ such that the induced map $\iota: \mathbb{G}\to \lambda^\D$ maps $\mathbb{H}$ bijectively to $\lambda\se\lambda^\D$ (this $\mathbb{H}$ will be the image of $\eta_f$ where $f$ is the distribution under consideration). As the following lemma shows, this is enough.

\begin{lem}\label{shtodistgraphseq}
	Let $(A_\alpha)_{\alpha<\lambda}$ be a sequence of graph-like $\varphi$-sets coming from the structures $(\M_\alpha)_{\alpha<\lambda}$ all in the common language $\mathcal{L}$. Suppose that
	\begin{enumerate}
	    \item there is a set $B\in\D$ and a sequence $(\mathbb{G}_\alpha,\theta_\alpha)_{\alpha\in B}$ such that for all $\alpha\in B$ we have that $(\mathbb{G}_\alpha,\theta_\alpha)\in\sh(A_\alpha)$,
	    \item for each $\alpha\in B$ there is a subset $Y_\alpha$ of $\lambda$ with $|Y_\alpha|=|\mathbb{G}_\alpha|$ and for each $\beta\in\lambda$ the set $X_\beta:=\{\alpha\in B:\beta\in Y_\alpha\}$ is an element of $\D$, and
	    \item there are fixed bijections $\iota_\alpha\colon Y_\alpha\to \mathbb{G}_\alpha$ for each $\alpha\in B$ such that the sequence $(\iota_\alpha)_{\alpha\in B}$ satisfies 
	    \[\{\alpha\in B:\mathbb{G}_\alpha\vDash \iota_\alpha(\beta)\mathrel{E}\iota_\alpha(\gamma)\}\in\D\] 
	    for each $\{\beta,\gamma\}\in[\lambda]^2$.
	\end{enumerate} 
	Then there is a graph-like $\varphi$-type $p(\bar{x})$ in $\widehat\M=\M_\alpha^\D$ and a distribution $f$ for $p(\bar{x})$ for which
	\begin{enumerate}
        \setcounter{enumi}{3}
	    \item  $(\mathbb{H}_\alpha)_{\alpha<\lambda}$, the distribution graph sequence of $f$, is such that $\mathbb{H}_\alpha\cong\mathbb{G}_\alpha$ for all $\alpha\in B$,
	    \item $\ess(f)= B$, and
	    \item for all $\alpha\in B$ we have that $\mathbb{H}_\alpha$ is an induced subgraph of $\langle k_1(\alpha),k_2(\alpha)\rangle$ where $(k_n)_{n\in\omega}$ is the \L o\'s conjugate.
	\end{enumerate}
\end{lem}
\begin{conv}
    In order to make sense of (3) in the lemma statement and several similar statements made in the proof below, we say that $\mathbb{G}_\alpha\vDash \psi(\bar{a})$ is false if any of the elements of $\bar{a}$ are not in the underlying set of $\mathbb{G}_\alpha$ or are undefined values of a function. 
\end{conv}
\begin{proof}[Proof of Lemma~\ref{shtodistgraphseq}]
	We begin by defining the type $p(\bar{x})$ and the distribution $f$, and then check that these values of $p(\bar{x})$ and $f$ satisfy conditions $(4)$--$(6)$. For each $\beta\in\lambda$ fix a function $\bar{r}_\beta$ defined by $\alpha\mapsto \theta_\alpha(\iota_\alpha(\beta))$ when $Y_\alpha$ is defined and $\beta\in Y_\alpha$, and $\bar{r}_\beta(\alpha)$ is an arbitrary value of $A_\alpha$ for all other values of $\alpha$ in $\lambda$. Define a collection of formulas $p(\bar{x})$ by \[p(\bar{x}):=\{\varphi\left(\bar{x},\bar{r}_\beta/\D\right):\beta<\lambda\}.\] The collection $p(\bar{x})$ is a type because, for any finite $\Delta\se\lambda$ with $|\Delta|\geq 2$, we have that 
	\begin{align*}
		\llbracket\Delta\rrbracket^\D_{p(\bar{x})} &=
		\{\alpha\in\lambda:\forall \{\beta,\gamma\}\in[\Delta]^2,\,\M_\alpha\vDash\exists \bar{x},\,\varphi(\bar{x},\bar{r}_\beta(\alpha))\wedge\varphi(\bar{x},\bar{r}_\gamma(\alpha))\}\\
		&\supseteq \{\alpha\in B:\forall \{\beta,\gamma\}\in[\Delta]^2,\, \mathbb{G}_\alpha\vDash \iota_\alpha(\beta)\mathrel{E}\iota_\alpha(\gamma)\}\\
		&=\bigcap_{\{\beta,\gamma\}\in[\Delta]^2}\{\alpha\in B:\mathbb{G}_\alpha\vDash \iota_\alpha(\beta)\mathrel{E}\iota_\alpha(\gamma)\}\\
		&\in\D
	\end{align*}
	by condition~$(3)$ of the lemma statement as well as the fact that $\Delta$ (and hence $[\Delta]^2$) is finite. Moreover, $p(\bar{x})$ is an $(A_\alpha)$-type and therefore a graph-like type.
	
	We use Theorem~\ref{extendingetalequalsdistribution} to find a distribution $\hat{f}$ of $p(\bar{x})$ and then show that there is a refinement $f$ of $\hat{f}$ satisfying conditions (4)--(6). Define 
	\[H_\alpha:=\pw{Y_\alpha; &\alpha\in B\\ \{0\}; &\text{otherwise}}\]
	and consider the hyperfinite set $H:=\prod_{\alpha<\lambda}H_\alpha/\D$. In particular, $H_\alpha\se k_1(\alpha)$, so $H$ is a hyperfinite subset of $G_L$. Furthermore, condition (2) guarantees that $\eta_L[\lambda]\se H$. By Theorem~\ref{extendingetalequalsdistribution}, there is a graph-like distribution $\hat{f}$ such that $i_{\hat f}[\mathbb{G}_{\hat f}]$ is the subgraph of $\mathbb{G}_L$ induced by $H$ and such that the distribution graph sequence $(\hat{\mathbb{G}}_\alpha)_{\alpha<\lambda}$ of $\hat{f}$ consists of the subgraphs of $\langle k_1(\alpha),k_2(\alpha)\rangle$ induced by $H_\alpha$. We claim that $\iota_\alpha\colon\hat{\mathbb{G}}_\alpha\to\mathbb{G}_\alpha$ is a graph isomorphism for all $\alpha\in B$. Let $\beta,\gamma$ be distinct elements of $\hat{\mathbb{G}}_\alpha$. Then
	\begin{align*}
	    \beta \mathrel{E}^{\hat{\mathbb{G}}_\alpha}\gamma&\iff \{\beta,\gamma\}\in k_2(\alpha)\\
	    &\iff \M_\alpha\vDash \exists \bar{x},\, \varphi(\bar{x},\bar{r}_\beta(\alpha))\wedge\varphi(\bar{x},\bar{r}_\gamma(\alpha))\\
	    &\iff \M_\alpha\vDash \exists \bar{x},\, \varphi(\bar{x},\theta_\alpha(\iota_\alpha(\beta)))\wedge\varphi(\bar{x},\theta_\alpha(\iota_\alpha(\gamma)))\\
	    &\iff \iota_\alpha(\beta) \mathrel{E}^{\mathbb{G}_\alpha}\iota_\alpha(\gamma).
	\end{align*}
	This is enough to show that $\iota_\alpha$ is a graph isomorphism as $\iota_\alpha$ is defined to be a bijection between the underlying sets of $\hat{\mathbb{G}}_\alpha$ and $\mathbb{G}_\alpha$. Moreover, by Corollary~\ref{indsublug}, the type $p(\bar{x})$ is realized in $\widehat\M$ if and only if $\hat{f}$ has a multiplicative refinement.
	
	As suggested by the above results, the desired refinement $f$ of $\hat{f}$ is the same distribution but with the information carried by $\mathbb{G}_\alpha$ removed when $\alpha\in\lambda\sm B$. This can be achieved by setting $f(\Delta)=\hat{f}(\Delta)\cap B$ for all $\Delta\in\P_\omega(\lambda)$. As $B\in\D$, the ultragraph of $f$ will be isomorphic to the ultragraph of $\hat{f}$. This isomorphism sends $\eta_f[\lambda]$ to $\eta_{\hat{f}}[\lambda]$ so  Corollary~\ref{multiffint} implies that $f$ has a multiplicative refinement if and only if $\hat{f}$ has a multiplicative refinement.
\end{proof}

For most of our applications we will not need to keep careful track of the function $\theta$ associated with an element of $\sh(A)$, so we focus instead on the isomorphism types of the graphs $\mathbb{G}$ that appear in $\sh(A)$.

\begin{defn}
    If $A$ is a graph-like $\varphi$-set coming from the structure $\mathcal{M}$ then the \emph{isomorphism shape of $A$}, denoted $\ish(A)$, is the collection 
    \[\{\mathrm{isoType}(\mathbb{G}):(\mathbb{G},\theta)\in\sh(A)\}\]
    of graph isomorphism types that arise from $\sh(A)$, where $\mathrm{isoType}(\mathbb{G})$ is the isomorphism type of the graph $\mathbb{G}$.
\end{defn}

\begin{lem}\label{graphseqfromshtype}
    Let $(A_\alpha)_{\alpha<\lambda}$ be a sequence of graph-like $\varphi$-sets coming from the structures $(\M_\alpha)_{\alpha<\lambda}$ and let $f\colon\P_\omega(\lambda)\to\D$ be a graph-like function such that
    \begin{enumerate}
        \item $f$ is monotone and satisfies condition $(2)$ in the definition of a distribution~(Definition~\ref{distribution}) and
        \item the graph sequence of $f$, given by $(\mathbb{G}_\alpha)_{\alpha<\lambda}$, is such that the isomorphism type of $\mathbb{G}_\alpha$ is an element of $\ish(A_\alpha)$ for all $\alpha\in\ess(f)$.
    \end{enumerate}
    Then there is an $(A_\alpha)$-type $p(\bar{x})$ of $\widehat{\M}:=\prod_{\alpha<\lambda}\M_\alpha/\D$ and representatives $(\bar{r}_\beta)_{\beta<\lambda}$ such that $f$ has a multiplicative refinement if and only if $p(\bar{x})$ is realized.
\end{lem}
\begin{proof}
    We proceed by checking that $f$ gives us enough information to satisfy conditions (1)--(3) of Lemma~\ref{shtodistgraphseq}.
    
    \ref{shtodistgraphseq}(1): $\ess(f)\in\D$ and, as $\mathbb{G}_\alpha$ has its isomorphism type in $\ish(A_\alpha)$, we may choose some $(\mathbb{G}'_\alpha,\theta_\alpha)\in\sh(A_\alpha)$ and an isomorphism $\iota_\alpha\colon\mathbb{G}_\alpha\to\mathbb{G}_\alpha'$ for each $\alpha\in\ess(f)$.
    
    \ref{shtodistgraphseq}(2): For each $\alpha\in\ess(f)$, let $Y_\alpha=G_\alpha$. Then
    \[X_\beta=\{\alpha\in\lambda:\beta\in G_\alpha\}=f(\{\beta\})\in\D\]
    by Equation~\eqref{disteqconj}.
    
    \ref{shtodistgraphseq}(3): For each $\{\beta,\gamma\}\in[\lambda]^2$ we have that
    \[\{\alpha\in \ess(f):\mathbb{G}'_\alpha\vDash \iota_{\alpha}(\beta) \mathrel{E} \iota_\alpha(\gamma)\}=f(\{\beta,\gamma\})\in\D\]
    by the same reasoning as above.
    
    We then have that there is an $(A_\alpha)$-type $p(\bar{x})$ for which $(\mathbb{G}_\alpha)_{\alpha<\lambda}$ is the distribution graph sequence of a distribution $f'$ that has a multiplicative refinement if and only if $p(\bar{x})$ is realized. Since graph-like functions are determined by their graph sequence, we must have that $f$ has a multiplicative refinement if and only if $p(\bar{x})$ is realized.
\end{proof}

Sometimes we will want to talk about induced subgraphs of graphs $\mathbb{G}$ arising in $\sh(A)$ but will not care about the specific mapping $\theta$ that tell us what element of $A$ each vertex corresponds to. In order to do this, we introduce some terminology to speak about induced subgraphs of isomorphism types.

\begin{conv}
    Given an isomorphism type $\mathrm{isoType}(\mathbb{G})$, we will say that a graph $\mathbb{H}$ (or isomorphism type $\mathrm{isoType}(\mathbb{H})$) is an induced subgraph of $\mathrm{isoType}(\mathbb{G})$ and write $\mathbb{H}\leq\mathrm{isoType}(\mathbb{G})$ if there is an injective graph homomorphism $\psi\colon\mathbb{H}\to\mathbb{G}$ that also preserves non-edges.
\end{conv}

We will characterize the isomorphism shape not by which isomorphism types belong to the isomorphism shape, but by the minimal graphs that cannot appear as induced subgraphs of elements of the isomorphism shape. This follows, roughly speaking, from the fact that induced subgraphs of $\mathbb{G}$ with $(\mathbb{G},\theta)\in\sh(A)$ are easily obtained by taking a subset of the vertices in $\mathbb{G}$ and the corresponding restriction to the function $\theta$ and from the well-used fact in graph combinatorics that classes $\mathcal{C}$ of graphs that are closed under taking subgraphs are characterized by the minimal graphs not occurring in $\mathcal{C}$. 

\begin{lem}\label{shclosedindsubg}
	Suppose that $A$ is a graph-like $\varphi$-set in the structure $\M$ and $\mathbb{H}$ a graph such that $\mathrm{isoType}(\mathbb{H})$ is not an element of $\ish(A)$. If $\mathrm{isoType}(\mathbb{G})\in\ish(A)$ then $\mathbb{H}\nleq\mathrm{isoType}(\mathbb{G})$.
\end{lem}
\begin{proof}
	Suppose that there were an element $(\mathbb{G},\theta)\in\sh(A)$ with $\psi\colon\mathbb{H}\to\mathbb{G}$ an injective homomorphism preserving non-edges. Let $C$ be $\psi[H]$. We may assume without loss of generality that $C$ is an initial segment of $\{1,2,\ldots,m\}=G$. Define $\theta\restriction C$ be the restriction of $\theta$ to $C$ and $\mathbb{G}_C$ to be the subgraph of $\mathbb{G}$ induced by $C$. Then $(\mathbb{G}_C,\theta\restriction C)\in\sh(A)$ with $\mathbb{G}_C\cong\mathbb{H}$, a contradiction.
\end{proof}

\begin{thm}\label{charshapes}
	Suppose that $A$ is a graph-like $\varphi$-set in the structure $\M$ and that $\Xi$ is a collection of isomorphism types of finite graphs such that both
	\begin{enumerate}
		\item $\Xi\cap\ish(A)=\emptyset$ and 
		\item whenever the $\mathrm{isoType}\mathbb{H}$ is not an element of $\ish(A)$ there is some $\xi\in\Xi$ for which $\xi\leq\mathrm{isoType}(\mathbb{H}$).
	\end{enumerate}
	Then \[\ish(A)=\{\mathrm{isoType}(\mathbb{G}):\forall \xi\in\Xi,\, \xi\nleq\mathrm{isoType}(\mathbb{G})\}.\]
\end{thm}
\begin{proof}
	Suppose that $\mathrm{isoType}(\mathbb{H})$ is not in $\ish(A)$. Then, by assumption~$(2)$, there exists $\xi$ in $\Xi$ with $\xi\leq\mathrm{isoType}(\mathbb{H})$. If $\mathrm{isoType}(\mathbb{G})\in\ish(A)$, then by assumption~$(1)$ and Lemma~\ref{shclosedindsubg} it cannot be that any element of $\Xi$ is an induced subgraph of $\mathrm{isoType}(\mathbb{G})$.
\end{proof}

Knowing the set of $\leq$-minimal isomorphism types of graphs that cannot appear in $\ish(A)$ (call this set $\Gamma$) is enough to determine all of $\ish(A)$. We will extract some extra information that is implicit in this characterization of isomorphism shapes in order to characterize the functions that can appear as distributions for $(A_\alpha)$-types. We know that graphs $\mathbb{H}$ whose isomorphism types appear in $\Gamma$ cannot be \emph{induced} subgraphs of graphs arising in $\sh(A)$, but $\mathbb{H}$ may still appear as a subgraph of graphs arising in $\sh(A)$. More precisely, $\mathrm{isoType}(\mathbb{G})$ is an element of $\ish(A)$ exactly in the case that whenever a graph $\mathbb{H}$, representing an element of $\Gamma$, is a subgraph of $\mathbb{G}$, we must have that $\mathbb{H}$ is not an \emph{induced} subgraph of $\mathbb{G}$ (that is, $\mathbb{G}$ must contain edges that do not occur in $\mathbb{H}$).

\begin{defn}\label{necsets}
	Let $S$ be a set of graphs (or graph isomorphism types) and $\mathbb{H}$ a graph such that for each $\mathbb{G}\in S$ we have that $\mathbb{H}\nleq\mathrm{isoType}(\mathbb{G})$. A collection $B\se[H]^2\sm E^\mathbb{H}$ is called \emph{necessary for $\mathbb{H}$ in $S$} if whenever $\mathbb{G}\in S$ and there is an injective graph homomorphism $\psi\colon\mathbb{H}\to\mathbb{G}$ then some element of 
	\[\{\{\psi(a),\psi(b)\}:(a,b)\in B\}\]
	appears in the edge relation for $\mathbb{G}$.
\end{defn}

\begin{rmrk}
    The definition of necessary sets for $\mathbb{H}$ in $S$ depends only on the isomorphism types of graphs that occur in $S$. Because of this, we sometimes speak of necessary sets for $\mathbb{H}$ in a collection of isomorphism types of graphs.
\end{rmrk}

We will typically be interested in necessary sets that are minimal in cardinality or minimal under the $\se$ relation. Our applications are in the particular case that $S=\ish(A)$ for some graph-like $\varphi$-set $A$ and $\mathbb{H}$ is a minimal graph not appearing as an induced subgraph of any graph arising in $\sh(A)$. In this situation, we will be able to use Theorem~\ref{charshapes} to gain some information about the structure of the conjugate for a distribution $f$ of an $(A_\alpha)$-type, which can then be translated into information about the distribution $f$ and will allow us to recognize what sorts of distributions are allowed for $(A_\alpha)$-types.

\begin{lem}\label{necviadist}
	Suppose that $p(\bar{x})$ is a $\lambda$-type in $\widehat{\M}=\prod_{\alpha<\lambda}\M_\alpha/\D$ and that $f$ is a graph-like distribution for $p(\bar{x})$ with distribution graph sequence $(\mathbb{G}_\alpha)_{\alpha<\lambda}$. Let $\mathbb{H}$ be a graph that is not an induced subgraph of any of the $\mathrm{isoType}(\mathbb{G}_\alpha)$ and let $B_\mathbb{H}\se[H]^2\sm E^\mathbb{H}$. Then $f$ satisfies 
	\[\bigcap_{\{a,b\}\in E^{\mathbb{H}}}f(\{x(a),x(b)\})\se\bigcup_{\{c,d\}\in B_{\mathbb{H}}}f(\{x(c),x(d)\})\] for all injective functions $x\colon H\to\lambda$ if and only if the set $B_\mathbb{H}$ is necessary for $\mathbb{H}$ in $\{\mathbb{G}_\alpha:{\alpha\in\ess(f)}\}$.
\end{lem}
\begin{proof}
	($\Leftarrow$): Let $x\colon H\to \lambda$ be an injective function. We wish to show that if 
	\begin{equation}\label{necback}
	\alpha\in\bigcap_{\{a,b\}\in E^\mathbb{H}} f(\{x(a),x(b)\})\,\text{ then }\,\alpha\in\bigcup_{\{c,d\}\in B_\mathbb{H}}f(\{x(c),x(d)\}).
	\end{equation}
	By the definition of the distribution graph sequence, saying $\alpha$ is an element of $f(\{x(a),x(b)\})$ is equivalent to the statement that $x(a)$ and $x(b)$ are edge related in $\mathbb{G}_\alpha$. We then have that \eqref{necback} is equivalent to \[\mathbb{G}_\alpha\vDash\bigwedge_{\{a,b\}\in E^\mathbb{H}}x(a)\mathrel{E} x(b)\rightarrow \bigvee_{\{c,d\}\in B_\mathbb{H}} x(c)\mathrel{E} x(d)\] for all $\alpha<\lambda$. This says that if $x$ can be thought of as an injective graph homomorphism $\mathbb{H}\to\mathbb{G}_\alpha$, then one of the edges in $B_\mathbb{H}$ must also appear in $\mathbb{G}_\alpha$. This follows from the definition of $B_\mathbb{H}$ being necessary for $\mathbb{H}$ in the collection of $\mathbb{G}_\alpha$.
	
	$(\Rightarrow)$: Suppose that $x:\mathbb{H}\to\mathbb{G}_\alpha$ is an injective graph homomorphism. As $G_\alpha\se \lambda$ and $x$ is a homomorphism, we have that
	\[\alpha\in\bigcap_{\{a,b\}\in E^{\mathbb{H}}}f(\{x(a),x(b)\})=\big\{\alpha\in\lambda:\forall\{a,b\}\in E^{\mathbb{H}},\, \{x(a),x(b)\}\in E^{\mathbb{G}_\alpha}\big\}.\]
	By assumption we then have that
	\[\alpha\in\bigcup_{\{c,d\}\in B_\mathbb{H}}f(\{x(c),x(d)\})=\big\{\alpha\in\lambda:\exists\{c,d\}\in B_{\mathbb{H}},\, \{x(c),x(d)\}\in E^{\mathbb{G}_\alpha}\big\}.\]
	That is, there is some pair $\{c,d\}\in B_\mathbb{H}$ for which $x(c)$ and $x(d)$ are edge related in $\mathbb{G}_\alpha$, as required.
\end{proof}

\begin{thm}\label{realAalphatype}
	Suppose that the sequence $(A_\alpha)_{\alpha<\lambda}$ is a sequence of graph-like $\varphi$-sets in the structures $(\M_\alpha)_{\alpha<\lambda}$ having a common language and such that for all $\alpha,\delta<\lambda$ we have that $\ish(A_\alpha)=\ish(A_\delta)$. Let $\Gamma$ be the collection of $\leq$-minimal isomorphsim types of graphs not contained in $\ish(A_0)$, and for each graph $\mathbb{H}\in\Gamma$, let $B_{\mathbb{H}}$ be a necessary set for $\mathbb{H}$ in $\ish(A_0)$. Then every $\lambda$-type that is also an $(A_\alpha)$-type is realized in $\widehat{\M}$ if and only if every graph-like function $f\colon\P_\omega(\lambda)\to\D$ satisfying
	\begin{equation}\label{necdisteqn}
		\bigcap_{\{a,b\}\in E^{\mathbb{H}}}f(\{x(a),x(b)\})\se\bigcup_{\{c,d\}\in B_{\mathbb{H}}}f(\{x(c),x(d)\})
	\end{equation}
	for all $\mathbb{H}\in\Gamma$ and all injective functions $x\colon H\to\lambda$ has a multiplicative refinement.
\end{thm}
\begin{proof}
	($\Rightarrow$): Let $f$ be a graph-like function satisfying \eqref{necdisteqn}, let $g_1,g_2$ belong to the conjugate of $f$, and fix $\mathbb{H}\in\Gamma$. By Lemma~\ref{necviadist}, the $B_\mathbb{H}$ are necessary for $\mathbb{H}$ in the graph sequence of $f$. Thus, if $\mathbb{H}\in\Gamma$ and $B_\mathbb{H}$ is non-empty, we must have that $\mathbb{H}$ cannot be an induced subgraph of $\mathbb{G}_\alpha:=\langle g_1(\alpha),g_2(\alpha)\rangle$ for all $\alpha<\lambda$. Even if $B_\mathbb{H}$ is empty, we then have that $\mathbb{H}$ cannot even be a subgraph of any of the $\mathbb{G}_\alpha$ since in this case the right-hand side of \eqref{necdisteqn} is $\emptyset$, so $\mathbb{H}\nleq\mathbb{G}_\alpha$ for all $\mathbb{H}\in\Gamma$ and all $\alpha<\lambda$. Then the graphs $\mathbb{G}_\alpha$ are elements of $\sh(A_0)$ by Theorem~\ref{charshapes}. By Lemma~\ref{graphseqfromshtype}, there is an $(A_\alpha)$-type $p(\bar{x})$ that is realized if and only if $f$ has a multiplicative refinement, and, by assumption, all $(A_\alpha)$-types are realized in $\widehat{\M}$.

	($\Leftarrow$): Let $p(\bar{x})$ be an $(A_\alpha)$-type of $\widehat\M$ having cardinality $\lambda$ and choose a set of representatives $(r_\beta)_{\beta<\eta}$ for the parameters of $p(\bar{x})$ from $\prod_{\alpha<\lambda}A_\alpha$ and notice that these representatives witness $p(\bar{x})$ being graph-like. Let $f$ be a distribution for $p(\bar{x})$ using the same representatives. By Corollary~\ref{typelikeexts} it must be that $f$ is a refinement of graph-like distribution $\hat{f}$ with the property that $(\hat{g}_n)_{n\in\omega}$, the conjugate of $\hat{f}$, is such that $\hat{g}_1=g_1$ (where $g_1$ comes from the conjugate of $f$) and the graphs $\hat{\mathbb{G}}_\alpha:=\langle\hat{g}_1(\alpha),\hat{g}_2(\alpha)\rangle$ are induced subgraphs of $\langle k_1(\alpha),k_2(\alpha)\rangle$ where $(k_n)_{n\in\omega}$ is the \L o\'s conjugate of $p(\bar{x})$. In particular, for all $\alpha\in\ess(f)$ we know that $\hat{\mathbb{G}}_\alpha$ is isomorphic to a graph arising in $\sh(A_\alpha)$ by Lemma~\ref{shapeinducedlos}. Lemma~\ref{necviadist} guarantees that $\hat{f}$ satisfies \eqref{necdisteqn}, so $\hat{f}$ has a multiplicative refinement by assumption. Thus, by Fact~\ref{keislerfact}, the type $p(\bar{x})$ is realized.
\end{proof}

\begin{thm}\label{graphcompletenessAtypes}
    Let $\D$ be a regular ultrafilter on the infinite cardinal $\lambda$, $A$ be a graph-like $\varphi$-set in the structure $\M$, and $\Gamma$ be the set of $\leq$-minimal graph isomorphism types not occurring in $\ish(A)$. Then the following are equivalent:
    \begin{enumerate}
        \item Every $(A)$-type of $\widehat\M:=\M^\D$ with cardinality $\leq\lambda$ is realized.
        \item If $(\mathbb{G}_\alpha)_{\alpha<\lambda}$ is a sequence of (possibly infinite) graphs with the property that for every $\gamma\in\Gamma$ and $\alpha <\lambda$ we have that $\gamma\nleq\mathrm{isoType}(\mathbb{G}_\alpha)$ (i.e.\ every finite induced subgraph of $\mathrm{isoType}(\mathbb{G}_\alpha)$ is an element of $\ish(A)$), then the ultraproduct $\mathbb{G}:=\mathbb{G}_\alpha^\D$ has the property that every complete $H\se\mathbb{G}$ with $|H|=\lambda$ extends to an internal complete subgraph of $\mathbb{G}$.
    \end{enumerate}
\end{thm}
\begin{proof}
    $(1)\Rightarrow(2)$: 
    We proceed by constructing an $(A)$-type $p(\bar{x})$ in $\widehat\M$ by taking advantage of the fact that the isomorphism types of the factors of $\mathbb{G}$ belong to $\ish(A)$. We then show that a multiplicative distribution for $p(\bar{x})$ in $\widehat\M$ gives instructions for extending $H$ to an internal complete subgraph of $\mathbb{G}$.
    
    For each $\alpha<\lambda$ choose a pair $(\mathbb{G}'_\alpha,\theta_\alpha)$ in $\sh(A)$ with a graph isomorphism $\psi_\alpha\colon\mathbb{G}_\alpha\to\mathbb{G}'_\alpha$. Enumerate the elements of $H$ as $(h_\beta)_{\beta<\lambda}$ and pick representatives $x_\beta\in\prod_{\alpha<\lambda}\mathbb{G}_\alpha$ for each $h_\beta$. For each $\alpha,\beta<\lambda$, define
    \[\bar{r}_\beta(\alpha)=\theta_\alpha(\psi_\alpha(x_\beta(\alpha)))\]
    and consider the collection of formulas $p(\bar{x})=\{\varphi(\bar{x},\bar{r}_\beta/\D):\beta<\lambda\}$. We claim that $p(\bar{x})$ is an $(A)$-type of $\widehat{\M}$. It will be enough to check that $p(\bar{x})$ is a type as the image of $\theta_\alpha$ is always a subset of $A$.
    
    If $\Delta\in\P_\omega(\lambda)$, then, because $A$ is graph-like, we have that
    \begin{align*}
    \llbracket\Delta\rrbracket^\D_{p(\bar{x})} &=
		\{\alpha\in\lambda:\forall \{\beta,\gamma\}\in[\Delta]^2,\,\M_\alpha\vDash\exists \bar{x},\,\varphi(\bar{x},\bar{r}_\beta(\alpha))\wedge\varphi(\bar{x},\bar{r}_\gamma(\alpha))\}\\
	&=\{\alpha\in\lambda: \forall \{\beta,\gamma\}\in[\Delta]^2,\, \mathbb{G}'_\alpha\vDash \psi_\alpha(x_\beta(\alpha))\mathrel{E} \psi_\alpha(x_\gamma(\alpha))\}\\
	&=\{\alpha\in\lambda: \forall \{\beta,\gamma\}\in[\Delta]^2,\, \mathbb{G}_\alpha\vDash x_\beta(\alpha)\mathrel{E} x_\gamma(\alpha)\}\\
	&=\bigcap_{\{\beta,\gamma\}\in[\Delta]^2}\{\alpha\in\lambda:\mathbb{G}_\alpha\vDash x_\beta(\alpha)\mathrel{E} x_\gamma(\alpha)\}
	\end{align*}
	Then, as $x_\beta/\D=h_\beta$ and $x_\gamma/\D=h_\gamma$ are neighbors in $\mathbb{G}$ and $[\Delta]^2$ is finite, we have that this last set is an element of $\D$ by \L o\'s theorem and the fact that ultrafilters are closed under finite intersection.
	
	By assumption, the type $p(\bar{x})$ is realized in $\widehat\M$ so, by Fact~\ref{keislerfact}, we may choose a multiplicative distribution $f$ for $p(\bar{x})$ and we may assume that $\ess(f)=\lambda$ (if not, let $f'(\Delta)=f(\Delta)$ if $\Delta\neq\{0\}$ and set $f'(\{0\})=f(\{0\})\cup(\lambda\sm\ess(f))$. Then $f'$ is a multiplicative distribution of $p(\bar{x})$ and $\ess(f')=\lambda$). We will use the information contained in $f$ to extend $H$.
	
	Let $(g_n)_{n\in\omega}$ be the conjugate of $f$. For each $\alpha<\lambda$, define $\mathbb{H}_\alpha$ to be the subgraph of $\mathbb{G}_\alpha$ induced by the set 
	\[\psi_\alpha^{-1}[\{\psi_\alpha(x_\beta(\alpha)):\beta\in g_1(\alpha)\}]=\{x_\beta(\alpha):\beta\in g_1(\alpha)\}.\]
	Because $\psi_\alpha$ and $\psi^{-1}_\alpha$ are graph isomorphisms and because the graph $\langle g_1(\alpha),g_2(\alpha)\rangle$ is complete by Lemma~\ref{multviaseq}, we must have that $\mathbb{H}_\alpha$ is complete for all $\alpha<\lambda$. Thus $\widehat{\mathbb{H}}:=\prod_{\alpha<\lambda}\mathbb{H}_\alpha/\D$ is complete by \L o\'s theorem and is internal by construction. Furthermore, $\widehat{\mathbb{H}}$ extends $H$ as $x_\beta(\alpha)\in \mathbb{H}_\alpha$ whenever $\alpha\in f(\{\beta\})$ and $f(\{\beta\})\in\D$.
    
    $(2)\Rightarrow(1)$: Let $p(\bar{x})$ be an $(A)$-type of $\widehat{\M}$ of cardinality $\lambda$. We then have that $p(\bar{x})$ is a graph-like type, witnessed by representations of the parameters coming from $A^\lambda$. The type $p(\bar{x})$ is realized in $\widehat\M$ if and only if the \L o\'s ultragraph $\mathbb{G}_L$ has the property that the complete subgraph induced by $\eta_L[\lambda]$ extends to an internal complete subgraph of $\mathbb{G}_L$ (see Theorem~\ref{realiffint}). 
\end{proof}

\section{Applications to the $\mathrm{SOP}$ Hierarchy}
\subsection{$\mathrm{SOP}_2$ and Comparability in Trees}

We know from the definition of $\mathrm{SOP}_2$ (Definition~\ref{sop2type}) that whenever $A$ is a $\varphi$-set whose elements form an $\mathrm{SOP}_2$-tree within the structure $\M$ then any $B\se A$ creates a consistent collection of formula 
\[p_B(\bar{x})=\{\varphi(\bar{x},\bar{b}):\bar{b}\in B\}\] 
if and only if $B$ is a subset of a branch of the $\mathrm{SOP}_2$-tree whose elements come from $A$. In other words, $p_B(\bar{x})$ is consistent if and only if every pair $\bar{b},\bar{b}'\in B$ is comparable in the $\mathrm{SOP}_2$-tree (that is, $B$ is a linear sub-order of the tree). We thus have that the $\varphi$-set $A$ is graph-like and the shape of $A$ is the collection of all finite graphs $\mathbb{G}$ that can be obtained from taking a finite (multi-)subset of the full binary tree. There are edges between the two distinct vertices $b$ $b'$ in $\mathbb{G}$ if and only if $b$ and $b'$ are comparable in the full binary tree. Fortunately, most of the combinatorial work of determining the isomorphism types of such graphs has already been done by \cite{wolk}.

\begin{defn}\label{sop2witness}
    We say that a tuple $(\varphi, A)$ \emph{witnesses $\mathrm{SOP}_2$ in a structure $\M$} if there is an $\mathrm{SOP}_2$-tree 
    \[(\varphi(\bar{x};\bar{y}),(\bar{a}_\eta)_{\eta\in 2^{<\omega}})\]
    in $\M$ such that $A=\{a_\eta:\eta\in 2^{<\omega}\}$.
\end{defn}

\begin{lem}\label{sop2shape}
	Suppose that $\varphi$ is a formula and $A$ is a $\varphi$-set such that the pair $(\varphi,A)$ witnesses $\mathrm{SOP}_2$ in the structure $\M$. Then $A$ is a graph-like $\varphi$-set and the set $\Gamma$ of $\leq$-minimal isomorphism types not in $\ish(A)$ are the two isomorphism types:
	\[C_4:=\vcenter{\hbox{\begin{tikzpicture}
		\tikzset{enclosed/.style={draw, circle, inner sep=0pt, minimum size=.1cm, fill=black}}
		 
		%Vertices
		\node[enclosed] (1) at (0.75,.75) {};
		\node[enclosed] (2) at (1.5,.75) {};
		\node[enclosed] (3) at (0.75,0) {};
		\node[enclosed] (4) at (1.5,0) {};
		
		%Edges
		\draw (1) -- (2);
		\draw (2) -- (4);
		\draw (1) -- (3);
		\draw (3) -- (4);
	\end{tikzpicture}}}\quad
	\text{ and }\quad
	\ell_4:=\vcenter{\hbox{
	\begin{tikzpicture}
		\tikzset{enclosed/.style={draw, circle, inner sep=0pt, minimum size=.1cm, fill=black}}
		
		%Vertices
		\node[enclosed] (5) at (2.5,0) {};
		\node[enclosed] (6) at (3.25,0) {};
		\node[enclosed] (7) at (3.25,.75) {};
		\node[enclosed] (8) at (2.5,.75) {};
		
		%Edges
		\draw (5) -- (6);
		\draw (6) -- (7);
		\draw (7) -- (8);
	\end{tikzpicture}}}\, .\]
\end{lem}
\begin{proof}
	Because $(\varphi, A)$ witnesses $\mathrm{SOP}_2$ in $\M$, there is a surjection $h:2^{<\omega}\to A$ with the property that a subset $B$ of $2^{<\omega}$ is linearly ordered if and only if the collection of formulas
	\[p_B(\bar{x})=\{\varphi(\bar{x},h(s)): s\in B\}\]
	is consistent in $\M$. If $s\in 2^{<\omega}$, we will denote $h(s)$ by $\bar{a}_s$. Let $\Delta\in\P_\omega(2^{<\omega})$. Then the set of formula $\{\varphi(\bar{x},\bar{a}_s):s\in\Delta\}$ is consistent (equivalently, realized) in $\M$ if and only if the set $\Delta$ is linearly ordered as a suborder of $2^{<\omega}$ if and only if the elements of $\Delta$ are pair-wise comparable in the ordering on $2^{<\omega}$. Thus the elements $(\mathbb{G},\theta)$ of $\sh(A)$ where $\theta\colon G \to A$ is injective are such that $\mathbb{G}$ is the graph of the comparability relation for a finite suborder of $2^{<\omega}$ (equivalently, $\mathbb{G}$ is the graph of the comparability relation for a finite forest of subtrees of $2^{<\omega}$; i.e.\ the connected components of $\mathbb{G}$ correspond to finite subtrees of $2^{<\omega}$). Every finite collection of finite trees can be embedded as a suborder of $2^{<\omega}$, so the shape of $A$ contains every possible disjoint union of finitely many comparability graphs arising from finite trees. We claim that no other isomorphism types appear in $\ish(A)$. We postpone the proof of the claim to study the structure of these isomorphism types in more detail.
	
	By \cite{wolk} a finite connected graph $\mathbb{G}=(G,E)$ is the comparability graph of a tree if and only if for all distinct $x_0,x_1,x_2,x_3\in G$ the formula \[x_0\mathrel{E}x_1\mathrel{E}x_2\mathrel{E}x_3\rightarrow (x_0\mathrel{E} x_2)\vee (x_1\mathrel{E} x_3)\] holds (called \emph{the diagonal property}). Clearly, every graph $\mathbb{G}$ with $|G|\leq 3$ satisfies the diagonal property. Suppose that a graph with vertices $\{0,1,2,3\}$ does not satisfy the diagonal property. We may suppose that $0\mathrel{E}1\mathrel{E} 2\mathrel{E} 3$. If these are the only edges, the graph is $\ell_4$ and clearly not diagonal. Now there are only 3 possible edges that can be added, $\{0,3\},\{1,3\},$ or $\{0,2\}$. If just $\{0,3\}$ is added, we have a graph isomorphic to $C_4$, which is not diagonal. We wish to show that the remaining possibilities are all diagonal. 
	The complete graph on 4 elements is the comparability graph of a tree that is a 4 element linear order. Up to isomorphism, the remaining possibilities are the comparability graphs of the two trees 
	\[
	\vcenter
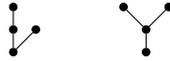
{\hbox{\begin{tikzpicture}
		\tikzset{enclosed/.style={draw, circle, inner sep=0pt, minimum size=.1cm, fill=black}}
		 
		%Vertices
		\node[enclosed] (1) at (0,0) {};
		\node[enclosed] (2) at (0,0.3) {};
		\node[enclosed] (3) at (0,0.6) {};
		\node[enclosed] (4) at (.3,.3) {};
		
		%Edges
		\draw (1) -- (2);
		\draw (2) -- (3);
		\draw (1) -- (4);
	\end{tikzpicture}}}\gap
	\vcenter{\hbox{\begin{tikzpicture}
		\tikzset{enclosed/.style={draw, circle, inner sep=0pt, minimum size=.1cm, fill=black}}
		 
		%Vertices
		\node[enclosed] (1) at (0,0) {};
		\node[enclosed] (2) at (0,0.3) {};
		\node[enclosed] (3) at (-.3,.6) {};
		\node[enclosed] (4) at (.3,.6) {};
		
		%Edges
		\draw (1) -- (2);
		\draw (2) -- (3);
		\draw (2) -- (4);
	\end{tikzpicture}}}
	\]
	and so are diagonal.
	
	If $\mathbb{G}$ is a non-diagonal graph with greater than $4$ vertices, there must be a 4-vertex induced subgraph of $\mathbb{G}$ that is not diagonal (pick the four vertices to be a set of vertices witnessing the fact that the graph is not diagonal). Thus a graph $\mathbb{G}$ is diagonal if and only if $\mathbb{G}$ is the disjoint union of a collection of comparability graphs of trees if and only if $C_4\nleq\mathbb{G}$ and $\ell_4\nleq\mathbb{G}$.
	
	It remains to check that $\ish(A)$ consists of precisely the finite graph isomorphism types that are diagonal. We already know that the $(\mathbb{G},\theta)\in\sh(A)$ with $\theta$ injective have the property that $\mathbb{G}$ is diagonal (such graphs are the comparability graphs of a finite subtree of the full binary tree). Suppose that $(\mathbb{G},\theta)\in\sh(A)$ and $\theta$ is not injective. Suppose that $\mathbb{G}$ is not diagonal. As before, we need only look at a set of vertices $V:=\{v_0,v_1,v_2,v_3\}$ that witness the failure of the diagonal property. In particular, $V$ must have the property that the subgraph of $\mathbb{G}$ induced by $V$ is either $C_4$ or $\ell_4$ as these are the only non-diagonal 4-vertex graphs. Because $(V,\theta\restriction V)$ is an element of $\sh(A)$ and because $V$ is diagonal, $\theta\restriction V$ is not injective. Suppose that $\theta(v_i)=\theta(v_j)$ for a distinct pair of indices $i,j\in\{0,1,2,3\}$. Then the two vertices $v_i$ and $v_j$ share the same neighbors (other than each other) and are neighbors of each other by the definition of $\sh(A)$. One can check that no distinct pair of vertices in $C_4$ or $\ell_4$ satisfy this condition. Thus $\mathbb{G}$ is diagonal.
\end{proof}

In order to use the information in Lemma~\ref{necviadist} to characterize the graph-like distributions of $\mathrm{SOP}_2$-types, we will need to determine the necessary sets for the two graph isomorphism types in $\ish(A)$.

\begin{lem}\label{sop2necsets}
	With $\varphi$, $A$, and $\Gamma$ as in Lemma~\ref{sop2shape}, the $\se$-minimal necessary sets for the elements of $\Gamma$ in $\sh(A)$ are the edges:
	\[\vcenter
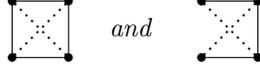
{\hbox{\begin{tikzpicture}[circle dotted/.style={line width = \wid, dash pattern= on .05mm off 1mm, line cap = round}]
		\tikzset{enclosed/.style={draw, circle, inner sep=0pt, minimum size=.1cm, fill=black}}
		 
		%Vertices
		\node[enclosed] (1) at (0.75,.75) {};
		\node[enclosed] (2) at (1.5,.75) {};
		\node[enclosed] (3) at (0.75,0) {};
		\node[enclosed] (4) at (1.5,0) {};
		
		%Edges
		\draw (1) -- (2);
		\draw (2) -- (4);
		\draw (1) -- (3);
		\draw (3) -- (4);
		\draw [circle dotted] (1) -- (4);
		\draw [circle dotted] (2) -- (3);
	\end{tikzpicture}}}\quad
	\text{ and }\quad
	\vcenter{\hbox{
	\begin{tikzpicture}[circle dotted/.style={line width = \wid, dash pattern= on .05mm off 1mm, line cap = round}]
		\tikzset{enclosed/.style={draw, circle, inner sep=0pt, minimum size=.1cm, fill=black}}
		
		%Vertices
		\node[enclosed] (5) at (2.5,0) {};
		\node[enclosed] (6) at (3.25,0) {};
		\node[enclosed] (7) at (3.25,.75) {};
		\node[enclosed] (8) at (2.5,.75) {};
		
		%Edges
		\draw (5) -- (6);
		\draw (6) -- (7);
		\draw (7) -- (8);
		\draw [circle dotted] (5) -- (7);
		\draw [circle dotted] (6) -- (8);
	\end{tikzpicture}}}\]
	where the dotted edges represent the elements of the corresponding necessary sets.
\end{lem}
\begin{proof}
	By the diagonal characterization of graphs arising from the comparability relation on finite trees in \cite{wolk}, the analysis in the proof of Lemma~\ref{sop2shape}, and Lemma~\ref{necviadist}, the above sets are necessary for $\ell_4$ and $C_4$. Since both $C_4$ and $\ell_4$ become diagonal when any one of the dotted edges is added, these necessary sets are $\se$-minimal.
\end{proof}

This analysis allows us to give both a characterization of the distributions of $\mathrm{SOP}_2$-types as well as giving a new characterization of good ultrafilters in terms of which complete subgraphs of ultragraphs with factors coming from $\sh(A)$ can be extended to \emph{internal} complete subgraphs.

\begin{thm}\label{sop2goodequivalents}
	Let $\D$ be a regular ultrafilter on the infinite cardinal $\lambda$. Then the following are equivalent:
	\begin{enumerate}
		\item $\D$ is good.
		\item If $\M$ is a structure and $(\varphi,A)$ witness $\mathrm{SOP}_2$ in $\M$ then every $(A)$-type of cardinality $\leq\lambda$ in $\widehat{\M}:=\M^\D$ is realized in $\widehat{\M}$.
		\item For all sequences of graphs $(\mathbb{G}_\alpha)_{\alpha<\lambda}$ such that $C_4\nleq \mathbb{G}_\alpha$ and $\ell_4\nleq \mathbb{G}_\alpha$ for all $\alpha<\lambda$ the ultraproduct $\mathbb{G}:=\mathbb{G}_\alpha^\D$ has the property that every complete $H\se\mathbb{G}$ with $|H|=\lambda$ is contained within an internal complete subgraph of $\mathbb{G}$.
		\item Every graph-like function $f:P_\omega(\lambda)\to\D$ that satisfies
	\[f(\{x_0,x_1\})\cap f(\{x_1,x_2\})\cap f(\{x_2,x_3\})\se f(\{x_0,x_2\})\cup f(\{x_1,x_3\})\] for all distinct $x_0,x_1,x_2,x_3\in\lambda$ has a multiplicative refinement.
	\end{enumerate}
\end{thm}
\begin{proof}
    $(1)$$\iff$$(2)$: This is a restatement of the fact that a regular ultrafilter $\D$ is good if and only if $\D$ realizes all $\mathrm{SOP}_2$-types over sets of size $\lambda$, which is proved by Malliaris and Shelah in \cite{mmss_pt} (see Conclusion~11.9 and Main Theorem~11.11).
    
   $(2)$$\iff$$(3)$: This is a special case of Theorem~\ref{graphcompletenessAtypes}.
    
    $(2)$$\iff$$(4)$: Combining Theorem~\ref{realAalphatype} and Lemma~\ref{sop2necsets} we have that every $(A)$-type of $\widehat{\M}$ of cardinality $\lambda$ is realized in $\widehat{\M}$ if and only if both
    \begin{align*}
        f(\{x_0,x_1\})\cap f(\{x_1,x_2\})\cap f(\{x_2,x_3\})&\se f(\{x_0,x_2\})\cup f(\{x_1,x_3\})\\
        f(\{x_0,x_3\})\cap f(\{x_0,x_1\})\cap f(\{x_1,x_2\})\cap f(\{x_2,x_3\})&\se f(\{x_0,x_2\})\cup f(\{x_1,x_3\})
    \end{align*}
    whenever $x_0,x_1,x_2,x_3$ are distinct elements of $\lambda$. As the intersection on the left-hand side of the second condition is always a subset of the left-hand side of the first condition, both conditions are satisfied if and only if the first condition is satisfied.
\end{proof}

\begin{rmrk}
    The condition on $f$ in statement $(4)$ of Theorem~\ref{sop2goodequivalents} is a direct translation of requiring that the distribution graph sequence for $f$ consists of graphs satisfying the diagonal property. That is, by Lemma~\ref{graphseqfromshtype}, statement $(4)$ requires that the distribution graph sequence of $f$ consists of induced subgraphs of the \L o\'s graph sequence for some $\mathrm{SOP}_2$-type in $\widehat\M$.
\end{rmrk}

\subsection{Cuts in Linear Orders and $\mathrm{SOP}$}

We make an analysis of types that witness a cut in an ultrapower of an infinite linear order being filled. This analysis follows the same basic structure as the analysis of $\mathrm{SOP}_2$-types presented above, but is slightly complicated by the shape corresponding to such types being more complicated. We first characterize the graphs $\mathbb{G}$ whose vertices are open intervals in some linear order $\mathbb{L}$ and which have an edge between two intervals $I_1$ and $I_2$ if and only if $I_1\cap I_2$ is nonempty in $\mathbb{L}$. Such graphs will be the elements of the shapes of the types being studied. Conveniently, the isomorphism types of these graphs will not depend on the linear order under consideration (see Lemma~\ref{linearordershape} below). Because of this, we will call such graphs interval intersection graphs without any reference to the linear order from which the intervals come from.

\begin{lem}\label{linearordershape}
	Suppose that $\mathbb{L}=(L,<)$ is an infinite linear order and that $\varphi(x;a,b)$ is the formula $a<x<b$. Then $L^2$ is a graph-like $\varphi$-set and the set $\Gamma$ of $\leq$-minimal graphs not in $\sh(L^2)$ consists exactly of the graphs/graph families:%

	\[\Rm{1}:=\vcenter
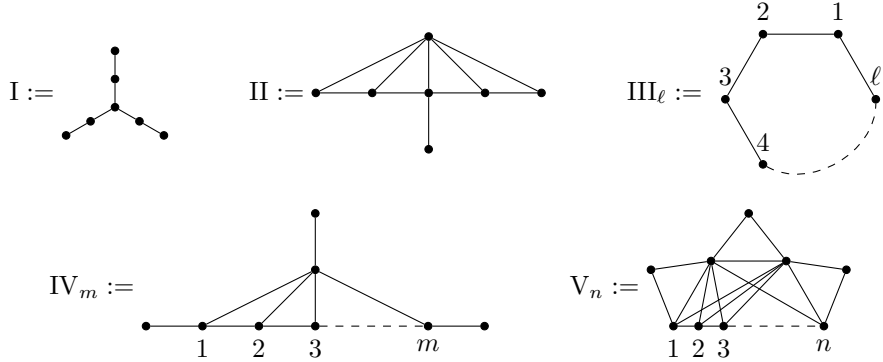
{\hbox{\begin{tikzpicture}
		\tikzset{enclosed/.style={draw, circle, inner sep=0pt, minimum size=.1cm, fill=black}}
		 
		%Vertices
		\node[enclosed] (1) at (0,0) {};
		\node[enclosed] (2) at (0,.75*\s) {};
		\node[enclosed] (3) at (0,1.5*\s) {};
		\node[enclosed] (4) at (-0.65*\s,-0.375*\s) {};
		\node[enclosed] (5) at (-1.3*\s,-.75*\s) {};
		\node[enclosed] (6) at (0.65*\s,-0.375*\s) {};
		\node[enclosed] (7) at (1.3*\s,-.75*\s) {};
		
		%Edges
		\draw (1) -- (2);
		\draw (2) -- (3);
		\draw (1) -- (4);
		\draw (4) -- (5);
		\draw (1) -- (6);
		\draw (6) -- (7);
	\end{tikzpicture}}}\gap
	\Rm{2}:=\vcenter{\hbox{\begin{tikzpicture}
		\tikzset{enclosed/.style={draw, circle, inner sep=0pt, minimum size=.1cm, fill=black}}
		 
		%Vertices
		\node[enclosed] (1) at (0,0) {};
		\node[enclosed] (2) at (-0.75*\scale,0) {};
		\node[enclosed] (3) at (-1.5*\scale,0) {};
		\node[enclosed] (4) at (0,0.75*\scale) {};
		\node[enclosed] (5) at (0.75*\scale,0) {};
		\node[enclosed] (6) at (1.5*\scale,0) {};
		\node[enclosed] (7) at (0,-.75*\scale) {};
		
		%Edges
		\draw (1) -- (2);
		\draw (2) -- (3);
		\draw (1) -- (5);
		\draw (5) -- (6);
		\draw (4) -- (1);
		\draw (4) -- (2);
		\draw (4) -- (3);
		\draw (4) -- (5);
		\draw (4) -- (6);
		\draw (1) -- (7);
	\end{tikzpicture}}}\gap
	\Rm{3}_\ell:=\vcenter{\hbox{\begin{tikzpicture}
		\tikzset{enclosed/.style={draw, circle, inner sep=0pt, minimum size=.1cm, fill=black}}
		
		%Set Up
		\newcommand\angleinit{-pi/3}
		\newcommand\anglec{pi/3}
		
		%Vertices
		\foreach \x/\y in {1/$\ell$,2/$1$,3/$2$,4/$3$,5/$4$}
		{\node[enclosed, label=\y] (\x) at ({cos(deg(\angleinit+\x*\anglec))*\scaleIII},{sin(deg(\angleinit+\x*\anglec))*\scaleIII}) {};}
		
		%Edges
		\foreach \x [evaluate=\x as \y using int(\x+1)] in {1,2,3,4}
		{\draw (\x) -- (\y);}
		\draw [dashed] (5) edge[bend right=60] (1);
		
	\end{tikzpicture}}}
	\]
	\[
		\Rm{4}_m:=\vcenter{\hbox{\begin{tikzpicture}
		\tikzset{enclosed/.style={draw, circle, inner sep=0pt, minimum size=.1cm, fill=black}}
		 
		%Vertices
		\node [enclosed] (bl) at (-3*\scaleIV,0) {};
		\node [enclosed] (br) at (3*\scaleIV,0) {};
		\node [enclosed] (u) at (0,1*\scaleIV) {};
		\node [enclosed] (uu) at (0,2*\scaleIV) {};
		\node [enclosed, label=below:$1$] (1) at (-2*\scaleIV,0) {};
		\node [enclosed, label=below:$2$] (2) at (-1*\scaleIV,0) {};
		\node [enclosed, label=below:$3$] (3) at (0,0) {};
		\node [enclosed, label=below:$m$] (n) at (2*\scaleIV,0) {};

		%Edges
		\foreach \x in {1,2,3,n}
		{\draw (u) -- (\x);}
		\draw (bl) -- (1);
		\draw (n) -- (br);
		\draw (uu) -- (u);
		\foreach \x [evaluate=\x as \y using int(\x+1)] in {1,2}
		{\draw (\x) -- (\y);}
		\draw [dashed] (3) -- (n);
		
	\end{tikzpicture}}}\gap
	\Rm{5}_n:=\vcenter{\hbox{\begin{tikzpicture}
		\tikzset{enclosed/.style={draw, circle, inner sep=0pt, minimum size=.1cm, fill=black}}
		 
		\newcommand\uinit{7*pi/6}
		\newcommand\oinit{pi}
		\newcommand\oscale{\scaleV}
		\newcommand\uscale{1.5*\scaleV}
		\newcommand\inc{pi/3}
		
		%Vertices
		\foreach \lab/\init/\scalar in {u/\uinit/\uscale,o/\oinit/\oscale}
		{\foreach \n [evaluate=\n as \angel using {deg(\init-\n*\inc)}] in {1,2}
			{\node [enclosed] (\lab\n) at ({cos(\angel)*\scalar},{sin(\angel)*\scalar}) {};}
		}
		\node [enclosed] (u3) at ({cos(deg(\inc/2))*\uscale}, {sin(deg(\inc/2))*\uscale}) {};
		
		\foreach \x/\y in {1/$1$,2/$2$,3/$3$}
		{\node [enclosed, label=below:\y] (\x) at ({(\x-1)/3-\oscale},0) {};}
		\node [enclosed, label=below:$n$] (4) at (\oscale,0) {};	
		
		%Edges
		\foreach \top in {o1,o2}
		{\foreach \bottom in {1,2,3,4}
			{\draw (\top) -- (\bottom);}
		}
		\draw (o1) -- (o2);
		
		\foreach \x [evaluate=\x as \y using int(\x+1)] in {1,2}
		{\draw (\x) -- (\y);}
		\draw [dashed] (3) -- (4);
		
		\draw (1) -- (u1);
		\draw (o1) -- (u1);
		\draw (o1) -- (u2);
		\draw (o2) -- (u2);
		\draw (o2) -- (u3);
		\draw (4) -- (u3);
		
	\end{tikzpicture}}}\]
	where $\ell\geq 4$, $m\geq 2$, and $n\geq 1$ (in $\Rm{5}_1$ the nodes labeled $1$ and $n$ are the same node).
\end{lem}
\begin{proof}
	That $L^2$ is a graph-like $\varphi$-set follows from the fact that $\{\varphi(x;a_i,b_i):i\in\{0,1,\dots,n\}\}$ is consistent if and only if the interval $(\max_i a_i,\min_j b_j)$ is non-empty and that this is equivalent to the intersection $(a_i,b_i)\cap (a_j,b_j)$ being nonempty for all $i,j\in\{0,1,\dots,n\}$.
	
	The statement of the lemma is now equivalent to asking that the above graphs are the $\leq$-minimal graphs that cannot arise in the following way: the vertices are a finite collection of possibly repeated non-empty open intervals in $\mathbb{L}$ and there is an edge between two intervals if and only if their intersection is non-empty.
	
	This result (with \emph{distinct} intervals) is proved in the particular case that $\mathbb{L}=\mathbb{R}$ by Lekkekerker and Boland \cite{lekkekerker_boland}. Our strategy will be to reduce the general case to the case of distinct intervals in $\mathbb{R}$. That is, we show that $\ish(\mathbb{L}^2)$ is equal to the set of isomorphism types of interval intersection graphs in $\R$ arising from distinct intervals.
	
	First, we note that we may assume that $\mathbb{L}$ contains $(\omega,<)$ as a suborder (if not, the reverse order on $\mathbb{L}$ works). Suppose that $\mathbb{G}$ is the intersection graph for the collection of open intervals $\{(a_i,b_i):i\in\{0,1,\dots,n\}\}$ where $a_i,b_i\in \R$. We may then construct a set of intervals in $\omega$ by stretching out all of the $(a_i,b_i)$ and translating them to the positive half of $\R$ so that the images of $(a_i,b_i)$ and $(a_j,b_j)$ intersect in $\omega$ if and only if $(a_i,b_i)\cap(a_j,b_j)\neq\emptyset$ and hence our new intervals have an intersection graph isomorphic to the intersection graph of the $(a_i,b_i)$. Mapping these intervals to intervals of $\mathbb{L}$ via the inclusion $\omega\hookrightarrow\mathbb{L}$ does not change the intersection graph so long as we did not allow intersections to look like $(n,n+1)$ in $\omega$ (this can be achieved by choosing how much to stretch out the original intervals in $\mathbb{R}$). Thus the set of intersection graphs in $\mathbb{L}$ contains all of the intersection graphs in $\mathbb{R}$.
	
	For the other inclusion, suppose that $\mathbb{G}=(\{I_i:i\in\{0,1,\dots,n\}\},E)$ is the intersection graph for the (possible repeating) non-empty intervals $I_i$ in $\mathbb{L}$. Let $I_i=(a_i,b_i)$ and, for each $\{I_i,I_j\}\in E$, choose an element $e_{ij}\in I_i\cap I_j$. Define an order $\mathbb{M}$ with underlying set the union of the sets
	\begin{align*}
	    &\{a_i:i\in\{0,1,\dots,n\}\}\\
	    &\{b_i:i\in\{0,1,\dots,n\}\}\\
	    &\{e_{ij}:i<j,\{I_i,I_j\}\in E\}
	\end{align*}
	and with ordering given by $<^\mathbb{L}$. $\mathbb{M}$ is finite and so isomorphic to an initial segment of $\omega$. The map $\mathbb{M}\hookrightarrow\omega\hookrightarrow\R$ sends the intervals $I_i$ to intervals of $\R$ having the same intersection graph as the $I_i$. If the intervals $I_{i_0},\dots,I_{i_m}$ are equal (and none of the other intervals are equal to $I_{i_0}$), then we may take the distinct intervals 
	\[\left\{\left(a_{i_0}+\frac{j}{m+1},b_{i_0}-\frac{j}{m+1}\right):j\in\{0,\dots,m\}\right\}\]
	to be the intervals corresponding to the $I_{i_j}$ without changing the intersection graph. \qedhere
	
%	Define \[\delta=\min\{|c_i-c_j|:c_k=a_k\text{ or }c_k=b_k\text{ and }c_i-c_j\neq 0\}\text{ and }s=\min_i a_i\] and consider the set of intervals in $\omega$ given by
%	\[\left\{I_i:=\left(f(a_i),f(b_i)\right):i\in[0,n]\cap\N\right\}.\] We claim that the intersection graph of the $I_i$ is isomorphic to the intersection graph of the intervals $(a_i,b_i)$ via the map $I_i\mapsto (a_i,b_i)$. This is equivalent to showing that $I_i\cap I_j=\emptyset$ if and only if $(a_i,b_i)\cap(a_j,b_j)=\emptyset$. If $(a_i,b_i)\cap(a_j,b_j)=\emptyset$, we may assume that $b_i<a_j$. 
\end{proof}

For the next several propositions we fix the following suppositions: $\M$ is a structure with an infinite $L\se M$ (not necessarily definable!) and a definable binary relation $<$ on $\M$ such that $<$ restricts to a linear ordering on $L$ that contains a copy of $\omega$ (as an ordered set). Let $\varphi(x;y,z)$ be the formula $y<x<z$.

The following fact from \cite[Theorem~3]{lekkekerker_boland} will be useful for determining minimal necessary sets for the graphs that cannot arise as the intersection graph of a finite collection of intervals.

\begin{defn}
    A graph $\mathbb{G}$ has an \emph{irreducible cycle of length $n$} if there are $n$ distinct vertices $v_1,\dots,v_n$ of $\mathbb{G}$ such that $v_1v_2\ldots v_{n-1}v_nv_1$ is a path and there is an edge between $v_i$ and $v_j$ if and only if $i=j\pm 1\pmod{n}$ .
\end{defn}

\begin{fact}\label{asteroidalfact}
	A finite graph $\mathbb{G}$ is not able to be represented as the intersection graph of a finite collection of intervals if and only if it has at least one of the two following properties:
	\begin{enumerate}
		\item $\mathbb{G}$ has an irreducible cycle of length $\geq 4$.
		\item $\mathbb{G}$ has a collection of three vertices that are pairwise not neighbors and, for each of the three vertices $v$, there is a path $P$ between the other two vertices such that $v$ is not a neighbor of any vertex in $P$.
	\end{enumerate}
\end{fact}

\begin{defn}
    A graph $\mathbb{G}$ is called \emph{asteroidal} if it satisfies the condition of Fact~\ref{asteroidalfact}(2). If $\mathbb{G}$ is a graph and the vertices $A=\{v_1,v_2,v_3\}$ witness $\mathbb{G}$ being asteroidal, we will say that $A$ is an \emph{asteroidal set in $\mathbb{G}$}.
\end{defn}

We will use the notion of asteroidal sets to help us to determine the necessary sets for the minimal asteroidal graphs given in Lemma~\ref{linearordershape}. In particular, we know that if $A$ is asteroidal in the graph $\mathbb{G}$ and $\mathbb{G}'$ is an interval intersection graph extending $\mathbb{G}$ then there must be some $a\in A$ and $g\in\mathbb{G}$ such that $a$ and $g$ are neighbors in $\mathbb{G}'$ but not in $\mathbb{G}$ (see Lemma~\ref{asteroidallemma}). There does not appear to be a nice way other than this to determine what the necessary sets are, so the proofs determining the minimal necessary sets will largely proceed by ruling out all of the elements not in the necessary set one at a time.

\begin{lem}\label{asteroidallemma}
	Let $\mathbb{G}$ be an asteroidal graph and $A=\{v_1,v_2,v_3\}$ be asteroidal in $\mathbb{H}$ with paths $P_{v_i}$ between the nodes $A\sm\{v_i\}$ for each $1\leq i\leq3$ witnessing $A$ being asteroidal in $\mathbb{G}$. Then any interval intersection graph $\mathbb{G}'$ having $\mathbb{G}$ as a subgraph must induce a new edge in $\mathbb{G}$ that is in the set $\bigcup_i (\{v_i\}\times [P_{v_i}])$ where $[P_{v_i}]\se G$ is the set of vertices in the path $P_{v_i}$.
\end{lem}
\begin{proof}
	If $\mathbb{G}'$ does not induce any new edges between elements of $A$ and the paths $P_{v_i}$, the $P_{v_i}$ will also witness the fact that $A$ is asteroidal in $\mathbb{G}'$, a contradiction.
\end{proof}

We start by analyzing the class of graphs labeled $\Rm{3}_\ell$ because knowing the necessary sets for this class of graphs will be useful in finding the minimal necessary sets for all of the other classes.

\begin{lem}\label{IIInecsets}
	The edges $A_\ell:=\{\{1,3\}\}\cup\{\{2,k\}:4\leq k\leq\ell\}$, or any translation of $A_\ell$ by an automorphism of $\Rm{3}_\ell$, forms a $\se$-minimal necessary set for $\Rm{3}_\ell$ in $L^2$.
\end{lem}
\begin{proof}
	That this set is necessary is proven by Lekkekerker and Boland in \cite[Lemma~1]{lekkekerker_boland} (note that the word \emph{acyclic} in the referenced lemma means that the graph has no irreducible cycles of length \emph{greater than 3}, so every interval intersection graph is acyclic in this sense). It will then be enough to find an interval intersection graph $\mathbb{G}$ for each $e\in A_\ell$ such that $\mathbb{G}$ extends $\Rm{3}_\ell$, $e$ is an edge of $\mathbb{G}$, and no other $e'\in A_\ell$ is an edge in $\mathbb{G}$.
	
	Let $e=\{i,j\}$ where $i\in\{1,2\}$ and consider the graph $\mathbb{G}$ that is $\Rm{3}_\ell$ with all of the edges $\{j,k\}$ for $k\neq j$ added. If $e=\{1,3\}$, no new edges incident to $2$ are added and if $j\neq 3$ then no elements of $A_\ell$ are added except $e$. We claim that $\mathbb{G}$ is an interval intersection graph.
	
	Since $j$ is a neighbor of every element of $\mathbb{G}$, we may construct a collection of intervals that have an intersection graph isomorphic to the subgraph of $\mathbb{G}$ induced by $G\sm\{j\}$ and then take $j$ to be the union of these intervals. The induced subgraph looks like
	\[\vcenter
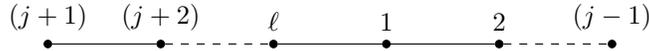
{\hbox{\begin{tikzpicture} [circle dotted/.style={line width = \wid, dash pattern= on .05mm off 1mm, line cap = round}]
		\tikzset{enclosed/.style={draw, circle, inner sep=0pt, minimum size=.1cm, fill=black}}
		 
		%Vertices
		\node[enclosed, label=above:$(j+1)$] (1) at (0,0) {};
		\node[enclosed, label=above:$(j+2)$] (2) at (3*\s,0) {};
		\node[enclosed, label=above:$\ell$] (3) at (6*\s,0) {};
		\node[enclosed, label=above:$1$] (4) at (9*\s,0) {};
		\node[enclosed, label=above:$2$] (5) at (12*\s,0) {};
		\node[enclosed, label=above:$(j-1)$] (6) at (15*\s,0) {};
		
		%Edges
		\draw (1) -- (2);
		\draw [dashed] (2) -- (3);
		\draw (3) -- (4);
		\draw (4) -- (5);
		\draw [dashed] (5) -- (6);
	\end{tikzpicture}}}\]
	which is the intersection graph of a collection of intervals that barely overlap at the endpoints.
	
	The definition of necessary sets for $\mathbb{G}$ in $S$ implies that any translation of a necessary set for $\mathbb{G}$ by an automorphism is a necessary set (in particular, composition with an automorphism of $\mathbb{G}$ is a permutation of the injective graph homomorphisms from $\mathbb{G}$ to any other graph).
\end{proof}

\begin{lem}
	The 6 edges (the two dotted edges below and their images under permutation of the `arms') shown below form a $\se$-minimal necessary set for $\Rm{1}$ in $L^2$.
	\[\vcenter
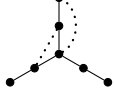
{\hbox{\begin{tikzpicture} [circle dotted/.style={line width = \wid, dash pattern= on .05mm off 1mm, line cap = round}]
		\tikzset{enclosed/.style={draw, circle, inner sep=0pt, minimum size=.1cm, fill=black}}
		 
		%Vertices
		\node[enclosed] (1) at (0,0) {};
		\node[enclosed] (2) at (0,.75*\s) {};
		\node[enclosed] (3) at (0,1.5*\s) {};
		\node[enclosed] (4) at (-0.65*\s,-0.375*\s) {};
		\node[enclosed] (5) at (-1.3*\s,-.75*\s) {};
		\node[enclosed] (6) at (0.65*\s,-0.375*\s) {};
		\node[enclosed] (7) at (1.3*\s,-.75*\s) {};
		
		%Edges
		\draw (1) -- (2);
		\draw (2) -- (3);
		\draw (1) -- (4);
		\draw (4) -- (5);
		\draw (1) -- (6);
		\draw (6) -- (7);
		%Necessary Set
		\draw [circle dotted] (1) edge[bend right = 60] (3);
		\draw [circle dotted] (2) -- (4);
	\end{tikzpicture}}}\]
\end{lem}
\begin{proof}
	For $0\leq n\leq3$, we will call the vertices along the `arm' at the angle $\pi/2+2n\pi/3$ by $a_{n1}$ and $a_{n2}$ where $a_{n1}$ is the neighbor of the center vertex and $a_{n2}$ is the outermost vertex on the given arm. We will denote the center vertex by $c$. We then have that the set $A=\{a_{n2}:0\leq n\leq 2\}$ is asteroidal in $\Rm{1}$. By Lemma~\ref{asteroidallemma}, any intersection graph having $\Rm{1}$ as a subgraph must add an edge incident to one of the elements of $A$, so we will start our search for necessary edges here.
	
	Let $\mathbb{G}$ be an intersection graph and $i:\Rm{1}\to\mathbb{G}$ an injective graph homomorphism (we will denote $i(a_{kj})$ by $a'_{kj}$). We may assume that $a'_{02}$ has an edge that does not occur in $\mathbb{G}$ since any permutation of $A$ can be achieved by a graph automorphism of $\Rm{1}$. If $a'_{02}\mathrel{E}^{\mathbb{G}}a'_{12}$ then the set of vertices $\{a'_{12},a'_{02},a'_{01},c',a'_{11}\}$ can be taken to form a subgraph of $\mathbb{G}$ isomorphic to $\Rm{3}_5$ with the nodes labeled in the order given. From the necessary sets found for $\Rm{3}_n$, one of the edges $e_1:=\{a'_{02},a'_{11}\},\{a'_{12},a'_{01}\},e_2:=\{a'_{02},c'\}$ must also appear in $\mathbb{G}$. We note that the first two edges differ by an automorphism of $\Rm{1}$ so we may treat only one of them explicitly (note that the proposed necessary set is stable under automorphism of $\Rm{1}$). If the edge $e_1$ is in $\mathbb{G}$, then the set $\{a'_{02},a'_{11},c',a'_{01}\}$ forms a subgraph isomorphic to $\Rm{3}_4$, which means that either $e_2$ or $e_3=\{a'_{11},a'_{01}\}$ appear in $\mathbb{G}$.
	
	Thus the statement that either $e_1$ or $e_2$ appears in $\mathbb{G}$ can be reduced to the statement that either $e_2$ or $e_3$ (the dotted edges) appear in $\mathbb{G}$. To show the desired result it will be enough to check that there are two intersection graphs having $\Rm{1}$ as a subgraph with the edge $e_2$ but not $e_3$ and vice versa. To do this, we note that the following two graphs are interval intersection graphs in $\R$ (with labels indicating a possible set of intervals with the given intersection graph):
	\[
		\vcenter
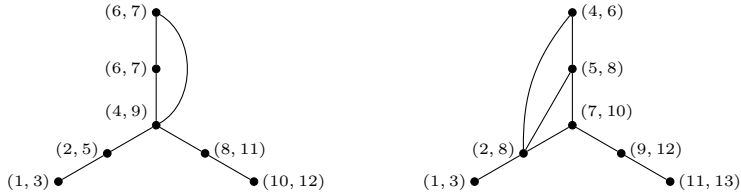
{\hbox{\begin{tikzpicture} [circle dotted/.style={line width = \wid, dash pattern= on .05mm off 1mm, line cap = round}, scale=2]
		\tikzset{enclosed/.style={draw, circle, inner sep=0pt, minimum size=.1cm, fill=black}}
		 
		%Vertices
		\node[enclosed, label={[label distance=-1.25mm, xshift=-.45mm, yshift=-.5mm]176:\tiny $(4{,}\,9)$}] (1) at (0,0) {};
		\node[enclosed, label={[label distance=-0.75mm]left:\tiny $(6{,}\,7)$}] (2) at (0,.75*\s) {};
		\node[enclosed, label={[label distance=-0.75mm]left:\tiny $(6{,}\,7)$}] (3) at (0,1.5*\s) {};
		\node[enclosed, label={[label distance=-0.75mm, yshift=.5mm]left:\tiny $(2{,}\,5)$}] (4) at (-0.65*\s,-0.375*\s) {};
		\node[enclosed, label={[label distance=-0.75mm]left:\tiny $(1{,}\,3)$}] (5) at (-1.3*\s,-.75*\s) {};
		\node[enclosed,  label={[label distance=-0.75mm, yshift=.5mm]right:\tiny $(8{,}\,11)$}] (6) at (0.65*\s,-0.375*\s) {};
		\node[enclosed, label={[label distance=-0.75mm]right:\tiny $(10{,}\,12)$}] (7) at (1.3*\s,-.75*\s) {};
		
		%Edges
		\draw (1) -- (2);
		\draw (2) -- (3);
		\draw (1) -- (4);
		\draw (4) -- (5);
		\draw (1) -- (6);
		\draw (6) -- (7);
		%Added Edges
		\draw (1) edge[bend right = 60] (3);
	\end{tikzpicture}}}\gap
	\vcenter{\hbox{\begin{tikzpicture} [circle dotted/.style={line width = \wid, dash pattern= on .05mm off 1mm, line cap = round}, scale= 2]
		\tikzset{enclosed/.style={draw, circle, inner sep=0pt, minimum size=.1cm, fill=black}}
		 
		%Vertices
		\node[enclosed, label={[label distance=-1.25mm, xshift=.45mm, yshift=-.5mm]4:\tiny $(7{,}\,10)$}] (1) at (0,0) {};
		\node[enclosed, label={[label distance=-0.75mm]right:\tiny $(5{,}\,8)$}] (2) at (0,.75*\s) {};
		\node[enclosed, label={[label distance=-0.75mm]right:\tiny $(4{,}\,6)$}] (3) at (0,1.5*\s) {};
		\node[enclosed, label={[label distance=-0.75mm, yshift=.5mm]left:\tiny $(2{,}\,8)$}] (4) at (-0.65*\s,-0.375*\s) {};
		\node[enclosed, label={[label distance=-0.75mm]left:\tiny $(1{,}\,3)$}] (5) at (-1.3*\s,-.75*\s) {};
		\node[enclosed,  label={[label distance=-0.75mm, yshift=.5mm]right:\tiny $(9{,}\,12)$}] (6) at (0.65*\s,-0.375*\s) {};
		\node[enclosed, label={[label distance=-0.75mm]right:\tiny $(11{,}\,13)$}] (7) at (1.3*\s,-.75*\s) {};
		
		%Edges
		\draw (1) -- (2);
		\draw (2) -- (3);
		\draw (1) -- (4);
		\draw (4) -- (5);
		\draw (1) -- (6);
		\draw (6) -- (7);
		%Added Edges
		\draw (2) -- (4);
		\draw (3) edge[bend right=20] (4);
	\end{tikzpicture}}}
	\]
	The other edges belong in the $\se$-minimal necessary set by symmetric arguments.
\end{proof}

\begin{lem}
	The dotted edges in \[
	\vcenter
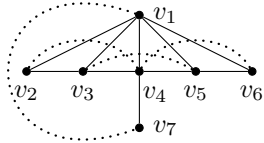
{\hbox{\begin{tikzpicture} [circle dotted/.style={line width = \wid, dash pattern= on .05mm off 1mm, line cap = round}]
		\tikzset{enclosed/.style={draw, circle, inner sep=0pt, minimum size=.1cm, fill=black}}
		 
		%Vertices
		\node[enclosed, label={[label distance=0mm,xshift=-1mm,yshift=-.2mm]290:$v_4$}] (1) at (0,0) {};
		\node[enclosed, label=below:$v_3$] (2) at (-0.75*\scale,0) {};
		\node[enclosed, label=below:$v_2$] (3) at (-1.5*\scale,0) {};
		\node[enclosed, label=right:$v_1$] (4) at (0,0.75*\scale) {};
		\node[enclosed, label=below:$v_5$] (5) at (0.75*\scale,0) {};
		\node[enclosed, label=below:$v_6$] (6) at (1.5*\scale,0) {};
		\node[enclosed, label=right:$v_7$] (7) at (0,-.75*\scale) {};
		
		%Edges
		\draw (1) -- (2);
		\draw (2) -- (3);
		\draw (1) -- (5);
		\draw (5) -- (6);
		\draw (4) -- (1);
		\draw (4) -- (2);
		\draw (4) -- (3);
		\draw (4) -- (5);
		\draw (4) -- (6);
		\draw (1) -- (7);
		
		\draw [circle dotted] (4) edge[bend right = 110, looseness = 4] (7);
		\draw [circle dotted] (1) edge[bend right = 60] (3);
		\draw [circle dotted] (2) edge[bend left = 30] (5);
		\draw [circle dotted] (1) edge[bend left = 60] (6);
		
	\end{tikzpicture}}}
	\]
	form a necessary set of minimal cardinality for $\Rm{2}$ in $L^2$.
\end{lem}
\begin{proof}
	As before, we recognize that the set $A=\{v_2,v_6,v_7\}$ is asteroidal in $\Rm{2}$ and start our analysis with edges that are adjacent to elements of $A$. We start by noticing that the edges $\{v_2,v_4\},\{v_4,v_6\},$ and $\{v_1,v_7\}$ must be in any necessary set as $\Rm{2}$ becomes an interval intersection graph in $\R$ when any one of these edges is added. For example, the collections of intervals given by:
	\[
	\vcenter
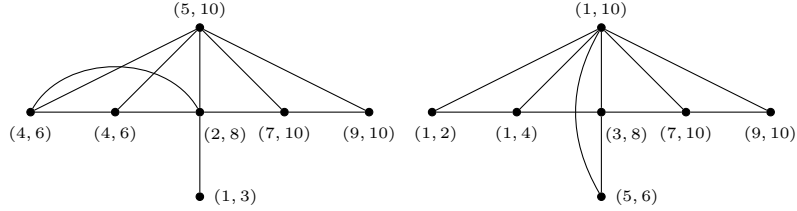
{\hbox{\begin{tikzpicture} [circle dotted/.style={line width = \wid, dash pattern= on .05mm off 1mm, line cap = round}, scale=1.5]
		\tikzset{enclosed/.style={draw, circle, inner sep=0pt, minimum size=.1cm, fill=black}}
		 
		%Vertices
		\node[enclosed, label={[label distance=0mm,xshift=-1mm,yshift=-.2mm]290:\tiny $(2{,}\,8)$}] (1) at (0,0) {};
		\node[enclosed, label=below:\tiny $(4{,}\,6)$] (2) at (-0.75*\scale,0) {};
		\node[enclosed, label=below:\tiny $(4{,}\,6)$] (3) at (-1.5*\scale,0) {};
		\node[enclosed, label={[yshift=-.5mm]above:\tiny $(5{,}\,10)$}] (4) at (0,0.75*\scale) {};
		\node[enclosed, label=below:\tiny $(7{,}\,10)$] (5) at (0.75*\scale,0) {};
		\node[enclosed, label=below:\tiny $(9{,}\,10)$] (6) at (1.5*\scale,0) {};
		\node[enclosed, label=right:\tiny $(1{,}\,3)$] (7) at (0,-.75*\scale) {};
		
		%Edges
		\draw (1) -- (2);
		\draw (2) -- (3);
		\draw (1) -- (5);
		\draw (5) -- (6);
		\draw (4) -- (1);
		\draw (4) -- (2);
		\draw (4) -- (3);
		\draw (4) -- (5);
		\draw (4) -- (6);
		\draw (1) -- (7);
		
		\draw (1) edge[bend right = 60] (3);
	\end{tikzpicture}}}
	\vcenter{\hbox{\begin{tikzpicture} [circle dotted/.style={line width = \wid, dash pattern= on .05mm off 1mm, line cap = round}, scale=1.5]
		\tikzset{enclosed/.style={draw, circle, inner sep=0pt, minimum size=.1cm, fill=black}}
		 
		%Vertices
		\node[enclosed, label={[label distance=0mm,xshift=-1mm,yshift=-.2mm]290:\tiny $(3{,}\,8)$}] (1) at (0,0) {};
		\node[enclosed, label=below:\tiny $(1{,}\,4)$] (2) at (-0.75*\scale,0) {};
		\node[enclosed, label={[xshift=.5mm]below:\tiny $(1{,}\,2)$}] (3) at (-1.5*\scale,0) {};
		\node[enclosed, label={[yshift=-.5mm]above:\tiny $(1{,}\,10)$}] (4) at (0,0.75*\scale) {};
		\node[enclosed, label=below:\tiny $(7{,}\,10)$] (5) at (0.75*\scale,0) {};
		\node[enclosed, label=below:\tiny $(9{,}\,10)$] (6) at (1.5*\scale,0) {};
		\node[enclosed, label=right:\tiny $(5{,}\,6)$] (7) at (0,-.75*\scale) {};
		
		%Edges
		\draw (1) -- (2);
		\draw (2) -- (3);
		\draw (1) -- (5);
		\draw (5) -- (6);
		\draw (4) -- (1);
		\draw (4) -- (2);
		\draw (4) -- (3);
		\draw (4) -- (5);
		\draw (4) -- (6);
		\draw (1) -- (7);
		
		\draw (4) edge[bend right = 30] (7);
		
	\end{tikzpicture}}}
	\]
	(note that the case of the edge $\{v_4,v_6\}$ being added is symmetric to the case of $\{v_2,v_4\}$).
	
	It remains to find any intersection graphs $\mathbb{G}$ with $\Rm{2}$ as a subgraph and which fail to contain any of the edges so far found. The remaining choices for adding an edge incident to an element of $A$ are (up to symmetry) $e_1=\{v_2,v_7\}$, $e_2=\{v_2,v_5\}$, and $e_3=\{v_2,v_6\}$. Note that the edge $\{v_3,v_7\}$ (or the symmetric $\{v_5,v_7\}$) can be added to $\Rm{2}$ without changing the fact that $A$ is asteroidal. Thus another edge incident to an element of $A$ must be still be added to produce an intersection graph, allowing us to ignore the edge $\{v_3,v_7\}$.
	
	Suppose first that an intersection graph $\mathbb{G}$ extends $\Rm{2}$ and has the edge $e_1$. Then the vertices $\{v_1,v_4,v_7,v_2\}$ form a cycle of length 4, and so either the edge $\{v_1,v_7\}$ or the edge $\{v_2,v_4\}$ belong to $\mathbb{G}$. Both of these edges are already accounted for, so all intersection graphs with the edge $\{v_2,v_7\}$ are already accounted for.
	
	Now suppose that $\mathbb{G}$ is an intersection graph containing the edge $\{v_2,v_5\}$. But then $\mathbb{G}$ contains the 4-cycle $v_2,v_3,v_4,v_5$ and so must contain either the edge $\{v_3,v_5\}$ or the edge $\{v_2,v_4\}$. The latter is already covered, so adding $\{v_3,v_5\}$ to the necessary set will cover all intersection graphs with the edge $\{v_2,v_5\}$ (and, by a symmetric argument, all intersection graphs with the edge $\{v_3,v_6\}$). It remains to check that the edge $\{v_3,v_5\}$ occurs in an intersection graph not having any of the other edges in the proposed necessary set. The graph $\Rm{2}$ with the edges $\{v_2,v_5\}$ and $\{v_3,v_5\}$ added accomplishes this goal. In particular, the collection of real intervals
	\[
	\vcenter
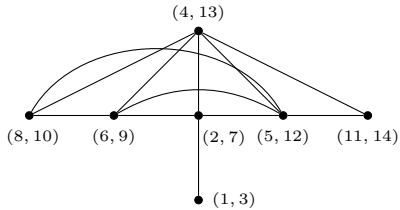
{\hbox{\begin{tikzpicture} [circle dotted/.style={line width = \wid, dash pattern= on .05mm off 1mm, line cap = round}, scale=1.5]
		\tikzset{enclosed/.style={draw, circle, inner sep=0pt, minimum size=.1cm, fill=black}}
		 
		%Vertices
		\node[enclosed, label={[label distance=0mm,xshift=-1mm,yshift=-.2mm]290:\tiny $(2{,}\,7)$}] (1) at (0,0) {};
		\node[enclosed, label=below:\tiny $(6{,}\,9)$] (2) at (-0.75*\scale,0) {};
		\node[enclosed, label={[xshift=.5mm]below:\tiny $(8{,}\,10)$}] (3) at (-1.5*\scale,0) {};
		\node[enclosed, label={[yshift=-.5mm]above:\tiny $(4{,}\,13)$}] (4) at (0,0.75*\scale) {};
		\node[enclosed, label=below:\tiny $(5{,}\,12)$] (5) at (0.75*\scale,0) {};
		\node[enclosed, label=below:\tiny $(11{,}\,14)$] (6) at (1.5*\scale,0) {};
		\node[enclosed, label=right:\tiny $(1{,}\,3)$] (7) at (0,-.75*\scale) {};
		
		%Edges
		\draw (1) -- (2);
		\draw (2) -- (3);
		\draw (1) -- (5);
		\draw (5) -- (6);
		\draw (4) -- (1);
		\draw (4) -- (2);
		\draw (4) -- (3);
		\draw (4) -- (5);
		\draw (4) -- (6);
		\draw (1) -- (7);
		
		\draw (3) edge[bend left = 60] (5);
		\draw (2) edge[bend left = 30] (5);
		
	\end{tikzpicture}}}
	\]
	has this intersection graph.
	
	Finally suppose that $\mathbb{G}$ is an intersection graph extending $\Rm{2}$ and containing the edge $\{v_2,v_6\}$. Then the vertices $\{v_2,v_6,v_5,v_4,v_3\}$ form a subgraph isomorphic to $\Rm{3}_5$ labels occurring in the given order, so one of the edges $\{v_2,v_5\}$, $\{v_6,v_4\}$, or $\{v_6,v_3\}$ must also be in $\mathbb{G}$. Any graph containing the first two of these edges has already been explicitly covered, and any intersection graph with $\{v_3,v_6\}$ is covered by an argument symmetric to the argument for graphs containing $\{v_2,v_5\}$. Thus every such $\mathbb{G}$ already has an edge in our proposed necessary set.
	
	This necessary set is of minimal cardinality because all necessary sets must contain all but one of the proposed edges
\end{proof}

\begin{lem}
	The dotted edges in
	\[ 
	\vcenter
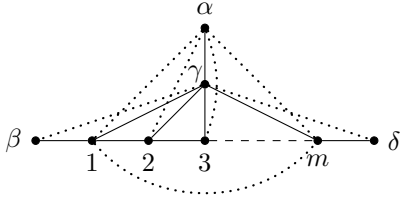
{\hbox{\begin{tikzpicture}[circle dotted/.style={dash pattern= on .05mm off 1mm, line cap = round, line width = \wid}]
		\tikzset{enclosed/.style={draw, circle, inner sep=0pt, minimum size=.1cm, fill=black}}
		 
		%Vertices
		\node [enclosed, label=left:$\beta$] (bl) at (-3*\scaleIV,0) {};
		\node [enclosed, label=right:$\delta$] (br) at (3*\scaleIV,0) {};
		\node [enclosed, label={[label distance=-2mm]135:$\gamma$}] (u) at (0,1*\scaleIV) {};
		\node [enclosed, label=above:$\alpha$] (uu) at (0,2*\scaleIV) {};
		\node [enclosed, label=below:$1$] (1) at (-2*\scaleIV,0) {};
		\node [enclosed, label=below:$2$] (2) at (-1*\scaleIV,0) {};
		\node [enclosed, label=below:$3$] (3) at (0,0) {};
		\node [enclosed, label=below:$m$] (n) at (2*\scaleIV,0) {};

		%Edges
		\foreach \x in {1,2,3,n}
		{\draw (u) -- (\x);}
		\draw (bl) -- (1);
		\draw (n) -- (br);
		\draw (uu) -- (u);
		\foreach \x [evaluate=\x as \y using \x+1] in {1,2}
		{\draw (\x) -- (\y);}
		\draw [dashed] (3) -- (n);

		\draw [circle dotted] (1) edge[bend right = 50] (n);
		\draw [circle dotted] (u) -- (bl);
		\draw [circle dotted] (u) -- (br);
		\draw [circle dotted] (uu) -- (1);
		\draw [circle dotted] (uu) -- (2);
		\draw [circle dotted] (uu) -- (n);
		\draw [circle dotted] (uu) edge[bend left = 20] (3);
		
	\end{tikzpicture}}}
	\]
	form a necessary set of minimal cardinality for $\Rm{4}_m$ in $L^2$.
\end{lem}
\begin{proof}
	As before, we start by identifying the set $A=\{\alpha,\beta,\delta\}$ as being asteroidal in $\Rm{4}_m$.
	The edges $\{\beta,\gamma\}$, $\{\gamma,\delta\}$, and $\{\alpha,i\}$ for each $1\leq i\leq m$ must be in every necessary set for $\Rm{4}_m$ as $\Rm{4}_m$ becomes an intersection graph as soon as any one of the above edges is added. Up to symmetry, representative examples of these intersection graphs are:
	\[ 
	\vcenter
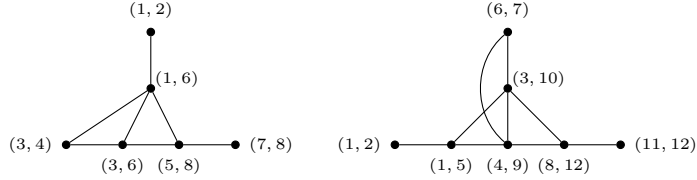
{\hbox{\begin{tikzpicture}[circle dotted/.style={dash pattern= on .05mm off 1mm, line cap = round, line width = \wid}]
		\tikzset{enclosed/.style={draw, circle, inner sep=0pt, minimum size=.1cm, fill=black}}
		 
		%Vertices
		\node [enclosed, label=left:\tiny {$(3,4)$}] (bl) at (-1.5*\scaleIV,0) {};
		\node [enclosed, label=right:\tiny {$(7,8)$}] (br) at (1.5*\scaleIV,0) {};
		\node [enclosed, label={[label distance=-2mm, xshift=.3mm]45:\tiny {$(1,6)$}}] (u) at (0,1*\scaleIV) {};
		\node [enclosed, label=above:\tiny {$(1,2)$}] (uu) at (0,2*\scaleIV) {};
		\node [enclosed, label=below:\tiny {$(3,6)$}] (1) at (-.5*\scaleIV,0) {};
		\node [enclosed, label=below:\tiny {$(5,8)$}] (2) at (.5*\scaleIV,0) {};

		%Edges
		\foreach \x in {1,2}
		{\draw (u) -- (\x);}
		\draw (bl) -- (1);
		\draw (2) -- (br);
		\draw (uu) -- (u);
		\foreach \x [evaluate=\x as \y using \x+1] in {1}
		{\draw (\x) -- (\y);}
		
		\draw (u) edge (bl);

	\end{tikzpicture}}}\quad
	\vcenter{\hbox{\begin{tikzpicture}[circle dotted/.style={dash pattern= on .05mm off 1mm, line cap = round, line width = \wid}]
		\tikzset{enclosed/.style={draw, circle, inner sep=0pt, minimum size=.1cm, fill=black}}
		 
		%Vertices
		\node [enclosed, label=left:\tiny {$(1,2)$}] (bl) at (-2*\scaleIV,0) {};
		\node [enclosed, label=right:\tiny {$(11,12)$}] (br) at (2*\scaleIV,0) {};
		\node [enclosed, label={[label distance=-2mm, xshift=.3mm]45:\tiny {$(3,10)$}}] (u) at (0,1*\scaleIV) {};
		\node [enclosed, label=above:\tiny {$(6,7)$}] (uu) at (0,2*\scaleIV) {};
		\node [enclosed, label=below:\tiny {$(1,5)$}] (1) at (-1*\scaleIV,0) {};
		\node [enclosed, label=below:\tiny {$(4,9)$}] (2) at (0*\scaleIV,0) {};
		\node [enclosed, label=below:\tiny {$(8,12)$}] (3) at (1*\scaleIV,0) {};

		%Edges
		\foreach \x in {1,2,3}
		{\draw (u) -- (\x);}
		\draw (bl) -- (1);
		\draw (3) -- (br);
		\draw (uu) -- (u);
		\foreach \x [evaluate=\x as \y using \x+1] in {1,2}
		{\draw (\x) -- (\y);}
		
		\draw (uu) edge[bend right=50] (2);

	\end{tikzpicture}}}
	\]
	It remains to show that any intersection graph extending $\Rm{4}_m$ and not containing any of the aforementioned edges must contain the edge $\{1,m\}$ (and that such intersection graphs do exist).
	
	If $\mathbb{G}$ extends $\Rm{4}_m$ and $\{\alpha,\beta\}$ is in $\mathbb{G}$ then $\{\alpha,\beta,1,\gamma\}$ forms a 4-cycle and so either $\{\alpha,1\}$ is in $\mathbb{G}$ or $\{\beta,\gamma\}$ is in $\mathbb{G}$. However, both of these options are already covered, so $\{\alpha,\beta\}$ cannot be in a $\se$-minimal necessary set. A symmetric argument gets the same result for the edge $\{\alpha,\delta\}$.
	
	Now, in order to avoid $A$ being asteroidal, an intersection graph extending $\Rm{4}_m$ and having none of the previously mentioned edges must have one of the edges $\{\beta,\delta\}$, $\{\beta,m\}$, or $\{\delta,1\}$. If $\{\beta,\delta\}$ is an edge of $\mathbb{G}$ then the vertices $\{\beta,\delta,m,\gamma,1\}$ form a 5-cycle so one of the edges $\{1,m\},\{\gamma,\delta\},$ or $\{\gamma,\beta\}$ must appear in $\mathbb{G}$, but we have already ruled out the latter two, so $\{1,m\}$ must be in $\mathbb{G}$.
	
	If $\{\beta,m\}$ is in $\mathbb{G}$ then the vertices $\{\beta,m,\gamma,1\}$ form a 4-cycle and so either $\{\beta,\gamma\}$ is in $\mathbb{G}$ or $\{1,m\}$ is in $\mathbb{G}$. We already know that graphs with $\{\beta,\gamma\}$ are covered, so we must have that $\{1,m\}$ is in $\mathbb{G}$. A symmetric argument for $\{\delta,1\}$ also shows that $\{1,m\}$ must be an edge in $\mathbb{G}$.
	
	The result now follows from the fact that the extension of $\Rm{4}_m$ with the vertices $\{\beta,\delta\}\cup\{1,\dots,m\}$ inducing a complete subgraph and otherwise having only the edges from $\Rm{4}_m$ is an intersection graph. An example of intervals producing such a graph could be given by setting all of the vertices labelled $1,\dots,m$ to the same non-empty interval and also identify the vertices labelled $\beta$ and $\delta$. These identifications lead to a quotient graph of the proposed graph that is isomorphic to a graph that is a single irreducible path, which was shown to be an intersection graph in Lemma~\ref{IIInecsets}.
\end{proof}

\begin{lem}\label{Vnecsets}
	The dotted edges in
	\[
		\vcenter{\hbox{\begin{tikzpicture}[circle dotted/.style={line width = \wid, dash pattern= on .05mm off 1mm, line cap = round}, scale=1.5]
		\tikzset{enclosed/.style={draw, circle, inner sep=0pt, minimum size=.1cm, fill=black}}
		 
		\newcommand\uinit{7*pi/6}
		\newcommand\oinit{pi}
		\newcommand\oscale{\scaleV}
		\newcommand\uscale{1.5*\scaleV}
		\newcommand\inc{pi/3}
		
		%Vertices
		\node [enclosed, label=above left:$\delta$] (o1) at ({cos(deg(2*pi/3))*\oscale},{sin(deg(2*pi/3))*\oscale}) {};
		\node [enclosed, label=above right:$\epsilon$] (o2) at ({cos(deg(pi/3))*\oscale},{sin(deg(pi/3))*\oscale}) {};
		
		\node [enclosed, label=left:$\alpha$] (u1) at ({cos(deg(5*pi/6))*\uscale},{sin(deg(5*pi/6))*\uscale}) {};
		\node [enclosed, label={[yshift=-1.3mm]above:$\beta$}] (u2) at ({cos(deg(pi/2))*\uscale},{sin(deg(pi/2))*\uscale}) {};
		\node [enclosed, label=right:$\gamma$] (u3) at ({cos(deg(\inc/2))*\uscale}, {sin(deg(\inc/2))*\uscale}) {};
		
		\foreach \x/\y in {1/$1$,2/$2$,3/$3$}
		{\node [enclosed, label=below:\y] (\x) at ({(\x-1)/3-\oscale},0) {};}
		\node [enclosed, label=below:$n$] (4) at (\oscale,0) {};	
		
		%Edges
		\foreach \top in {o1,o2}
		{\foreach \bottom in {1,2,3,4}
			{\draw (\top) -- (\bottom);}
		}
		\draw (o1) -- (o2);
		
		\foreach \x [evaluate=\x as \y using \x+1] in {1,2}
		{\draw (\x) -- (\y);}
		\draw [dashed] (3) -- (4);
		
		\draw (1) -- (u1);
		\draw (o1) -- (u1);
		\draw (o1) -- (u2);
		\draw (o2) -- (u2);
		\draw (o2) -- (u3);
		\draw (4) -- (u3);
		
		\foreach \x in {1,2,3}
		{\draw [circle dotted] (u2) edge[bend left = 10] (\x);}
		\draw [circle dotted] (u2) edge[bend right = 10] (4);
		
		\draw [circle dotted] (o1) edge[bend left = 80, looseness = 2.4] (u3);
		\draw [circle dotted] (o2) edge[bend right = 80, looseness = 2.4] (u1);
	\end{tikzpicture}}}
	\]
	form the unique necessary set of minimal cardinality for $\Rm{5}_n$ in $L^2$.
\end{lem}
\begin{proof}
	We first note that the set $A=\{\alpha,\beta,\gamma\}$ is asteroidal in $\Rm{5}_n$. We also note that if any one of the edges $\{\gamma,\delta\}$, $\{\alpha,\epsilon\}$, or $\{\beta,i\}$ for some $1\leq i\leq n$ are added to $\Rm{5}_n$, then the resulting graph is an intersection graph, and so every necessary set for $\Rm{5}_n$ contains the edges $\{\alpha,\epsilon\}$, $\{\gamma,\delta\}$, and $\{\beta,i\}$ for each $1\leq i\leq n$. Representative examples of such graphs are (up to symmetry):
	
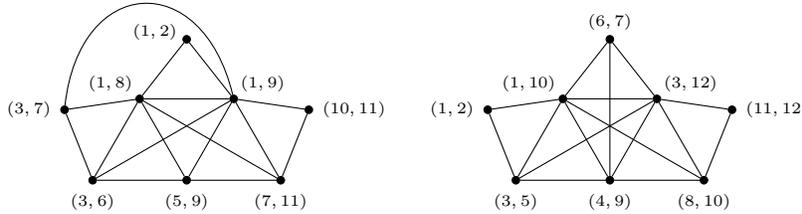
\[
		\hbox{\begin{tikzpicture}[circle dotted/.style={line width = \wid, dash pattern= on .05mm off 1mm, line cap = round}, scale=1.25]
		\tikzset{enclosed/.style={draw, circle, inner sep=0pt, minimum size=.1cm, fill=black}}
		 
		\newcommand\iscale{\scaleV}
		\newcommand\jscale{1.5*\scaleV}
		
		%Vertices
		\node [enclosed, label={[label distance=-1mm]above left:\tiny ${(1,8)}$}] (o1) at ({cos(deg(2*pi/3))*\iscale},{sin(deg(2*pi/3))*\iscale}) {};
		\node [enclosed, label={[label distance=-1mm]above right:\tiny ${(1,9)}$}] (o2) at ({cos(deg(pi/3))*\iscale},{sin(deg(pi/3))*\iscale}) {};
		
		\node [enclosed, label=left:\tiny ${(3,7)}$] (u1) at ({cos(deg(5*pi/6))*\jscale},{sin(deg(5*pi/6))*\jscale}) {};
		\node [enclosed, label={[label distance=-.5mm, yshift=-1.3mm]above left:\tiny ${(1,2)}$}] (u2) at ({cos(deg(pi/2))*\jscale},{sin(deg(pi/2))*\jscale}) {};
		\node [enclosed, label=right:\tiny ${(10,11)}$] (u3) at ({cos(deg(pi/6))*\jscale}, {sin(deg(pi/6))*\jscale}) {};
		
	    \node [enclosed, label=below:\tiny ${(3,6)}$] (1) at (-\iscale,0) {};
	    \node [enclosed, label=below:\tiny ${(5,9)}$] (2) at (0,0) {};
	    \node [enclosed, label=below:\tiny ${(7,11)}$] (3) at (\iscale,0) {};

		%Edges
		\foreach \top in {o1,o2}
		{\foreach \bottom in {1,2,3}
			{\draw (\top) -- (\bottom);}
		}
		\draw (o1) -- (o2);
		
		\foreach \x [evaluate=\x as \y using \x+1] in {1,2}
		{\draw (\x) -- (\y);}
		
		\draw (1) -- (u1);
		\draw (o1) -- (u1);
		\draw (o1) -- (u2);
		\draw (o2) -- (u2);
		\draw (o2) -- (u3);
		\draw (3) -- (u3);

		\draw (o2) edge[bend right = 80, looseness = 2] (u1);
	\end{tikzpicture}}\quad
	\hbox{\begin{tikzpicture}[circle dotted/.style={line width = \wid, dash pattern= on .05mm off 1mm, line cap = round}, scale=1.25]
		\tikzset{enclosed/.style={draw, circle, inner sep=0pt, minimum size=.1cm, fill=black}}
		 
		\newcommand\iscale{\scaleV}
		\newcommand\jscale{1.5*\scaleV}
		
		%Vertices
		\node [enclosed, label={[label distance=-1mm]above left:\tiny ${(1,10)}$}] (o1) at ({cos(deg(2*pi/3))*\iscale},{sin(deg(2*pi/3))*\iscale}) {};
		\node [enclosed, label={[label distance=-1mm]above right:\tiny ${(3,12)}$}] (o2) at ({cos(deg(pi/3))*\iscale},{sin(deg(pi/3))*\iscale}) {};
		
		\node [enclosed, label=left:\tiny ${(1,2)}$] (u1) at ({cos(deg(5*pi/6))*\jscale},{sin(deg(5*pi/6))*\jscale}) {};
		\node [enclosed, label={[yshift=-.5mm]above:\tiny ${(6,7)}$}] (u2) at ({cos(deg(pi/2))*\jscale},{sin(deg(pi/2))*\jscale}) {};
		\node [enclosed, label=right:\tiny ${(11,12)}$] (u3) at ({cos(deg(pi/6))*\jscale}, {sin(deg(pi/6))*\jscale}) {};
		
	    \node [enclosed, label=below:\tiny ${(3,5)}$] (1) at (-\iscale,0) {};
	    \node [enclosed, label=below:\tiny ${(4,9)}$] (2) at (0,0) {};
	    \node [enclosed, label=below:\tiny ${(8,10)}$] (3) at (\iscale,0) {};

		%Edges
		\foreach \top in {o1,o2}
		{\foreach \bottom in {1,2,3}
			{\draw (\top) -- (\bottom);}
		}
		\draw (o1) -- (o2);
		
		\foreach \x [evaluate=\x as \y using \x+1] in {1,2}
		{\draw (\x) -- (\y);}
		
		\draw (1) -- (u1);
		\draw (o1) -- (u1);
		\draw (o1) -- (u2);
		\draw (o2) -- (u2);
		\draw (o2) -- (u3);
		\draw (3) -- (u3);

		\draw (u2) -- (2);
	\end{tikzpicture}}
	\]
	
	Suppose that $\mathbb{G}$ is an intersection graph extending $\Rm{5}_n$ and containing the edge $\{\alpha,\beta\}$. Then the vertices $\{\alpha,\beta,\epsilon,1\}$ form a 4-cycle, and so one of the edges $\{\beta,1\}$ or $\{\alpha,\epsilon\}$ must be contained in $\mathbb{G}$ as well. However, we already know that both of these edges must be in the necessary set.
	
	If $\mathbb{G}$ contains the edge $\{\alpha,\gamma\}$ then the vertices $\{\alpha,\gamma,\epsilon,\delta\}$ form a 4-cycle, so one of the edges $\{\alpha,\epsilon\}$ or $\{\gamma,\delta\}$ must be in $\mathbb{G}$. Again, we already know that both of these edges must be in every necessary set.
	
	At this point, no other edges can prevent $A$ from being asteroidal in an extension, so we have covered all the required cases.
\end{proof}

Now that we have computed the minimal necessary sets for each of the minimal graphs that cannot appear as induced subgraphs of interval intersection graphs, we can state a theorem for $\mathrm{SOP}$ that is an analogue of Theorem~\ref{sop2goodequivalents}.

\begin{thm}
	Let $\D$ be a regular ultrafilter on the infinite cardinal $\lambda$. The following are equivalent:
	\begin{enumerate}
		\item $\D$ is good.
		\item $\D$ saturates some theory with $\mathrm{SOP}$.
		\item If $(\mathbb{G}_\alpha)_{\alpha<\lambda}$ is a sequence of finite graphs such that $\Rm{1},\Rm{2},\Rm{3}_\ell,\Rm{4}_m,\Rm{5}_n\nleq\mathbb{G}_\alpha$ for each $\alpha$ and $\ell\geq4,m\geq2,$ and $n\geq 1$, then every complete $\mathbb{H}\se\mathbb{G}_\alpha^\D=:\mathbb{G}$ with $|\mathbb{H}|\leq\lambda$ has some internal complete subgraph of $\mathbb{G}$ containing it.
		\item Every graph-like function $f:\P_\omega(\lambda)\to\D$ that satisfies the conditions given by specializing the subset relation~\eqref{necdisteqn} to the necessary sets found for $\mathbb{H}=\Rm{1},\Rm{2},\Rm{3}_\ell,\Rm{4}_m,\Rm{5}_n$ in Lemmas~\ref{IIInecsets}--\ref{Vnecsets} has a multiplicative refinement.
	\end{enumerate}
\end{thm}
\begin{proof}
    $(1)$$\iff$$(2)$: Follows from e.g. \cite[Thm.~2.9]{sh500}.
    
    $(2)$$\iff$$(3)$: A special case of Theorem~\ref{graphcompletenessAtypes}.
    
    $(2)$$\iff$$(4)$: Follows from Theorem~\ref{realAalphatype} and Lemmas~\ref{IIInecsets}--\ref{Vnecsets}.
\end{proof}

\subsection{$\mathrm{SOP}_n$ for $n>2$}

In this section, our goal is to show that graph-like distributions corresponding to types whose realizations would fill a cut in an ultrapower of an infinite linear order (that is, types that characteristically appear in theories with $\mathrm{SOP}$) also appear as graph-like distributions for types that appear in theories with $\mathrm{SOP}_3$, regardless of whether they also have $\mathrm{SOP}$. In essence this will show that distributions in theories with $\mathrm{SOP}_n$ for $n\geq 3$ are not only complicated enough to guarantee that the theory is Keisler maximal but are complicated enough to mimic the distributions of theories where models contain infinite linear orders. This argument is a close cousin to the argument made in \cite[Thm~2.9]{sh500} that theories with $\mathrm{SOP}_3$ are Keisler maximal, but is more elementary.

\begin{thm}\label{sop3}
Suppose that $\D$ is a regular ultrafilter on the infinite cardinal $\lambda$, $\mathbb{L}=(L,<)$ is a linear order, $\varphi(x;y,z)$ is the formula $y<x<z$, and $p(\bar{x})=\{\varphi(\bar{x};a_\beta,b_\beta):\beta<\lambda\}$ is a type of $\mathbb{L}^\D$. Then, for every distribution $f$ of $p(\bar{x})$ and graph-like distribution $\hat{f}$ extending $f$ and every complete first-order theory $T$ with $\mathrm{SOP}_3$, there is a model $\mathbb{M}\vDash T$ and a graph-like type $q(\bar{y})$ of $\mathbb{M}^\D$ with a distribution $\hat{f}'$ satisfying $\hat{f}=\hat{f}'$ (up to bijection between $p(\bar{x})$ and $q(\bar{y})$).
\end{thm}
\begin{proof}
	Fix representations of $a_\beta,b_\beta$ and a distribution $f$ of $p(\bar{x})$ given these representations. Since $p$ is graph-like under any choice of representations in $L$, there is a graph-like distribution $\hat{f}$ extending $f$. Let the distribution graph sequence of $\hat{f}$ be $(\mathbb{G}_\alpha)_{\alpha<\lambda}$ where each $\mathbb{G}_\alpha$ is a finite intersection graph (see Lemma~\ref{linearordershape}). Let $I(\lambda)$ be the collection of open intervals $I$ in $\lambda$ with the property that $\min(I)$ is a successor ordinal. For each $\alpha<\lambda$ let $i_\alpha:\lambda\to I(\lambda)$ be a function such that $i_\alpha\restriction_{\mathbb{G}_\alpha}$ is a bijection and $i_\alpha(\beta)$ and $i_\alpha(\gamma)$ intersect if and only if $\beta$ and $\gamma$ are neighbors in $\mathbb{G}_\alpha$. Furthermore, we require that no two intervals in the image of $i_\alpha$ have endpoints in common. Let the functions $\ell,r:I(\lambda)\to\lambda$ be the functions defined by \[\ell(I)=\inf(I)\;\text{ and }\;r(I)=\sup(I).\]
	The function $\ell(I)$ is well-defined by the requirement that $\min(I)$ be a successor ordinal.
	
	By \cite{sh500}, $T$ has $\mathrm{SOP}_3$ if and only if there is a model $\mathbb{M}\vDash T$ containing an indiscernible sequence $(\bar{a}_i)_{i\in\omega}$ and formulas $\theta(\bar{x},\bar{y})$ and $\eta(\bar{x},\bar{y})$ such that 
	\begin{enumerate}
		\item the set $\{\theta(\bar{x},\bar{y}),\eta(\bar{x},\bar{y})\}$ is contradictory,
		\item for some sequence $(b_j)_{j\in\omega}$ we have
			\[i\leq j\Rightarrow \theta(\bar{b}_j,\bar{a}_i)\;\,\text{ and }\;\,i>j\Rightarrow \eta(\bar{b}_j,\bar{a}_i)\]
		\item for $i<j$, the set $\{\theta(\bar{x},\bar{a}_j),\eta(\bar{x},\bar{a}_i)\}$ is contradictory.
	\end{enumerate}
	Using compactness, we may extend the sequences $\bar{a}$ and $\bar{b}$ to sequences indexed by $\lambda$. Let $(\bar{a}_i)_{i\in\lambda}$, $(\bar{b}_j)_{j\in\lambda}$, $\theta(\bar{x},\bar{y})$, and $\eta(\bar{x},\bar{y})$ be as above. We then define 
	\[\psi(\bar{x};\bar{y},\bar{z})\equiv \theta(\bar{x},\bar{y})\wedge\eta(\bar{x},\bar{z}).\] Furthermore, define $(\bar{c}_\beta)_{\beta\in\lambda}$ and $(\bar{d}_\beta)_{\beta\in\lambda}$ in $\mathbb{M}^\D$ by the representations
	\[\bar{c}_\beta[\alpha]=\bar{a}_{\ell(i_{\alpha}(\beta))}\;\text{ and }\;\bar{d}_\beta[\alpha]=\bar{a}_{r(i_{\alpha}(\beta))}.\] 
	We claim that the set of formula
	\[q(\bar{x})=\{\psi(\bar{x};\bar{c}_\beta,\bar{d}_\beta):\beta<\lambda\}\] 
	is a graph-like type with distribution $\hat{f}$.
	
	In order to check that $q(\bar{x})$ is graph-like, suppose that $\Delta\in\P_\omega(\lambda)$ is such that $|\Delta|\geq 2$. We wish to show that 
	\[L_q(\Delta)=\bigcap_{\Phi\in[\Delta]^2}L_q(\Phi),\]
	where $L_q$ is the \L o\'s map for $q(\bar{x})$. In order to do this, we will instead show that 
	\begin{equation}\label{sop3L}
	    L_q(\Delta)=\Big\{\alpha\in\lambda:\bigcap i_\alpha[\Delta]\neq\emptyset\Big\},
	\end{equation}
	which will be enough because of the fact that an intersection of open intervals is non-empty if and only if the intersection of every pair of intervals in the collection is non-empty. Suppose that $j\in\bigcap i_\alpha[\Delta]$. Then we have that $\mathbb{M}\vDash \psi(\bar{b}_j;\bar{c}_\beta[\alpha],\bar{d}_\beta[\alpha])$ for each $\beta\in\Delta$ because $\ell(i_\alpha(\beta))<j<r(i_\alpha(\beta))$ by the definition of $j$. Now suppose that $\bigcap i_\alpha[\Delta]=\emptyset$. Then there is some distinct pair $\beta,\gamma\in\Delta$ such that $r(i_\alpha(\beta))\leq \ell(i_\alpha(\gamma))$. Furthermore, because the intervals in the image of $i_\alpha$ cannot repeat endpoints, this inequality is strict. In particular, this means that the formula
	\[\psi(\bar{x};\bar{c}_\beta[\alpha],\bar{d}_\beta[\alpha])\wedge\psi(\bar{x};\bar{c}_\gamma[\alpha],\bar{d}_\gamma[\alpha])\]
	cannot be satisfied in $\mathbb{M}$.
	
	We check that $\hat{f}$ is a distribution of the type $q(\bar{x})$ with the given representations. By the fact that $q(\bar{x})$ and $\hat{f}$ are graph-like, it will be enough to check that ${\hat{f}(\Delta)\se L(\Delta)}$ for all $\Delta\in\P_\omega(\lambda)$ with ${|\Delta|\leq 2}$. However, this follows immediately from Equation \eqref{sop3L} and the fact that $\hat{f}$ is a distribution for $p(\bar{x})$.
\end{proof}

\begin{cor}
	Theorem~\ref{sop3} is true whenever $T$ is replaced with a theory with $\mathrm{SOP}_n$ for $n>3$.
\end{cor}
\begin{proof}
	By \cite{sh500}, $\mathrm{SOP}_n\Rightarrow\mathrm{SOP}_3$ whenever $n\geq 3$.
\end{proof}

Given the difficulty in proving that theories with $\mathrm{SOP}_2$ are maximal in Keisler's order, we believe that the above theorem is not true for $T$ with $\mathrm{SOP}_2$. We note that if this conjecture is true, then $\mathrm{SOP}_3$ and $\mathrm{SOP}_2$ must be different on the level of theories as has been previously conjectured by Shelah.

\section*{Acknowledgements}
\noindent Much thanks goes to \'Agnes Szendrei for boundless advice given throughout the writing process and for suggesting looking at distributions related to $\mathrm{SOP}_n$ for $n>2$. This material is based upon work supported by the National Science Foundation grant no. DMS1500254.

\bibliographystyle{abbrvnat}
\bibliography{Graphlike_Distributions}

\end{document}